\documentclass[11pt,leqno]{amsart}
\usepackage[latin1]{inputenc}
\usepackage{dsfont}
\usepackage{amsxtra}
\usepackage{amsmath}
\usepackage{amscd}
\usepackage{amssymb}
\usepackage{amsfonts}
\usepackage[all]{xy}
\usepackage{mathrsfs}
\usepackage{amsthm} 
\usepackage{color}
\usepackage{stmaryrd}
\usepackage{mathabx}
\usepackage{enumitem}
\usepackage[left=2cm, right=2cm, top=2cm]{geometry}
\usepackage{geometry}

\numberwithin{itemcounter}{subsection}

\theoremstyle{plain}
\newtheorem*{thm}{Theorem}
\newtheorem*{thmA}{Theorem A}
\newtheorem*{thmB}{Theorem B}
\newtheorem*{thmC}{Theorem C}
\newtheorem*{Conj}{Conjecture}

\newtheorem{theorem}{Theorem}[section]
\newtheorem{lemma}[theorem]{Lemma}
\newtheorem{lemma-definition}[theorem]{Lemma-Definition}
\newtheorem{definition-lemma}[theorem]{Definition-Lemma}

\newtheorem{proposition}[theorem]{Proposition}

\newtheorem{corollary}[theorem]{Corollary}
\theoremstyle{definition}
\newtheorem{definition}[theorem]{Definition}
\theoremstyle{remark}
\newtheorem{remark}[theorem]{Remark}
\newtheorem{example}[theorem]{Example}
\numberwithin{equation}{section}

\def\bbA{\mathbb{A}}

\def\bbC{\mathbb{C}}

\def\bbG{\mathbb{G}}

\def\bbN{\mathbb{N}}

\def\bbQ{\mathbb{Q}}

\def\bbX{\mathbb{X}}
\def\bbY{\mathbb{Y}}
\def\bbZ{\mathbb{Z}}

\def\frakg{\mathfrak{g}}

\def\frakl{\mathfrak{l}}

\def\frakC{\mathfrak{C}}

\def\frakh{\mathfrak{H}}

\def\frakL{\mathfrak{L}}
\def\frakM{\mathfrak{M}}

\def\frakp{\mathfrak{P}}

\def\calE{\mathcal{E}}

\def\calU{\mathcal{U}}
\def\calV{\mathcal{V}}
\def\calW{\mathcal{W}}

\def\frakL{\mathfrak{L}}
\def\fraka{\mathfrak{a}}

\def\frakg{\mathfrak{g}}
\def\frakh{\mathfrak{h}}

\def\frakl{\mathfrak{l}}

\def\frakp{\mathfrak{p}}

\def\fraku{\mathfrak{u}}

\def\frakt{\mathfrak{t}}

\def\fraksl{\mathfrak{sl}}
\def\frakgl{\mathfrak{gl}}

\def\bfQ{\mathbf{Q}}

\def\bfY{\mathbf{Y}}

\def\lescc{{\trianglelefteqslant}}
\def\TT{{{T}}}

\def\mo{{{\bf g}}}

\def\corad{{{\operatorname{corad}}}}
\def\gr{{{\operatorname{gr}}}}
\def\rank{{{\operatorname{rank}}}}

\def\fq{\mathbb{F}_q}

\def\Irr{{\operatorname{Irr}\nolimits}}

\def\sp{{\operatorname{sp}\nolimits}}

\def\x{{{\operatorname{x}\nolimits}}}
\def\op{{\operatorname{op}\nolimits}}
\def\ch{{\operatorname{ch}\nolimits}}
\def\GL{\operatorname{GL}\nolimits}

\def\diag{\operatorname{diag}\nolimits}

\def\k{{\operatorname{k}\nolimits}}

\def\A{{A}}

\def\H{\operatorname{H}\nolimits}

\def\J{\operatorname{J}\nolimits}
\def\L{\operatorname{L}\nolimits}
\def\M{{\operatorname{M}\nolimits}}

\def\R{\operatorname{R}\nolimits}
\def\SS{\operatorname{S}\nolimits}

\def\Y{\operatorname{Y}\nolimits}
\def\Z{\operatorname{Z}\nolimits}

\def\rad{\operatorname{rad}\nolimits}

\def\top{\operatorname{top}\nolimits}

\def\cl{{\operatorname{cl}\nolimits}}
\def\tr{{\operatorname{tr}\nolimits}}

\def\codim{{\operatorname{codim}\nolimits}}

\def\Ker{{\operatorname{Ker}\nolimits}}

\def\Im{{\operatorname{Im}\nolimits}}

\def\Hom{\operatorname{Hom}\nolimits}

\def\End{\operatorname{End}\nolimits}

\def\Exp{\operatorname{Exp}\nolimits}

\def\Ext{\operatorname{Ext}\nolimits}

\def\Id{\operatorname{Id}\nolimits}
\def\id{\operatorname{id}\nolimits}
\def\v{v}

\def\hyp{{\operatorname{h}\nolimits}}
\def\el{{\operatorname{e}\nolimits}}

\def\re{{\operatorname{r}\nolimits}}

\def\Mon{\operatorname{Mon}\nolimits}

\def\dil{{\operatorname{dil}}}
\def\sp{{\operatorname{sp}}}

\author{O. Schiffmann, E. Vasserot}

\title{On cohomological Hall algebras of quivers : generators}

\begin{document}

\maketitle

\begin{abstract}
We study the cohomological Hall algebra $\Y^\flat$ of a lagrangian substack $\Lambda^{\flat}$ of the moduli stack of 
representations of 
the preprojective algebra of an arbitrary quiver $Q$, and their actions on the cohomology of Nakajima quiver varieties. 
We prove that $\Y^\flat$ is pure and we compute its Poincar\'e polynomials in terms of (nilpotent) Kac 
polynomials. We also provide a family of algebra generators. 
We conjecture that $\Y^{\flat}$ is equal, after a suitable extension of scalars,
to the Yangian $\bbY$ introduced by Maulik and Okounkov.
As a corollary, we prove a variant of Okounkov's conjecture, which is a generalization of the Kac conjecture relating the constant term of Kac polynomials to root multiplicities of Kac-Moody algebras.

\end{abstract}

\setcounter{tocdepth}{2}

\tableofcontents

\section{Introduction}

\medskip

In the mid 90s, H. Nakajima associated to each quiver $Q$ and each pair $(v,w)$ of dimension vectors for $Q$ 
a projective morphism
$$\pi : \frakM(v,w) \to \frakM_0(v,w)$$
where $\frakM(v,w)$ is a smooth quasi-projective symplectic variety and $\frakM_0(v,w)$ is a (usually 
singular) affine variety.
The varieties $\frakM(v,w)$ and $\frakM_0(v,w)$ have several remarkable geometric properties : the variety 
 $\frakM(v,w)$ has a hyperk\"ahler structure, 
 it possess a  resolution of the diagonal, its cohomology is generated by algebraic cycles,
  the varieties $\frakM(v,w)$ and $\frakM_0(v,w)$ carry a natural action of the reductive group $G(w) \times T$, where $T$ is a certain torus whose rank depends on $Q$ and $G(w)$ preserves the symplectic form, etc.
These varieties include the Hilbert schemes of points on $\mathbb{A}^2$ or on resolutions of Kleinian surface singularities, the moduli spaces of instantons on $\mathbb{A}^2$ and the resolutions of Slodowy slices in the nilpotent cone of $\mathfrak{gl}_n$. 

\smallskip

Quiver varieties have played an important role in geometric representation theory. For instance, when the quiver $Q$ contains no edge loops it can be regarded as an orientation of the generalized Dynkin diagram of a Kac-Moody algebra $\mathfrak{g}_Q$, and Nakajima 
constructed an action of $\mathfrak{g}_Q$ on the space 
$$L_w=\bigoplus_vH_{\top}(\frakL(v,w)),$$
where $\frakL(v,w)=\pi^{-1}(0)$ is a Lagrangian subvariety of $\frakM(v,w)$, see \cite{N98}. 
Here and below all (co)homology groups are taken with rational coefficients. The resulting module is identified with the 
integrable irreducible highest weight module of highest weight $\sum_i w_i \Lambda_i$ where the $\Lambda_i$'s are the fundamental weights of $\mathfrak{g}_Q$. In a similar vein, when the quiver $Q$ is of finite type, Nakajima constructed a 
representation of the 
quantum affine algebra of $\mathfrak{g}_Q$ on the space $$\bigoplus_{v} K^{\mathbb{C}^* \times G(w)}(\frakL(v,w)).$$ The resulting module is called a universal \textit{standard} module, and it is a geometric analog of the global Weyl modules. 
A cohomological version of this construction, 
due to Varagnolo \cite{V}, yields an action of the Yangian of $\mathfrak{g}_Q$ on the space 
$$\bigoplus_{v} H^{\mathbb{C}^* \times G(w)}_*(\frakL(v,w)),$$
and the resulting module is again the universal standard module. Nakajima's and Varagnolo's actions in 
equivariant K-theory and Borel-Moore homology extend to the case of arbitrary quivers with no edge loops, but the precise nature of the algebra which acts, or of the structure of the resulting module, are not well understood
in general.  There are a few notable exceptions : the algebras associated to the Jordan quiver are known as the \textit{elliptic Hall algebra} or \textit{quantum toroidal algebra of $\mathfrak{gl}_1$} and the \textit{affine Yangian}
of $\mathfrak{gl}_1$ respectively. They play an important role in the study of the Alday-Gaiotto-Tachikawa (AGT) correspondence in string theory, see
\cite{SV13b}, \cite{MO}, \cite{Ne15}. Similarly, the algebra associated to affine quivers are the quantum
toroidal algebras and affine Yangians, see \cite{N01}.

\smallskip

Since we have no algebraic candidate for the symmetry algebra of the equivariant K-theory and Borel-Moore homology of quiver 
varieties when the quiver is not of finite, affine or Jordan type, one may hope to describe this algebra in another way, for instance as a 
\emph{Hall algebra}. This was the \emph{raison d'\^etre} of the 
\emph{cohomological Hall algebras} in equivariant K-theory and equivariant Borel-Moore homology introduced in 
\cite{SV12}, \cite{SV13a}, \cite{SV13b} (in the particular case of quivers with only one vertex).

\smallskip

The aim of this paper is to introduce and study, for an arbitrary quiver, a certain convolution algebra $\Y$ acting on the equivariant Borel-Moore homology groups of arbitrary Nakajima quiver varieties $F_w$, where
$$F_w=\bigoplus_{v} H^{T \times G(w)}_*(\frakM(v,w)).$$
This algebra contains the algebra considered by Varagnolo. 
To be more precise, set $\Bbbk$ equal to the $T$-equivariant cohomology ring $H^*_T$ of the point.
We will introduce \textit{three} $\Bbbk$-algebras $\Y, $ $\Y^0$ and $\Y^1$, all acting on $F_w$, 
and which all become isomorphic after extension of scalars to the fraction field $K$ of $\Bbbk$.
Our algebras $\Y^\flat$ are in some sense the largest algebras which act on the homology of quiver varieties by means of some Hecke correspondences. These correspondences are required to be lagrangian if $\flat \in \{0,1\}$.

\smallskip

Before stating more precisely our results, let us introduce some notation. Let $Q=(I,\Omega)$ be an arbitrary 
quiver. We will call \textit{real} the vertices of $I$ without $1$-cycles and \textit{imaginary} the other vertices. 
Let $v$ be a dimension vector and let $Rep (\mathbb{C}Q,v)/ G(v)$ be the moduli stack of complex representations 
of $Q$ of dimension $v$. Let $\Pi$ be the preprojective 
algebra of $Q$ and let $Rep(\Pi, v)/G(v)$ be the moduli stack of complex $\Pi$-representations of dimension $v$, 
which is isomorphic to the cotangent stack of $Rep (\mathbb{C}Q,v)/ G(v)$. 

\bigskip

\emph{Geometry of Lagrangian substacks.}
In \cite{BSV} we defined Lagrangian substacks $\Lambda^0(v)/G(v)$ and $\Lambda^1(v)/G(v)$ of the quotient stack 
$Rep(\Pi, v)/G(v)$, by using some semi-nilpotency condition. See also \cite{B14}. If $Q$ has no $1$-cycle and no oriented cycle, 
then $\Lambda^0(v)$ and $\Lambda^1(v)$ both coincide with Lusztig's nilpotent variety $\Lambda^L(v)$. We always have 
$$\Lambda^L(v) \subseteq \Lambda^1(v) \subseteq \Lambda^0(v),$$ but the inclusions are strict in general. The torus $T$ acts on 
$Rep(\Pi,v)/G(v)$ and $\Lambda^\flat(v)/G(v)$ for $\flat \in \{0,1\}$. We prove the following in 
Theorem~\ref{thm:1.1}, Proposition \ref{prop:torsionfree}, and Theorem~\ref{thm:dim}
(under some mild condition on the torus $T$).

\smallskip

\begin{thmA} For any quiver $Q=(I,\Omega)$, any $v \in \mathbb{N}^I$ and any $\flat \in \{0,1\}$, we have
\begin{enumerate}
\item[$\mathrm{(a)}$]  $H^{T \times G(v)}_*(\Lambda^\flat(v))$ is pure and even.
\item[$\mathrm{(b)}$]  $H^{T \times G(v)}_*(\Lambda^\flat(v))$ is free as a $\Bbbk$-module.
\item[$\mathrm{(c)}$] $H^{T \times G(v)}_*(\Lambda^{\flat}(v))$ is torsion-free as a $H^*_{T \times G(v)}$-module.
\item[$\mathrm{(d)}$] The Poincar\'e polynomial $P(\Lambda^\flat(v),q)=\sum_{i}  \dim_{\mathbb{Q}}(H^{T \times G(v)}_{2i}(\Lambda^{\flat}(v)))\,q^{i}$ is given by
$$\sum_v P(\Lambda^\flat(v),q)\,q^{\langle v, v \rangle + \tau}z^v= (1-q^{-1})^{-\tau}\Exp\Big(\sum_v \frac{A^{\flat}(q^{-1})}{1-q^{-1}}z^v\Big)$$
where $\tau$ is the rank of $T$, $A^{\flat}(t)$ is the $\flat$-nilpotent Kac polynomial and $\Exp$ is the plethystic exponential.
\end{enumerate}
\end{thmA}

\smallskip

Statement (d) of Theorem~A is a consequence of (a) and of the computation, conducted in \cite{BSV}, of the number of rational points of $\Lambda^\flat(v)$ over finite fields

%
%
%
%

\bigskip

\emph{Cohomological Hall algebras.}
The construction of $\Y^\flat$ is given in terms of the cohomological Hall algebras of the above stacks. More precisely, 
if $\flat\in\{0,1\}$ we have
$$\Y^{\flat}=\bigoplus_v H_*^{T \times G(v)}(\Lambda^\flat(v)), \qquad \Y=\bigoplus_v H_*^{T \times G(v)}(Rep(\Pi,v)).$$ 
For any dimension vector $v$, we set 
$$\Lambda_{(v)} =\{0\} \times Rep(\mathbb{C}Q^*,v) \subset \Lambda^0(v).$$ 
This is always an irreducible component of $\Lambda^0(v)$.
It is also an irreducible component of $\Lambda^1(v)$ in case $v$ is supported on a subquiver without oriented cycles. 
From a geometric perspective, the quotient stack
$\Lambda_{(v)}/G(v)$ is the zero section the cotangent bundle map 
$$Rep(\Pi,v)/G(v) \to Rep(\mathbb{C}Q,v)/G(v).$$ 
Note that we have $H^{T \times G(v)}_*(\Lambda_{(v)})\simeq H^*_{T \times G(v)}$. 
Motivated by the analogy with Yangians, for each $\flat\in\{\emptyset,0,1\}$,
we consider a slight extension $\bfY^\flat$ of $\Y^\flat$ by adding a \emph{loop Cartan} part equal to 
$$\bfY(0)=H^*_{T \times G(\infty)}.$$
Then, we write
$$\Y_K^\flat=\Y^\flat\otimes_{\Bbbk}K,\quad \bfY^\flat_K=\bfY^\flat \otimes_{\Bbbk} K.$$
Finally, we say that an imaginary vertex is \textit{elliptic}
or \textit{isotropic} if it carries a single $1$-loop and \textit{hyperbolic} if it carries more than one $1$-loop. We denote by 
$I^\re,$ $ I^\el$ and $ I^\hyp$ respectively the real, elliptic and hyperbolic vertices of $I$.
Let  $\{\delta_i\,;\,i \in I\}\subset\mathbb{N}^I$ be the basis of delta functions.
In Propositions~\ref{prop:1.15}, \ref{prop:ratform}, \ref{prop:2.2} and Theorem~\ref{thm:gen} we prove the following.

\smallskip

\begin{thmB}
For any $\# \in \{\emptyset, 0,1\}$ and any $\flat \in \{0,1\}$ we have the following.
\begin{enumerate}
\item[$\mathrm{(a)}$] There is an associative $\mathbb{Z} \times \mathbb{N}^I$-graded $\Bbbk$-algebra structure on $\Y^{\#}$.
\item[$\mathrm{(b)}$] There is a representation of $\Y^\#$ on $F_w$ for each $w$.
\item[$\mathrm{(c)}$] 
The diagonal action of $\Y^\flat$ on $\bigoplus_w F_w$ is faithful.
\item[$\mathrm{(d)}$] There are $K$-algebra isomorphisms
$\Y^1_K  \simeq  \Y^0_K \simeq \Y_K.$
\item[$\mathrm{(e)}$] The $K$-algebra $\Y_K$ is generated by the subspaces 
$H^{T \times G(v)}_*(\Lambda_{(v)})\otimes_\Bbbk K$ where $v$ runs over the set
$\{\delta_i,l\delta_j\;;\; i \in I,j \in I^\hyp\,,\, l>1\}.$
The $K$-algebra $\bfY_K$ is generated by $\bfY(0)$ and the collection of fundamental classes
$\{[\Lambda_{(\delta_i)}],[\Lambda_{(l\delta_j)}]\;;\; i \in I,j \in I^\hyp, l>1\}.$
\end{enumerate}
\end{thmB}

\smallskip

There is also a more precise result providing a family of generators for $\Y^1$, see Proposition \ref{prop:2.6}. 
In particular, when $Q$ has no 
vertex carrying a $1$-cycle we have $\Y^1=\Y^0$ and these coincide with the $\Bbbk$-algebra constructed by Varagnolo, see 
Remark~\ref{rem:q=0}. In general,  the $\Bbbk$-algebras $\Y,$ $ \Y^0$ and $\Y^1$ are different integral forms inside 
the $K$-algebra
$\Y_K$. This choice of integral form is responsible for the different types of Kac polynomials entering the Poincar\'e polynomial formulas in Theorem~A(c).
The proof of statements (a) and (b) follows the approach in \cite{SV13b}. 
See also \cite{YZ14}. Part (d) is a consequence of the localization theorem. 
The proof of (e) is more involved. First, we use the algebra $\Y^1$ 
and the geometry of $\Lambda^1(v)$, in particular, the crystal structure studied in \cite{B14}, to reduce ourselves to a one vertex 
quiver. Then, we deal separately with the real, elliptic and hyperbolic cases. A crucial step in the argument is the following 
version of the Kirwan surjectivity for \textit{local quiver varieties}. 
Let $\frakM(u,y)$ be a quiver variety, and let 
$$\frakM(v,w)^\# \subset \frakM(u,y)^{\mathbb{C}^*}$$ be a local quiver subvariety of $\frakM(u,y)$, that is a fixed point quiver variety associated to the $\mathbb{C}^*$-action as in \eqref{action1}. 

\smallskip

\begin{thm}The cohomology ring $H^*_{T\times G(w)}(\frakM(v,w)^\#)$ is generated by the Chern classes of the universal and tautological bundles.
\end{thm}

\smallskip
 
T. Nevins and K. McGerty recently gave a proof of Kirwan surjectivity for the quiver varieties $\frakM(v,w)$ 
themselves, see \cite{MNKirwan}. The above theorem does not seem to directly follow from it. Our proof bears 
some similarity to but is independent from theirs, see Theorem~\ref{thm:Kirwangraded} and Appendix A.4. In 
fact, the same proof works for the fixed point quiver varieties associated to any cocharacter which scales the 
symplectic form by a nontrivial factor. As a corollary, we obtain (see Corollary~\ref{cor:Mrhoconn} for notations)

\smallskip

\begin{corollary} Let $\sigma : \mathbb{C}^* \to T \times G(w)$ be a cocharacter acting nontrivially on the 
symplectic form of $\frakM(v,w)$. Then the fixed point quiver varieties $\frakM[\rho] \subset \frakM(v,w)$ 
associated to the $\mathbb{C}^*$-action induced by $\sigma$ are connected.
\end{corollary}

This includes several cases of interest, such as the \textit{handsaw} and 
\emph{graded} quiver varieties.

\bigskip

\emph{Comparison with Kontsevich-Soibelman COHAs.}
We introduced K-theoretic Hall algebras in \cite{SV13a}, \cite{SV12} to study the moduli stack 
of preprojective representations 
of a one vertex quiver. Then, we studied its cohomological version in \cite{SV13b}.
Independently, Kontsevich and Soibelman introduced some COHA associated to various Calabi-Yau categories
in \cite{KS10}. It has been recently proved by Davison in
\cite{Davison} that the COHA of the moduli stack of preprojective representations of quivers that we used is isomorphic to a particular case of the
COHA of Kontsevich-Soibelman, viewing the moduli stack of preprojective representations of the quiver as 
a particular 2 dimensional Calabi-Yau category.

\bigskip

\emph{Comparison with Maulik-Okounkov Yangians.}
In \cite{MO}, the authors defined and studied by some totally different means another associative algebra acting 
on the (co)homology of Nakajima quiver varieties associated to an arbitrary quiver $Q=(I,\Omega)$. Their 
construction, which stems from ideas in symplectic 
geometry, hinges on the notion of \textit{stable enveloppe} to produce a quantum $R$-matrix, and then on the 
RTT formalism to define an associative $\mathbb{Z}\times \mathbb{Z}^I$-graded $\Bbbk$-algebra 
$\mathbb{Y}=\mathbb{Y}_Q$ acting on the space $F_w$ for 
any dimension vector $w$. Taking a quasi-classical limit, they also define a classical $R$-matrix and a $\mathbb{N} \times \mathbb{Z}^I$-graded Lie algebra $\mo=\mo_Q$. If $Q$ is of finite type, then 
$\mo_Q$ is the semisimple Lie algebra $\mathfrak{g}_Q$ associated with $Q$, and $\mathbb{Y}_{Q}$ is the Yangian of the same type. 
In general, the $\Bbbk$-algebra $\mathbb{Y}_Q$ is a deformation of the enveloping algebra of the Lie algebra of polynomial loops 
$\mo[u]$.
Their construction provides triangular decompositions for these algebras
 $$\mathbb{Y}=\mathbb{Y}_+ \otimes \mathbb{Y}_0 \otimes \mathbb{Y}_-, \qquad \mo=\mo_{+} \oplus \mo_{0} \oplus \mo_{-}.$$
In \cite{SV17b}, we compare our cohomological algebra $\Y$ with the positive half $\mathbb{Y}_+$. 
Set $\bbY_{+,K}=\mathbb{Y}_{+}\otimes_{\Bbbk} K$. We make the following conjecture.

\smallskip

\begin{Conj} There is a unique $K$-algebra isomorphism $\mathbb{Y}_{+,K} \simeq \Y_K$ which intertwines the
respective actions of $\mathbb{Y}_{+,K}$ and $\Y_K$ on $F_w \otimes_{\Bbbk} K$ for any $w$.
\end{Conj}

\smallskip

In \cite{SV17b} we prove one half of the above conjecture. 

\smallskip

\begin{thmC} There is a unique embedding $ \Y_K \subset \mathbb{Y}_{+,K}$ which intertwines the
actions of $\mathbb{Y}_{+,K}$ and $\Y_K$ on $F_w \otimes_{\Bbbk} K$ for any $w$.
\end{thmC}

\smallskip

By Theorem~B, the proof of the above theorem boils down to checking that certain \textit{generalized Hecke correspondences} corresponding to generators of $\Y^1$ occur with a  non zero and constant coefficient in a suitable stable enveloppe. Again, generalized Hecke correspondences associated to real, elliptic or hyperbolic vertices behave in very different ways and we have to treat each case separately.

\bigskip

\emph{Okounkov's conjecture.}
Let us finish this introduction by mentioning one other motivation for this work. In \cite{O13}, A. Okounkov conjectured that
the graded dimensions of the root spaces $\mo[\alpha]\subset\mo$ are, after a suitable grading shift, precisely given by the Kac polynomial $A_{\alpha}(t)$. This is a generalization of the Kac conjecture, proved in
\cite{Ha10}, stating that for quivers with no $1$-cycles the multiplicity of the root $\alpha$ in the Kac-Moody algebra 
$\mathfrak{g}_Q$ is equal to the constant term $A_{\alpha}(0)$. Indeed, one expects to have an isomorphism 
$\mathfrak{g}_Q \simeq\mo_Q[0]$, where the grading in $\mo_Q$ is counted from middle dimension up.
In other words, the Lie algebra $\mo_Q$ is a graded extension of $\mathfrak{g}_Q$, whose character is conjecturally given by the \textit{full} Kac polynomials rather than their constant terms. Our Theorem~A(d) proves a variant of the above conjecture~: the graded character of the cohomological Hall algebras $\Y^1$ and $ \Y^0$ are encoded by the full \textit{nilpotent} Kac polynomials 
$A^0_{\alpha}(t)$ and $ A^1_{\alpha}(t)$. Note that, contrary to that of Maulik and Okounkov, our construction does not directly provide a construction of an underlying Lie algebra. See some recent work of Davison and Meinhardt in that direction in \cite{DM}. 
Moreover, we expect the discrepancy between the various types of Kac polynomials involved to correspond to different gradings on the \textit{same} Lie algebra. Indeed, note that we have $A_{\alpha}(1)=A^1_{\alpha}(1)=A^0_{\alpha}(1)$, see \cite{BSV}.

\bigskip

\emph{Plan of the paper.} \hfill{}

\smallskip

- \S 2 is essentially a reminder of some standard properties of equivariant Borel-Moore homology and equivariant Chow groups. We use both cohomological theories simultaneously throughout the paper. We also provide a proof of
the folklore fact that the counting polynomial and equivariant Poincar\'e polynomials of a pure $G$-variety $X$ are equal if $X$ is of polynomial count (and $G$ is a product of general linear groups). 

\smallskip

- In \S 3 we introduce the varieties $\M(v)$ of representations of the preprojective algebra $\Pi$ of an arbitrary quiver $Q=(I,\Omega)$ as well as its Lagrangian subvarieties $\Lambda^\flat(v)$ for $\flat \in \{0,1\}$. In \S 3.5 we describe the stratification of $\Lambda^1(v)$ responsible for the crystal structure, see \cite{B14}. This is used in our proof of Theorem B in order
to reduce ourselves to one vertex quivers. The geometry of $\Lambda^0(v)$ for one vertex quivers is the subject of \S 3.6 where the
real, elliptic and hyperbolic cases are dealt with in details. After recalling the definition of quiver varieties, we study in \S 3.8
some generalized Hecke correspondences. We prove that they are Lagrangian local complete intersections in Proposition~\ref{prop:LAG}, and that they are irreducible in the hyperbolic one-vertex case in Proposition~\ref{prop:irrheckenu}.

\smallskip

- In \S 4 we prove Theorem~A. The purity statement is proved using a compactification of $\Lambda^\flat(v)$ constructed out
of Lagrangian quiver varieties. Statements (b) and (d) follow, using the point count computations in \cite{BSV}. We also consider the 
structure of $H_*^{T \times G(v)}(\Lambda^{\flat}(v))$ as a $H^*_{T \times G(v)}$-module and show that it is torsion free using
some recent work of B. Davison on dimensional reduction. Finally, we briefly state the Kirwan surjectivity result for fixed point quiver varieties.

\smallskip

- The algebras $\Y^{\flat}$ are defined in \S 5, which also contains the proof of statements (a), (c) and (d) of Theorem~B. The most delicate part is (d). We first use the crystal stratification of $\Lambda^1(v)$ to reduce ourselves to one-loop quivers in Proposition 
\ref{prop:2.2}, then we use the compactification via quiver varieties and finally conclude using the Kirwan surjectivity for graded quiver 
varieties in Theorem \ref{thm:Kirwangraded} and  Proposition \ref{prop:2.6}. We also define the extension $\bfY^{\flat}$ of $\Y^\flat$, by adding a 
\emph{Cartan part}, and construct an action of the resulting algebras on the spaces $F_w$ for any $w$, thus completing the proof of Theorem~B.

\smallskip

- We have postponed to the appendix the proof of the Kirwan surjectivity result for local quiver varieties as well as
some folklore homological and geometric results concerning preprojective algebras of arbitrary quivers.

\bigskip

\medskip

\centerline{\bf{\large{Acknowledgements}}}

\smallskip

We would like to thank T. Bozec, F. Charles,  M. McBreen, B. ~Davison, H. Nakajima, A. Negut and A.~Okounkov for useful discussions and correspondences. Special thanks are due to B.~Crawley-Boevey for providing us the proof of Proposition~\ref{P:CYpreproj} and B.~Davison for pointing out an omission in the hypothesis of Proposition~4.6.

\newpage

\section{Equivariant homology}

In this section we set $\k=\bbC$. 

\subsection{Definition}\label{sec:BM}\hfill\\

By a \emph{variety} we mean a reduced scheme over an algebraically closed field. All schemes considered here will be reduced.

\smallskip

Given an algebraic variety $X/\k$, let
$H_*(X)$ be the Borel-Moore homology group (= the locally finite singular homology group) of $X$
and let $\A_*(X)$ the Chow group of $X$, both with rational coefficients.
In particular, we have $H_i(X)=H_c^i(X)^\vee$ where $H^*_c(X)$ is the cohomology group with compact support. 
By \cite{F98}, the additive functor $A_*$ is an 
\emph{oriented Borel-Moore homology} in the sense of \cite[\S 5]{LM07}.
This means in particular that there are pullback and refined pullback maps for any
local complete intersection morphisms, pushforward maps
for proper morphisms, external products and Chern classes 
which satisfy all the properties mentioned there.
See below for more details. The functor $H_*$ is also oriented Borel-Moore homology.
In particular, locally finite singular homology groups admit refined pullback maps. 
Note that Fulton's definition of the refined pullback maps
uses only the standard operations (proper pushforward, smooth pullback, external products and specializations)
whose construction can be found in \cite{DV76} in the setting of $H_*$. 
See \cite[\S 6]{LM07} for more details.
Let $A^*$ and $H^*$ be the corresponding oriented cohomology functors.
Thus $A^*$ is the Chow cohomology group, or operational Chow group,
introduced by Fulton and MacPherson.
We write $M_*$ for either $H_{*}$ or $\A_*$ and $M^*$ for either $H^{*}$ or $\A^*$.
Set $\varepsilon=2$ if $M_{*}=H_*$ and $\varepsilon=1$ if $M_{*}=\A_*$.

\smallskip

Now, assume that $X$ is equipped with the action of a linear algebraic group $G/\k$.
We always assume that $X$ is quasi-projective and that we have fixed a very ample $G$-linearized
line bundle over $X$.
Thus, the variety $X$ embeds equivariantly in a projective space with a linear $G$-action. 

\smallskip
  
Let $H_*^G(X)$ and $\A_*^G(X)$ be the equivariant Borel-Moore homology and the equivariant Chow group of $X$.
Standard references for equivariant Borel-Moore homology and equivariant Chow groups are \cite{B0}, \cite{Gr} and
\cite[\S 2.2, \S 2.8]{EG98}.
We write $M_*^G$ to denote either $H_{*}^G$ or $\A_*^G$.
Let $ M^*_G$ be the corresponding equivariant cohomology functor.
More precisely $H^*_G(X)$ is the equivariant cohomology group while
$\A^*_G(X)$ is the \emph{equivariant operational Chow group} defined in \cite{EG98}.
Let $g$ be the dimension of $G$.
If $X$ is $n$-dimensional then for each integer $j\geqslant i$
we have  
\begin{align}\label{A}
 M^G_j(X)=M_{j+\varepsilon(l_i-g)}(X\times_G U_i)
\end{align}
where $U_i$ is an open set with a free 
$G$-action in a finite dimensional representation of $G$ such that the complement has
codimension $\geqslant n-i$ and $l_i$ is the dimension of $U_i$.
We call the quotient $X\times_G U_{i}$ an \emph{approximation} to the Borel construction $X_G:=X\times_GEG$ of $X$.
In general, it  is only representable by an algebraic space. 
However, since we only consider the case where $G$ is a product of linear groups, we may choose the approximation such that
$X\times_G U_{i}$ and $U_{i}/G$ are both $\k$-varieties by \cite[lem.~7, prop.~23]{EG98}.


\smallskip


\smallskip

\subsection{Equivariant homology and homology of Artin stacks}\hfill\\

The \emph{quotient stack} of a $G$-variety $X$ by $G$ is the Artin stack associated with
the groupoid $G\times X\to X\times X$.  We denote it by the symbol $X/G$. We use
the symbol $[X/G]$ for the fundamental class of $X/G$ in $M_*(X/G)$, and the symbol $X/\!\!/G$ for the categorical quotient.

\smallskip

A morphism of quotient stacks $f:X/G\to Y/H$ is \emph{safe} if it is given by a compatible pair $(f^\flat,f^\sharp)$ 
consisting of a scheme homomorphism
$f^\flat:X\to Y$ and a surjective algebraic group homomorphism $f^\sharp:G\to H$ with unipotent kernel.
The relative dimension of $f$ is the relative dimension of $f^\flat$ minus the dimension of the
kernel of $f^\sharp$. We only consider safe morphisms.

\smallskip

Assume that there is a stack isomorphism $X/G\simeq Y/H$. 
Let $g,h$ be the dimension of $G,H$. By \cite[prop.~16]{EG98} we have an isomorphism 
$M_{*+\varepsilon g}^{G}(X)\simeq M_{*+\varepsilon h}^{H}(Y).$
As a consequence, we can define the Borel-Moore homology or the Chow group of $X/G$ by 
\begin{align}\label{B}M_{*}(X/G)=M_{*+\varepsilon g}^{G}(X).
\end{align}

\smallskip

Any $G$-invariant closed subset $Y$ of $X$ admits an equivariant fundamental class $[Y]$ in $M_*^G(X)$.
For any $G$-invariant closed subset $Y$ of $X$, the fundamental class of $Y/G$ 
is the element $[Y/G]$ in $M_{*}(X/G)$
which coincides under \eqref{B} with the equivariant fundamental class of $Y$ in $M_*^G(X)$.
If $Y$ is equidimensional then we have $$\deg([Y/G])=\varepsilon\dim(Y/G)=\varepsilon\dim Y-\varepsilon g.$$ 
If $X$ has dimension $d$ then 
$M_i(X/G)=0$ if $i>\varepsilon(d-g),$ while 
$M_{\varepsilon(d-g)}(X/G)$ is spanned by the equivariant fundamental classes of the irreducible components of $X$.

\smallskip

Finally, if $X$ is an $H\times G$-variety such that the quotient map $f:X\to X/\!\!/ G$ is an $H$-equivariant principal $G$-bundle, 
then we have a \emph{descent} $M_H^*$-linear isomorphism
$$f^\heartsuit:
M_*^{H}(X/\!\!/G)\to M_{*+\varepsilon\,g}^{H\times G}(X).$$

\smallskip

\subsection{Properties}\hfill\\

The functors $M_*^G$, $M_G^*$ satisfy the same properties as $M_*$, $M^*$, for which we refer to  
\cite{CG}, \cite{DV76}, \cite{EH}, \cite{LM07}.
Let us recall briefly a few facts. 

\subsubsection{Poincar\'e duality} 
The cup product equips $M_*^G(X)$ with a  graded ring structure, 
and there is a cap product $\cap$ making $M_*^G(X)$ into an $M^*_G(X)$-module.
We abbreviate $$M_G^*= M^*_{G}(\text{Spec}(\k)).$$ 
If $X$ is pure dimensional there is a map
$$M_G^*(X)\to M^G_{\varepsilon\dim X-*}(X),\ \alpha\mapsto \alpha\cap[X].$$
If $X$ is smooth the intersection product $\cap$
on $M_*^G(X)$ is induced by the product $\cup$ on $M^*_G(X)$ via this map, which is invertible.

\subsubsection{Excision}

If $Y\subseteq X$ is a $G$-invariant closed subset, there is the localization exact sequence 
$\A_*^G(Y)\to \A_*^G(X)\to \A_*^G(X\setminus Y)\to 0.$

\smallskip

\subsubsection{Mayer-Vietoris}
If $Y_1,$ $Y_2$ are $G$-invariant closed subsets of $X$, the Mayer-Vietoris exact sequence in equivariant Chow groups is
$$\A_*^G(Y_1\cap Y_2)\to\A_*^G(Y_1)\oplus \A^G_*(Y_2)\to\A_*^G(Y_1\cup Y_2)\to 0.$$


\subsubsection{Pullback}

The functor $M_*^G$ has the same functoriality as $M_*$. See, e.g., \cite[\S 2.3]{EG98} for Chow groups.
Recall that a closed immersion of $G$-varieties $i : Z \to X$ is called a \emph{regular embedding} if
the ideal sheaf of $Z$ in $X$ is locally generated by a regular sequence. A
\emph{local complete intersection morphism}, an \emph{l.c.i.~morphism} for short,
is a morphism $f : Z \to X$ which admits a
factorization as $f = q \circ i$, where $i$ is a regular embedding and
$q$ is a smooth, quasi-projective morphism.
For any l.c.i.~morphism $f$ of 
relative dimension $d$ there is a pullback morphism which is a graded vector space homomorphism
$$f^* : M_*^G(X ) \to M^G_{*+\varepsilon d}(Z).$$

\smallskip

An \emph{affine space bundle} of rank $d$ over $X$ is a map $f:Z\to X$ such that $X$ can be covered by open affine subsets 
$U_i$ with an isomorphism
$f^{-1}(U_i)\simeq\bbA^d\times U_i$ under which the map $f$ corresponds to the projection on the second factor.
If the map $f$ is a $G$-equivariant affine space bundle then the pullback $f^*$ is surjective.
If $f$ admits a $G$-equivariant section  then $f^*$ is an isomorphism.

\smallskip

\subsubsection{Refined pullback}

Consider the following cartesian square of $G$-varieties
\begin{align}\label{DDIAG}
\begin{split}
\xymatrix{W\ar[r]^{g'}\ar[d]_{f'}&Z\ar[d]_f\\Y\ar[r]^{g}&X.}
\end{split}
\end{align}
If $f$ is an l.c.i.~morphism of
relative dimension $d$, we have a refined pullback homomorphism
$$(f,f')^! : M_*^G(Y ) \to  M^G_{*+\varepsilon d}(W).$$
If no confusion is possible we abbreviate $f^!=(f,f')^!$. 
We use the following properties of the refined pullback~:
the \emph{multiplicativity} 
(i.e., $f^!(x\cap y)=(f')^*(x)\cap f^!(y)$ for all $x\in M^*_G(Y),$ $y\in M_*^G(Y)$), 
the \emph{proper base change property}, 
the \emph{functoriality with respect to smooth morphisms} and the \emph{excess intersection formula}
as in \cite[thm.~6.6.6, thm.~6.6.9]{LM07} respectively.
It follows from the excess intersection formula that $f^!=(f')^*$ whenever
the map $f'$ is an l.c.i.~morphism of the same relative dimension as $f$.

\smallskip

\subsubsection{Transversality}
Irreducible subvarieties $Y,$ $Z$ of a pure dimensional
variety $X$ are \emph{dimensionally transverse} if for every irreducible component $W$ of $Y\cap Z$
we have $\codim W=\codim Y+\codim Z$. 
They are \emph{generically transverse} if every irreducible component of $Y\cap Z$ contains a point $x$ at which 
$X,Y,Z$ are smooth and the tangent spaces satisfy $T_xY+T_xZ=T_xX$. 
Note that $Y$ and $Z$ are generically transverse if and only if they are dimensionally transverse and 
each irreducible component of $Y\cap Z$ is reduced and contains a smooth point of $X$.

\smallskip

Now, let $X$ be a smooth $G$-variety and $Y$, $Z$ be $G$-equivariant subvarieties.
If $Y,$ $Z$ are generically transverse then the intersection product on
$M_*^G(X)$ is such that $[Y]\cap [Z]=[Y\cap Z]$. More generally, 
if $Y,$ $Z$ are dimensionally transverse and l.c.i.~at the generic point of every component of $Y\cap Z$ then we still have 
$[Y]\cap [Z]=[Y\cap Z]$.

\smallskip

\subsubsection{The cycle map}
There is a cycle map $\cl:\A_*^G(X)\to H^G_{2*}(X)$ defined in
\cite[\S 2.8]{EG98}.
It is a degree doubling homomorphism which is natural with respect to pullback by l.c.i.~morphisms or pushforward by proper maps.

\smallskip

\subsubsection{Chern classes}
Equivariant vector bundles on $X$ have equivariant Chern classes in $M^*_G(X)$. 
For any $G$-equivariant vector bundle $\calE$ let $\ch(\calE)=\sum_{l\geqslant 0}\ch_l(\calE)$
be its \emph{equivariant Chern character,} with $\ch_l(\calE)\in M^{\varepsilon l}_G(X)$ for each $l$.

\smallskip

\subsection{Mosaic}\hfill\\

Let us finish this section with a few reminders.

\smallskip

First, given a variety $X$ with an action of a linear algebraic group $G$, a \emph{$G$-equivariant cycle} on $X$ is a 
$\bbQ$-linear combination of $G$-invariant closed subvarieties of $X$. It can be viewed as an element in $M_*^G(X)$.
If $X$ is a symplectic manifold then a cycle on $X$ is Lagrangian if each component is a Lagrangian subvariety.

\smallskip

Next, for any closed subset $Y\subseteq X$ of a $G$-variety $X$ we say that a class in
$M_*^G(X)$ is \emph{supported on $Y$}
if it comes from an element in $M_*^G(Y)$, or, equivalently, if its restriction to $X\setminus Y$ is zero.

\smallskip

Finally, a commutative diagram is called a \emph{fiber diagram} if it consists of squares which are cartesian.

\smallskip

\subsection{Purity and polynomial count}\label{sec:purity} \hfill\\

Recall that $X$ is an algebraic variety $X/\k$ and that $\k=\bbC$.
Using Deligne's construction of mixed Hodge structure on relative cohomology, one can define a well-behaved MHS on  $H_c^*(X)$.
Since $H_i(X)$ is the dual of $H^i_c(X)$ for each degree $i$, 
it is equipped with a MHS as well, which is functorial for open immersions and proper maps. 
See, e.g., \cite[\S 6.3.1]{PS08}. 
We say that $H_*(X)$ is even if $H_i(X)=0$ whenever $i$ is odd, and that $H_*(X)$ is pure if the
MHS is pure.

\smallskip

We say that $X$ has \emph{polynomial count} if there is a polynomial $P(t)\in\bbQ[t]$
and an $R$-form $X_R$ of $X$ over a finitely generated subring $R\subseteq\bbC$ such that
for any morphism $R\to\fq$ to a finite field, the number of $\fq$-points of the scheme
$X_{\fq}:=X_R\otimes_R\fq$ is $P(q)$. We abbreviate $X(\fq)$ for $X_{\fq}(\fq)$.
Note that we have
$$|(X/G)(\fq)|=|X(\fq)|\,/\,|G(\fq)|.$$

\smallskip

For $G=GL(n)$ we may choose $U_i$
to be the Stiefel variety $\Mon(\k^n,\k^w)$ consisting of injective linear maps $\k^n\to\k^w$ with $w$ large enough.
In this case the variety $U_i/G$ is a Grassmanian, hence it has polynomial count.
Since we only consider the case where $G$ is a product of linear groups, from now on 
\emph{we  always assume that $U_i/G$ has polynomial count for all integers $i$}.

\smallskip

Since Deligne defined a MHS for any simplicial variety, his construction holds as well for 
the Borel construction $X_G$. This yields a MHS on $H^*_G(X)$. 
Using the approximation to $X_G$ we also define a MHS on $H_*^G(X)$ such that for each $i$
the MHS on $H_j^G(X)$ is the one on $H_{j+2(l_i-g)}(X\times_GU_i)$ for all $j\geqslant i$.
Using the functoriality of MHS, see, e.g., \cite[\S 5.5.1]{PS08}, one checks as in \cite[def-prop.~1]{EG98} that
the MHS on $H_j^G(X)$ does not depend on the choice of the integer $i$.

\smallskip

The \emph{$E$-polynomial} of $H_*^G(X)$ is indeed a power series. It is given by
\begin{align}\label{E}
E(H_*^G(X);x,y)=\sum_{p,q,i}(-1)^i\dim_\bbC(\gr_F^p\,\gr^W_{p+q} (H_i^G(X)\otimes_\bbQ\bbC))\,x^py^q,
\end{align}
where the (increasing) filtration $(W_p)$ is the weight filtration on $H_i^G(X)$ and
the (decreasing) filtration $(F^p)$ is the Hodge filtration on $H_i^G(X)\otimes\bbC$.
We have the following \cite[thm.~2.1.8]{HRV08}.

\begin{proposition}[Katz]\label{prop:Katz} 
Assume that $X$ has polynomial count with count polynomial $P(t)\in\bbZ[t]$. 
Then, the $E$-polynomial of $H_*(X)$ is given by $E(H_*(X);x,y)=P(xy)$.
\qed
\end{proposition}

\smallskip

We deduce the following.

\smallskip

\begin{corollary}\label{cor:poincare} Assume that $H_*(X/G)$ is pure and $X$ has polynomial count. 
Then, the grading of $H_*(X/G)$ is even and its Poincar\'e polynomial is given by
\begin{align}\label{PC}
\sum_i(-1)^i\dim_\bbQ(H_i(X/G))\,q^{i/2}=\sum_i\dim_\bbQ(H_{2i}(X/G))\,q^{i}=|(X/G)(\fq)|
\end{align}
for all finite fields $\fq$.
\end{corollary}

\begin{proof} Let $P_X(t)$ be the count polynomial of $X$.
For each integer $i$ the approximation $X\times_GU_i$ has polynomial count
with count polynomial $P_{U_i/G}(t)\,P_X(t)$, the product of the count polynomials of $U_i/G$ and $X$.
From Proposition \ref{prop:Katz} we deduce that
$$E(H_*(X\times_GU_i);x,y)=P_{U_i/G}(xy)\,P_X(xy).$$
Taking the limit $i\to-\infty,$ from \eqref{A}, \eqref{B} we get that the $E$-polynomial of $H_*(X/G)$ is
$$E(H_*(X/G);x,y)=P_X(xy)/P_G(xy)=E(H_*(X);x,y)/E(H_*(G);x,y).$$
Since $H_*(X/G)$ is pure, we deduce from \eqref{E} that it is also even, hence its Poincar\'e polynomial is 
$$P_X(t^2)/P_G(t^2)=E(H_*(X/G);t,t)=\sum_i\dim_\bbQ(H_i(X/G))\,t^{i}=\sum_i(-1)^i\dim_\bbQ(H_i(X/G))\,t^{i}.$$
Therefore, we have
\begin{align*}
\sum_i(-1)^i\dim_\bbQ(H_i(X/G))\,q^{i/2}
&=|(X/G)(\fq)|.
\end{align*}
\end{proof}

\medskip

\section{The semi-nilpotent variety}\hfill

Let $\k$ be any algebraically closed field.

\subsection{Quivers} \label{sec:quiver}\hfill\\

Let $Q=(I,\Omega)$ be a finite quiver with set of vertices $I$ and set of arrows $\Omega$.
Let $Q^*=(I,\Omega^*)$ be the opposite quiver, with the set of arrows $\Omega^*=\{h^*\,;\,h\in \Omega\}$
where $h^*$ is the arrows obtained by reversing the orientation of $h$.
The double quiver is $\bar Q=(I,\bar\Omega)$ with $\bar\Omega=\Omega\sqcup\Omega^*$.
We set $\varepsilon(h)=1$ if $h\in \Omega$ and $-1$ if $h\in \Omega^*$.
Let $h'$, $h''$ be the source and the target in $I$ of an arrow $h\in\Omega$.
For each vertex $i$ let 
$\Omega_{i\bullet}$ be the set of arrows in $\Omega$ with source $i$ and
$\Omega_{\bullet i}$ be the set of arrows with target $i$. Put
$\Omega_{ij}=\Omega_{\bullet j}\cap\Omega_{i\bullet}$ and
$\bar\Omega_{ij}=\Omega_{ij}\cup(\Omega_{ji})^*$.
If $i\neq j$ we write
$$q_i=|\Omega_{ii}|,\quad q_{ij}=|\Omega_{ij}|,\quad \bar q_{ij}=|\bar\Omega_{ij}|,\quad q=(q_{ij}\,;\,i,j\in I),\quad $$

\smallskip

Fix a tuple $v=(v_i\,;\,i\in I)$ in $\bbZ^I$.
The Ringel bilinear form $\langle \bullet,\bullet\rangle$ on $\bbZ^I$ is given by
$$\langle v,w\rangle=v\cdot w-\sum_{h\in\Omega}v_{h'}w_{h''},\quad v\cdot w=\sum_{i\in I}v_iw_i.$$
Let $(v,w)=\langle v,w\rangle+\langle w,v\rangle$ be the Euler bilinear form. To avoid confusions we may write
$(\bullet,\bullet)_Q$ for $(\bullet,\bullet)$. We set
\begin{align}d_v=v\cdot v-(v,v)/2.\end{align}
For each $i\in I$ let $\delta_i\in\mathbb{Z}^I$ denote the delta function at the vertex $i$.

\smallskip

Let $\k Q$ be the path algebra of $Q$.
A \emph{dimension vector} of $Q$ is a tuple $v\in\bbN^I$.
Let $\k^v$ denote the $I$-graded vector space $\k^v=\bigoplus_{i\in I}\k^{v_i}$.
We may abbreviate 
$$V=\k^v,\quad V_i=\k^{v_i}.$$
Let $Rep(\k Q,v)$ be the set of representations of $\k Q$ in $V$, with its natural structure of affine $\k$-variety.
We abbreviate 
\begin{align}\label{f1}
\R(v)=Rep(\k \bar Q,v).
\end{align}
An element of $\R(v)$ is a pair $\bar x=(x,x^*)$ where $x=(x_h\,;\,h\in\Omega)$ belongs to $Rep(\k Q,v)$ and 
$x^*=(x_h\,;\,h\in\Omega^*)$
belongs to $Rep(\k Q^*,v)$.
Since $\R(v)$ decomposes as
\begin{align*}
\R(v)=Rep(\k Q,v)\oplus Rep(\k Q^*,v),\end{align*}
it has a canonical structure of a symplectic vector space.

\smallskip

\subsection{The moment map and the preprojective algebra}\label{sec:moment}\hfill\\

The algebraic group $G(v)=\prod_{i\in I}GL(v_i)$ acts by conjugation on $\R(v)$, preserving the symplectic form.
Let $\frakg(v)$ be the Lie algebra of $G(v)$. 
The moment map for the action of $G(v)$ on $\R(v)$ is the map $\mu\,:\,\R(v)\to\frakg(v)$ given by
$$\mu(\bar x)=[x,x^*]=\sum_{h\in\Omega} [x_{h},x_{h^*}].$$
We define
$\M(v)=\mu^{-1}(0).$
The preprojective algebra is defined as
$$\Pi= \k \bar{Q} / \langle \mu(\bar x)\,;\,\bar x\in\R(v)\rangle.$$

\smallskip

\begin{proposition}\label{P:CYpreproj} Assume that $Q$ is not of finite Dynkin type. Then the algebra $\Pi$ is Calabi-Yau of homological dimension two, i.e., there are
functorial isomorphisms
\begin{equation}\label{E:prepro1}
\Ext^i_\Pi(M,N) \simeq \Ext^{2-i}_\Pi(N,M)^*, \qquad i=0,1,2
\end{equation}
for any finite-dimensional $\Pi$-modules $M$, $N$.
Moreover, we have
\begin{equation}\label{E:prepro}
\langle M,N\rangle_\Pi:= \sum_{i=0}^2 (-1)^i \dim(\Ext^i_\Pi(M,N)) = (\dim M,\dim N)_Q.
\end{equation}
\end{proposition}

\smallskip

The above result is well-known to experts. We could not, however, locate a precise reference in this generality in the literature, and we give some details on the proof in the appendix. The restriction on the quiver is irrelevant to all the geometric constructions we need in this paper, as we can always embed a Dynkin quiver into a non-Dynkin one, and hence view a moduli space or stack of representations of a Dynkin quiver as a moduli space or stack of representations of a non-Dynkin quiver.

\smallskip

\subsection{The group action}\label{sec:groupaction}\hfill\\

Consider the group
\begin{align}\label{GO}
G_\Omega=\prod_{i\in I}SP(2q_i)\times\prod_{i\neq j}GL(q_{ij})\times G_\dil.
\end{align}
The second product is over all pairs $(i,j)$ in $I\times I$ such that $i\neq j$ and $G_\dil=\bbG_m$.
The group $G_\Omega$ acts on $\R(v)$ as follows. 
The factor $G_\dil$ acts by dilation on the summand $Rep(\k Q^*,v)$ of $\R(v)$. 
Write
\begin{align*}
\R(v)
&=\bigoplus_{i\in I}\Big(\End(V_i)\otimes\k^{q_i}\Big)\oplus\Big(\End(V_i)\otimes\k^{q_i}\Big)^*\oplus\\
&\,\,\,\quad\oplus\bigoplus_{i\neq j}\Big(\!\Hom(V_i,V_j)\otimes\k^{q_{ij}}\Big)\oplus
\Big(\!\Hom(V_i,V_j)\otimes\k^{q_{ij}}\Big)^*.
\end{align*}
Then $SP(2q_i)$ and $GL(q_{ij})$ act in the obvious way on the first and the second summands.
The zero set $\M(v)$ is preserved by the action of the groups $G(v)$ and $G_\Omega$.
Let $\theta,\theta^*$ be the cocharacters of $G_\Omega$ given by
\begin{align*}
\theta(z)=(z\oplus z^{-1}\,,\,z\,,\,z),\quad \theta^*(z)=(1\,,\,1\,,\,z).
\end{align*}
Thus $\theta(z)$ acts by multiplication by $z$ on the summand $Rep(\k Q,v)$ in $\R(v)$, 
while $\theta^*(z)$ acts by multiplication by $z$ on the summand $Rep(\k Q^*,v)$.
Fix a closed connected subgroup $ \TT$ of $G_\Omega$ which centralizes $\theta$.
To simplify, we assume $T$ to be a subtorus
\begin{align}\label{T}\TT=\TT_{\sp}\times\TT_{\dil},\quad
\TT_{\sp}\subseteq(\bbG_m)^\Omega,\quad\TT_\dil\subseteq G_\dil\end{align}
where a tuple $(z_h\,,\,z)$ acts by $z_h$ on $x_h$ and
by $z/z_h$ on $x_{h^*}$ for all $h\in\Omega$.

\smallskip

\subsection{Definition}\label{sec:sn}\hfill\\

Fix an increasing flag  $W$ of $I$-graded vector spaces in $V$
$$W=(\{0\}=W_0\subsetneq W_1\subsetneq\cdots\subsetneq W_r=V).$$
Then, we consider the closed subset $\Lambda_W$ of $\M(v)$ given by
\begin{equation}\label{E:seminilpdef}
\Lambda_{W}=\{\bar x\in \M(v)\,;\,x(W_p)\subseteq W_{p-1}\,,\,x^*(W_p)\subseteq W_{p}\}.
\end{equation}
Up to conjugacy by an element of $G(v),$ the flag $W$ is completely determined by the sequence of dimension vectors 
\begin{align*}\nu_1=\dim(W_1/W_0),\dots,\nu_r=\dim(W_r/W_{r-1}).
\end{align*}
The tuple $\nu=(\nu_1,\dots,\nu_r)$ is a \emph{composition of $v$}, i.e., it is a tuple of dimension vectors  
with sum $v$.  We write $\nu\vDash v$.
We say that the flag $W$ is \emph{of type $\nu$}.
Then, we define
\begin{align}\label{Lambda}\Lambda_{\nu}=G(v)\cdot\Lambda_{W}\subseteq \M(v),\end{align}
where the dot denotes the $G(v)$-action on $\M(v)$.

\smallskip

We say that a composition $\nu\vDash v$ is \emph{restricted}
if each $\nu_p$ is concentrated in a single vertex.
Then, we also say that the flag $W$ is \emph{restricted}.
The  \emph{semi-nilpotent variety} $\Lambda^0(v)$ and the \emph{strongly semi-nilpotent variety} $\Lambda^1(v)$ are the closed 
subsets 
of $ \M(v)$ given by
\begin{align}\label{decomposition}
\Lambda^0(v)=\bigcup_{\mu}\Lambda_{\mu},\quad\Lambda^1(v)=\bigcup_{\nu}\Lambda_{\nu},
\end{align}
where $\mu$ runs over the set of all compositions of $v$ and $\nu$ over the set of all restricted compositions of $v$.
We have an obvious closed embedding
$$\Lambda^1(v)\subseteq\Lambda^0(v).$$

\smallskip

The $T\times G(v)$-action on $\M(v)$ yields a $T\times G(v)$-action on $\Lambda^\flat(v)$ for each $\flat$.
In the next section we prove the following.

\begin{theorem}\label{thm:1.1} Assume that $\k=\bbC$. Let $\flat$ be either 0 or 1.\hfill 
\begin{itemize}[leftmargin=8mm]

\item[$\mathrm{(a)}$] $\Lambda^\flat(v)$ is a closed Lagrangian subvariety of $\R(v)$, of dimension $d_v$,

\item[$\mathrm{(b)}$] $H_*^{T\times G(v)}(\Lambda^\flat(v))$ is pure and even,


\item[$\mathrm{(c)}$] $\cl : \A_*^{T\times G(v)}(\Lambda^\flat(v))\to H_{2*}^{T\times G(v)}(\Lambda^\flat(v))$ is surjective,

\item[$\mathrm{(d)}$] $M_*^{T\times G(v)}(\Lambda^\flat(v))$ is free as an $M^*_{T}$-module.
\end{itemize}
\end{theorem}

\smallskip

\begin{remark}
Our definition of the semi-nilpotent and strongly semi-nilpotent varieties is the same as in \cite[\S 1.1]{BSV}.
If the quiver $Q$ has no oriented cycle which does not involve 1-loops, i.e., if any oriented cycle in $Q$ is a product of
1-loops, then the strongly semi-nilpotent variety and the semi-nilpotent one coincide.
\end{remark}

\smallskip

\begin{remark}
Switching the roles of $\Omega$ and $\Omega^*$ we get the notion of 
\emph{$*$ semi-nilpotent} variety which is denoted by the symbol 
$\Lambda^{*,\flat}(v)$. 
To avoid confusion we may also indicate the quiver $Q$ and write $\Lambda^\flat(v)/Q$ for $\Lambda^\flat(v)$.
Thus, we have a canonical isomorphism $\Lambda^\flat(v)/Q^*=\Lambda^{*,\,\flat}(v)/Q$.
We have also an isomorphism $\Lambda^\flat(v)/Q\to\Lambda^\flat(v)/Q^*$ 
such that $(x,x^*)\mapsto({}^tx,{}^tx^*)$, where the upper script ${}^t\bullet$ indicates
the transpose matrix.
\end{remark}

\smallskip

\begin{remark} In the theorem it is essential to consider the $T\times G(v)$-equivariant 
Borel-Moore homology rather than usual Borel-Moore homology. 
For instance, let $Q$ be the quiver of type $A_2$ and set $v=(1,1)$.
Then, we have 
$\Lambda^\flat(v)=\{(x,x^*)\in\k^2\,;\,x\,x^*=0\}$,
hence
$$H_*(\Lambda^\flat(v))=\bbQ[-2]^{\oplus 2}\oplus\bbQ[-1],$$ which
is neither even nor pure. The formulas in Theorem \ref{thm:dim} below, with $T=1$, yield the following equality of formal series
$$\sum_k\dim_\bbQ\!\big(H_{2k}(\Lambda^\flat(v)\,/\,G(v))\big)\,q^{2+k}=
(2q-1)\,\big(\sum_{k\geqslant 0}q^{-k}\,\big)^2.$$
This example also shows that $H_*^{T\times G(v)}(\Lambda^\flat(v))$ is not, in general, free as an $
H^*_{T\times G(v)}$-module.
\end{remark}

\smallskip

\begin{remark}\label{rem:ind}
Let $(v)$ be the composition of $v$ with only one term.
Then, we have
$$\Lambda_{(v)}=\{0\}\times Rep(\k Q^*,v).$$
\end{remark}

\smallskip

\subsection{The stratification}\label{sec:strat1}\hfill\\

Let us introduce two stratifications of $\Lambda^0(v)$ by $G(v)$-invariant locally closed subsets
\begin{align}\Lambda^0(v)=\bigsqcup_\nu\Lambda^\circ_{\nu}=\bigsqcup_\nu{}^\circ\!\Lambda_{\nu}.\end{align}
To do this, first consider the path algebra $\k \bar{Q}$ equipped with the $\mathbb{N}$-grading such that
all arrows in $\Omega$ are of degree $1$ and all arrows in $\bar{\Omega}$ are of degree $0$.
For $\bar x\in\Lambda^0(v)$ we consider the $\k$-subalgebra $A_{\bar x}$ of $\End(V)$ 
generated by 1 and the action of the elements $x_h$ with $h\in\bar\Omega$.
We equip it with the filtration such that
$$A_{\bar x}[\geqslant\! p]=\text{Im}( ev_{\bar x}:\k\bar{Q}[\geqslant\! p] \to A_{\bar x}),$$
where the map $ev_{\bar x }$ takes a path $\sigma=h_1 \cdots h_s$ to the element
$x_{h_1} \cdots x_{h_s}$. We set
$$r=\text{min}\{p>0\,;\,A_{\bar x}[\geqslant \!p]=0\}.$$

\smallskip

Let $W_{\bar x}$ be the unique increasing flag of $I$-graded vector spaces 
$$W_{\bar x}=(\{0\}=W_0\subseteq W_{1}\subseteq\cdots\subseteq W_r=V)$$
where $W_p$ is the annihilator of $A_{\bar x}[\geqslant \!p]$.  
Given an increasing flag $W$ of $I$-graded vector spaces in $V$ of type $\nu$,
we consider the open subsets
$\Lambda_{W}^\circ\subseteq\Lambda_{W}$  and $\Lambda_{\nu}^\circ\subseteq\Lambda_{\nu}$ given  by
\begin{align}
\label{Lambdacirc}
\begin{split}
\Lambda_{W}^\circ=\{\bar x\in\Lambda_W\,;\,W_{\bar x}=W\},\quad
\Lambda^\circ_{\nu}=G(v)\cdot\Lambda^\circ_{W}.
\end{split}
\end{align}
Note that the set $\Lambda_W$ may be reducible, and  $\Lambda^\circ_W$ 
may not be dense in $\Lambda_W$.

\smallskip

Consider also the decreasing flag 
$$F^{\bar x}=(\{0\}=F^r\subseteq F^{r-1}\subseteq\cdots\subseteq F^0=V)$$ 
given by $F^p=A_{\bar x}[\geqslant \!p](V)$ for each $p$,
and the sets
\begin{align*}
{}^\circ\!\Lambda_{F}=\{\bar x\in\Lambda_F\,;\,F^{\bar x}=F\},\quad
{}^\circ\!\Lambda_{\nu}=G(v)\cdot{}^\circ\!\Lambda_{F}.
\end{align*}
Note that $F^{\bar x}$ is of type $\nu$ if we have $\nu_p=\dim(F^{p-1}/F^{p})$ for all $p$.

\smallskip

\subsection{The one vertex quiver}\label{sec:1.sgquivers}\hfill\\

In this section we consider in more details the case of the quiver $\J(i,q)$
with vertex set $\{i\}$ and $q$ loops. Then, the dimension vector $v$ is an integer
and we have $\Lambda^0(v)=\Lambda^1(v)$. We write $\Lambda$ for either $\Lambda^0$ or $\Lambda^1$.

\smallskip

\subsubsection{The isotropic case}\label{sec:isotrope}

Assume that $q=1$.
Then $Q$ is the Jordan quiver and the variety $\Lambda(v)$ is the set of commuting pairs $(x,x^*)$ in $\frakg(v)\times\frakg(v)$ with $x$ nilpotent.
We identify $\frakg(v)\times\frakg(v)$ with $\frakg(v)\times\frakg(v)^*$ via the trace map.
For each partition $\nu\vdash v$ let $\nu^*$ be the partition dual to $\nu$ and
let $O_\nu$ be the nilpotent $G(v)$-orbit of Jordan type $\nu^*$.
Let $P\subseteq G(v)$ be the standard parabolic subgroup with block-type $\nu$.
Let $U,$ $L$ be the unipotent radical of $P$ and its standard Levi complement.
Let $\fraku$, $\frakl$, $\frakp$ be their Lie algebras.
Then $O_\nu$ is the $G(v)$-saturation of the unique dense 
$P$-orbit $\fraku^\circ$ in the Lie algebra $\fraku$ of $U$.
In other words $O_\nu$ is the \emph{Richardson} orbit associated with the standard parabolic $P$ of $G(v)$, see e.g.,
\cite[thm.~7.1.1]{CMcG}, \cite[\S 3]{He78} for details (note that $O_\nu$ is usually labelled by the partition $\nu^*$).
Then, we have 
$$\Lambda_{W}=\Lambda(v)\cap(\fraku\times\frakp),\quad\Lambda_{W}^\circ=\Lambda(v)\cap(\fraku^\circ\times\frakp).$$
We equip the set of partitions $\{\nu\,;\,\nu\vdash v\}$ with the \emph{anti-dominant ordering} 
$$\mu\,\lescc\,\nu\iff\sum_{p\leqslant k}\mu_p\geqslant\sum_{p\leqslant k}\nu_p\,,\ \forall k.$$
We define $\Lambda_{\lescc\,\nu}^\circ=
\bigsqcup_{\mu\,\lescc\,\nu}\Lambda_{\mu}^\circ.$

\smallskip

\begin{proposition} \label{prop:2.6a} Assume that $q=1$. Then, we have
\begin{itemize}[leftmargin=8mm]
\item[$\mathrm{(a)}$] $\Lambda_{\nu}^\circ$ is the conormal bundle to $O_\nu$,
\item[$\mathrm{(b)}$] $\Irr(\Lambda(v))=\{\overline{\Lambda_{\nu}^\circ}\,;\,\nu\vdash v\}$,
\item[$\mathrm{(c)}$] if $\nu \vdash v$ then $\overline{\Lambda^\circ_{\nu}}
\subseteq\Lambda_{\nu}\subseteq\Lambda_{\lescc\,\nu}^\circ,$
\item[$\mathrm{(d)}$] $\Lambda_W$, $\Lambda_\nu$ are reducible in general,
\item[$\mathrm{(e)}$] the obvious map $G(v)\times_P\Lambda_W^\circ\to\Lambda^\circ_\nu$ is an isomorphism.
\end{itemize}
\qed
\end{proposition}

\smallskip

\subsubsection{The hyperbolic case}

Assume that $q>1$.
We equip the set of compositions $\{\nu\,;\,\nu\vDash v\}$ 
with the anti-dominant ordering.

\smallskip

\begin{proposition}\label{prop:geomirrlambda} Assume that $q>1$. Then, we have
\begin{itemize}[leftmargin=8mm]
\item[$\mathrm{(a)}$] $\Lambda_W$ is irreducible, $\Lambda_{\nu}$ is irreducible and Lagrangian in $\R(v)$,
\item[$\mathrm{(b)}$] $\Irr(\Lambda(v))=\{\Lambda_{\nu}\,;\,\nu\vDash v\}$,
\item[$\mathrm{(c)}$] if $\nu \vDash v$ then $\overline{\Lambda^\circ_{\nu}}
=\Lambda_{\nu}\subseteq\Lambda_{\lescc\,\nu}^\circ,$
\item[$\mathrm{(d)}$] the obvious map $G(v)\times_P\Lambda_W^\circ\to\Lambda_\nu^\circ$ is an isomorphism.

\end{itemize}
\end{proposition}

\begin{proof} Note first that $\Lambda^0(v)=\Lambda^1(v)$ since $Q$ has no oriented cycle which is not a
 product of $1$-loops. We write $\Lambda(v)=\Lambda^0(v)=\Lambda^1(v)$.
 By \cite{B14} the set $\Lambda^{\circ}_{\nu}$ is non-empty for any composition
$\nu \vDash v$ and we have
$$\Irr(\Lambda(v))=\{\overline{\Lambda^\circ_{\nu}}\,;\,\nu \vDash v\}.$$ 
Let us now fix $\nu=(\nu_1, \ldots, \nu_r)$. It is clear that $\overline{\Lambda^\circ_{\nu}} \subset \Lambda_{\nu}$.
Let us prove that $\Lambda_{\nu}$ is irreducible of dimension $qv^2$. Fix a flag $W$ of type $\nu$ and let
$P$ be its stabilizer in $G(v)$. Since the map $G(v) \times_P \Lambda_{W} \to \Lambda_{\nu}$
is surjective, it is enough to show that $\Lambda_W$ is irreducible of dimension 
$$d(\nu)=qv^2-\sum_{i<j} \nu_i\nu_j.$$ This is a consequence
of the following two facts~:

\smallskip

\begin{itemize}

\item[(i)] every irreducible component of $\Lambda_W$ is of dimension $\geqslant d(\nu)$,

\item[(ii)] for any composition $\mu \neq \nu$ we have $\Lambda^{\circ}_{\mu} \not\subseteq \Lambda_{\nu}$.

\end{itemize}
\smallskip

Indeed, claim (i) implies that the closed subset $\Lambda_\nu\subseteq\Lambda(v)$ is of the form
$\Lambda_{\nu}=\overline{\Lambda^\circ_{\nu}} \cup \overline{\Lambda^\circ_{\mu_1}} \cup \cdots \cup
\overline{\Lambda^\circ_{\mu_r}}$ for some compositions $\mu_1, \ldots, \mu_r$. 
Now, claim (ii) implies
that $r=0$ and thus $\Lambda_{\nu}=\overline{\Lambda^\circ_{\nu}}$. 

\smallskip

It remains to prove both claims. For part (i) we write 
$$W_i =\k^{\nu_{\leqslant i}},\quad \nu_{\leqslant i}=\sum_{j \leqslant i} \nu_j.$$ 
View $\Lambda_W$ as a closed subvariety of the product
$$\M(\nu_{\leqslant 1}) \times \M(\nu_{\leqslant 2}) \times \cdots \times \M(v)$$
via the map $\bar{x} \mapsto (\bar{x}|_{W_1}\,,\, \bar{x}|_{W_2}\,, \cdots).$ 
For each $(a,b) \in \bbN^2$ with $a<b$, consider the set
$$\M(a,b)=\Lambda_{\k^a,\k^b}=\{(\bar{x}_a, \bar{x}_b)\in \M(a) \times \M(b)\,;\, 
\bar{x}_b|_{\k^a}=\bar{x}_a\,,\,x_b(\k^b) \subseteq \k^a\}.$$
It is the zero locus of $a^2+a(b-a)-1$ equations in a vector space of dimension $2qa^2+q(b-a)^2 + 2qa(b-a)$.
Hence every irreducible component of $X(a,b)$ is of dimension 
$$\geqslant (2q-1)a^2+q(b-a)^2+(2q-1)a(b-a) +1=qb^2+qa^2+1-ab.$$
The following result is proved in Appendix A.

\smallskip

\begin{lemma} \label{L:muirr}
For any $q>1$ and $v>0$ the variety $\M(v)$ is irreducible and of dimension $(2q-1)v^2+1$.
\qed
\end{lemma}

We deduce that the codimension of any  irreducible component of $\M(a,b)$ is 
$$\leqslant c(a,b):=(q-1)a^2+(q-1)b^2+1+ab$$ in $\M(a) \times \M(b).$ 
Now observe that
$$\Lambda_W=
\bigcap_{i \geqslant 0} \M(\nu_{\leqslant 1}) \times \cdots \times \M(\nu_{\leqslant i}, \nu_{\leqslant i+1}) \times \cdots \times \M(v).$$
It follows that the codimension of any irreducible component of $\Lambda_W$ is 
$\leqslant \sum_{i \geqslant 0} c(\nu_{\leqslant i}\,,\, \nu_{\leqslant i+1})$
in $\M(\nu_{\leqslant 1}) \times \cdots \times \M(v)$ and hence of dimension
$$\geqslant \sum_{i \geqslant 1}  ((2q-1)\nu_{\leqslant i}^2+1) - \sum_{i \geqslant 0}c(\nu_{\leqslant i}, \nu_{\leqslant i+1})= 
qv^2+\sum_{i<s} \nu_{\leqslant i}^2-\sum_{i<s} \nu_{\leqslant i} \nu_{\leqslant i+1}= d(\nu)$$
as wanted. This proves claim (i).

\smallskip

We now turn to claim (ii). To show that $\Lambda^{\circ}_{\mu} \not\subseteq \Lambda_{\nu}$ we will prove
the existence of an element $\bar{x} \in \Lambda^{\circ}_{\mu}$ which satisfies the semi-nilpotency condition (\ref{E:seminilpdef})
with respect to a single flag (necessarily of type $\mu$). 
For this we will use the following basic construction. Pick distinct non-zero simple $\Pi$-modules $S_1, \ldots,S_k$.
We have $\Hom(S_i,S_j)=\Ext^2(S_i,S_j)=\{0\}$ for any $i \neq j$. Thus 
$$\dim(\Ext^1(S_i,S_j))=-\langle S_i,S_j\rangle_{\Pi} =2(q-1)\dim(S_i)\dim(S_j)>0.$$
Choose
$\xi_1 \in \Ext^1(S_2,S_1)\backslash \{0\}$ and denote by 
$$0 \to S_1 \to M_1 \to S_2 \to 0$$
the corresponding short exact sequence. We have
\begin{equation}
\begin{split}
\dim(\Ext^1(S_3,M_1)) &=-\langle S_2, M_1\rangle_{\Pi} + \dim(\Hom(S_3,M_1)) + \dim(\Ext^2(S_3,M_1))\\
&=-\langle S_2, M_1\rangle_{\Pi} >0.
\end{split}
\end{equation}
We may thus choose $\xi_2 \in \Ext^1(S_3,M_1)\backslash \{0\}$ and denote by 
$$0 \to M_1 \to M_2 \to S_3 \to 0$$
the corresponding short exact sequence. Continuing this process we construct a sequence of extensions 
$0 \to M_{l-1} \to M_{l} \to S_{1+1} \to 0$ for $l=1, \ldots, k-1$.
We claim that $M_{k-1}$ has a unique composition sequence
$$0 \subset S_1 = M_0 \subset M_1 \subset \cdots \subset M_{k-2} \subset M_{k-1}.$$
Arguing by induction we see that it is enough to prove that $M_{k-1}$ has a simple socle $S_1$. 
From the fact that the $S_i$'s are distinct we easily get that
$\Hom(S_1, M_{k-1}/S_1)=\{0\}$ and thus $\Hom(S_1,M_{k-1})=\k$. We now prove that
$\Hom(S_l,M_{k-1})=\{0\}$ for any $l \neq 1$. As above, we have $\Hom(S_l,M_{k-1}/ M_{l-1})=\{0\}$
and thus $\Hom(S_l,M_{k-1})=\Hom(S_l,M_{l-1})$. Because $\Hom(S_l,M_{l-2})=\{0\}$ and $S_l$ is simple,
any map $S_l \to M_{l-1}$ is, up to a scalar, a section of the canonical map $M_{l-1} \to S_l$. But
$M_{l-1}$ is a non-split extension of $S_l$ and $M_{l-2}$. Thus,  such a section doesn't exist.
We conclude that $\Hom(S_l, M_{l-1})=\{0\}$ as wanted. 

\smallskip

We may now prove the claim (ii) above. 
Fix $\mu=(\mu_1, \ldots, \mu_k) \neq \nu$. Because the ground field $\k$ is infinite, we may choose distinct simple
$\k Q$-modules (and thus also $\Pi$-modules) $S_1, \ldots, S_k$ of respective dimensions $\mu_1, \ldots, \mu_k$ and build
a $\Pi$-module $M_{k-1}$ as above. It is easy to see that $M_{k-1} \in \Lambda^\circ_{\mu}$ but $M_{k-1}$ doesn't
stabilize any flag of subspaces of type $\nu$. This proves the parts (a) and (b) of the proposition. 

\smallskip

The first part of the statement (c) was proved
above. To prove the rest of (c), note that 
\begin{align*}\Lambda_{W'}^\circ\cap\Lambda_{W}\neq\emptyset
&\Rightarrow W_p\subseteq W'_p\,,\ \forall p,\\
&\Rightarrow \nu'\lescc\,\nu,
\end{align*}
where $\nu$, $\nu'$ are the types of $W$, $W'$ respectively.
We deduce that we have $\Lambda_{\nu}\subseteq\Lambda_{\lescc\nu}^\circ.$
Part (d) of the proposition is obvious.
\end{proof}

\smallskip

\subsection{Quiver varieties}\label{sec:1.3}\hfill\\

This section is a reminder on the Nakajima quiver varieties associated with $Q$.
We assume that the reader is familiar with the formalism of quiver varieties and we refer to
\cite{CB01}, \cite{CB03}, \cite{N94}, \cite{N98},  \cite{N01}, \cite{N04}, \cite{N09} for the proofs of the facts recalled below.

\smallskip

\subsubsection{Basics}\label{sec:QV}
The space of representations of dimension vectors $v,$ $w$ of the \emph{framed quiver} associated with $\bar Q$
is
\begin{align*}
\R(v,w)=\R(v)\oplus\Hom_I(W,V)\oplus\Hom_I(V,W),
\end{align*}
where $\Hom_I(V,W)$ is the set of $I$-graded $\k$-linear homomorphisms and
$$V=\k^v,\quad W=\k^w,\quad V_i=\k^{v_i},\quad W_i=\k^{w_i}.$$
Since $\R(v,w)$ is canonically identified with the cotangent of a vector space, it admits a canonical symplectic structure.

\smallskip

The algebraic group $T\times G(v)\times G(w)$ acts on $\R(v,w)$ so that the element
$(z_h,z)\in T$ takes $(a,a^*)$ to $(a,za^*)$.
Consider the cocharacters
$\xi$, $\xi^*$ of $T\times G(v)$ given by
\begin{align*}
\xi(t)=(t\,,\,\theta(t)),\quad \xi^*(t)=(1\,,\,\theta^*(t))
\end{align*}
We can view them as cocharacters of $T\times G(v)\times G(w)$ 
under the obvious inclusion 
$$T\times G(v)\subseteq T\times G(v)\times G(w).$$
Then $\xi(t)$, $\xi^*(t)$ act by multiplication by $t$ on the summands of $\R(v,w)$ given by
$$Rep(\k Q,v)\oplus\Hom_I(W,V),\quad Rep(\k Q,v)^*\oplus\Hom_I(V,W).$$
Equivalently, we may write
\begin{align}\label{action0}\begin{split}
\xi(t)\cdot(\bar x,\bar a)&=(tx,x^*,ta,a^*),\\ 
\xi^*(t)\cdot(\bar x,\bar a)&=(x,t^*x^*,a,t^*a^*).\end{split}\end{align}

\smallskip

The $G(v)$-action on $\R(v,w)$ preserves the symplectic form and admits the
moment map 
$$\mu:\R(v,w)\to\frakg(v),\quad(\bar x,\bar a)\mapsto[x,x^*]+aa^*,$$
where we write 
\begin{align*}
&\bar x=(x,x^*)\in \R(v),\\
&\bar a=(a,a^*)\in\Hom_I(w,v)\oplus\Hom_I(v,w).
\end{align*}
The categorical quotient of the zero set 
\begin{align*}\M(v,w)=\mu^{-1}(0)\end{align*} 
by $G(v)$ is the variety
$$\frakM_0(v,w)=\M(v,w)/\!\!/G(v)=\text{Spec}\Big(\k[\M(v,w)]^{\,G(v)}\Bigr).$$ 
It is affine, reduced, irreducible, singular in general and $T\times G(w)$-equivariant. 
Let 
$$\rho_0:\M(v,w)\to\frakM_0(v,w)$$
denote the canonical map. We may write $0=\rho_0(0,0)$.
The set of $\k$-points of $\frakM_0(v,w)$ is in bijection with the set of closed $G(v)$-orbits in 
$\M(v,w)$ so that
$\rho_0(\bar x,\bar a)$ is identified with the unique closed orbit in the closure of the $G(v)$-orbit of $(\bar x,\bar a)$.
Equivalently, the set of $\k$-points of $\frakM_0(v,w)$ is in bijection with the set of isomorphism classes of
semisimple representations in $\M(v,w)$ and $\rho_0$ maps a representation to the sum of its constituents.

\smallskip

Given a character $\theta$ of $G(v)$ we consider the space of semi-invariants of weight $\theta$
$$\k[\M(v,w)]^{\,\theta}\subseteq\k[\M(v,w)].$$
We have the $T\times G(w)$-equivariant projective morphism
$$\pi_\theta : \frakM_\theta(v,w)=\text{Proj}\Big(\bigoplus_{n\in\bbN}\k[\M(v,w)]^{\,\theta^n}\Bigr) \to \frakM_0(v,w).$$
The Hilbert-Mumford criterion implies that 
$\frakM_\theta(v,w)$ is the geometric quotient by $G(v)$ of an open subset $\M_\theta(v,w)$ of $\M(v,w)$ 
consisting of the $\theta$-\emph{semistable} representations.
Let 
\begin{align}\label{f4}
\begin{split}
\M_\theta(v,w)\subset\M(v,w),\quad
\R_\theta(v,w)\subset\R(v,w)
\end{split}
\end{align} 
be the open subsets of semistable points. 
Replacing everywhere $\M(v,w)$ by $\R(v,w)$ we define an open subset $\R_\theta(v,w)$ of $\R(v,w)$ such that
$\M_\theta(v,w)=\R_\theta(v,w)\cap \M(v,w)$.
Let $\rho_\theta(\bar x,\bar a)$ be the image in $\frakM_\theta(v,w)$ of the tuple $(\bar x,\bar a)\in\M_\theta(v,w)$.
We have 
$$\pi_\theta\rho_\theta(\bar x,\bar a)=\rho_0(\bar x,\bar a).$$

\smallskip

We say that the character $\theta$ given by $\theta(g)=\prod_{i\in I}\det(g_i)^{-\theta_i}$ is \textit{generic} 
if neither equations
$$\sum_{i \in I} \theta_i u_i=0 \ \text{and}\  \sum_{i \in I} \theta_i u_i + \theta_{\infty}=0$$
with 
$\theta_\infty=-\sum_{i\in I}\theta_iv_i$
have integer solutions $(u_i)$ satisfying $0 \leqslant u_i \leqslant v_i$ other than the trivial solutions 
$(u_i)=0$ or $(u_i)=(v_i).$ 
If $\theta$ is generic then any semistable pair $(\bar x,\bar a)$ is stable and in that case the map 
$$\rho_\theta:\M_\theta(v,w) \to \frakM_{\theta}(v,w)$$ is a $G(v)$-torsor.
In particular, the variety $\frakM_{\theta}(v,w)$ is smooth, symplectic, of dimension 
\begin{equation}\label{E:dimquivervar}
d_{v,w}=2\,v\cdot w-(v,v)_Q.
\end{equation} 

\smallskip

The character $s$ given by $s(g)=\prod_{i\in I}\det(g_i)^{-1}$ is generic.
We abbreviate
\begin{align*}
[\bar x,\bar a]=\rho_s(\bar x,\bar a),\quad \frakM(v,w)=\frakM_s(v,w),\quad \pi=\pi_s
\end{align*}
and we write (semi)stable for $s$-(semi)stable.
A representation $(\bar x,\bar a)$ in $\R(v,w)$ is semistable
if and only if it does not admit any nonzero subrepresentation
whose dimension vector belongs to $\bbN^I\times\{0\}$.

\smallskip

\begin{remark}\label{rem:torsor} Assume that $\k=\bbC$.
By \cite{N94}, the $G(v)$-action on $\M_s(v,w)$ is \emph{set-theoretically free} and $\M_s(v,w)$ is a smooth scheme.
Hence, the natural morphism of smooth schemes
$$G(v)\times\M_s(v,w)\to\M_s(v,w)\times_{\frakM(v,w)}\M_s(v,w)$$
is a bijection on closed points and, therefore, it is an isomorphism.
In other words, the $G(v)$-action on $\M_s(v,w)$ is \emph{free}. Hence the $G(v)$-scheme $\M_s(v,w)$ is a 
$G(v)$-torsor over $\frakM(v,w)$ and
the quotient stack $\M_s(v,w)/G(v)$ is represented by the scheme 
$\frakM(v,w)=\M_s(v,w)/\!\!/G(v)$.
\end{remark}

\smallskip

\subsubsection{Crawley-Boevey's trick}\label{S:CBtrick} It may be useful to realize the 
quiver varieties as moduli spaces of representations of some preprojective algebra. 
More precisely, consider the quiver $\tilde{Q}=(\tilde{I}, \tilde{\Omega})$ obtained from $Q$ 
by adding one new vertex $\infty$ and $w_i$ arrows from $\infty$ to the vertex $i$ for all $i \in I$.
For each $v \in \mathbb{N}^I$ set 
$$\tilde{v}=v+\delta_\infty\ \text{if}\ w\neq 0,\ \text{and}\ \tilde v=v\ \text{else}.$$
If $w\neq 0$ we identify $G(v)$ with $PGL(\tilde{v})$. 
Then, there is a 
canonical $G(v)$-equivariant isomorphism $\M(v,w) \simeq \M(\tilde{v})$ and $\frakM(v,w)$ or $\frakM_0(v,w)$ 
may be viewed as moduli spaces 
of (stable, resp. semisimple)  modules of dimension $\tilde{v}$
over the preprojective algebra $\tilde\Pi$.
The symmetric bilinear forms for $Q$ and $\tilde{Q}$ are related as follows
\begin{equation}\label{E:eulerformqtilde}
(v+a\delta_\infty\,,\,v'+a'\delta_\infty)_{\tilde{Q}}=(v,v')_Q +  a\cdot a'-\sum_i w_i (v_ia'_i+ v'_i a_i).
\end{equation}
Let $\tilde\Pi$ be the preprojective algebra of $\tilde Q$.

\smallskip

\subsubsection{Representation types}

If the representation $z \in\M(v,w)$ is semisimple then we can decompose it into its simple constituents
$z=z_1^{\oplus d_1}\oplus z_2^{\oplus d_2}\oplus\cdots z_s^{\oplus d_s}$ where the $z_r$'s are non-isomorphic simples.
If $u_r=(v_r,w_r)$ is the dimension vector of $z_r$, i.e., if $z_r$ is a simple representation in $\M(v_r,w_r)$, 
then we say that $z$ has the 
\emph{representation type} 
$$\tau=(d_1,u_1\,;\,d_2,u_2\,;\,\dots\,;\,d_s,u_s).$$
If $w\neq 0$ and $z$ is stable,
then there is a unique  integer $r$ such that $w_r=w$ and $w_{r'}=0$ for all $r'\neq r$, hence
we may assume that the representation type of $z$ has the following form
\begin{align*}\tau=(1,v_1,w\,;\,d_2,v_2\,;\,\dots\,;\,d_s,v_s),\end{align*}
where $(v_1,v_2,\dots,v_s)\vDash v$ and the tuples
$(d_2,v_2),\dots(d_s,v_s)$ are only defined up to a permutation.
Let $RT(v,w)$ be the set of all representation types of dimension $(v,w)$. Let
$$\frakM_0(\tau)\subseteq\frakM_0(v,w)$$
be the set of semisimple representations with representation type equal to $\tau$.
We have the following stratification by smooth irreducible locally closed subsets
$$\frakM_0(v,w)=\bigsqcup_{\tau\in RT(v,w)}\frakM_0(\tau).$$
Defining $\tilde u_1,\dots,\tilde u_s$ as in \S\ref{S:CBtrick} we get
$$\dim \frakM_0(\tau)=\sum_{r=1}^s\big(2-(\tilde u_r,\tilde u_r)_{\tilde Q}\big).$$

\smallskip

Since $\frakM_0(v,w)$ is irreducible, there is a unique representation type $\kappa_{v,w}$ such that
$\frakM_0(\kappa_{v,w})$ is a dense open subset of $\frakM_0(v,w)$.
We call it the \emph{generic representation type} of $RT(v,w)$.

\smallskip

The stabilizer in $G(v)$ of an arbitrary element of $\M(\tau)$
is a reductive group which 
is conjugate by an element of $G(v)$ to the group $G_\tau=\prod_{r=1}^sGL(d_r)$. 
Write $\tau'\leqslant\tau$ if and only if $G_\tau$ is conjugate to a subgroup of $G_{\tau'}$.
Then, we have
$$\overline{\frakM_0(\tau)}=\bigsqcup_{\tau'\leqslant\tau}\frakM_0(\tau').$$

\smallskip

Given two representation types
\begin{align*}
\tau=(1,v_1,w\,;\,d_2,v_2\,;\,\dots\,;\,d_s,v_s)\in RT(v,w),\\
\kappa=(1,u_1,z\,;\,e_2,u_2\,;\,\dots\,;\,e_r,u_r)\in RT(u,z),
\end{align*}
we define their sum by
$$\tau\oplus\kappa=(1,v_1+u_1,w+z\,;\,d_2,v_2\,;\,\dots\,;\,d_s,v_s\,;\,e_2,u_2\,;\,\dots\,;\,e_r,u_r).$$
Whenever $\sum_{t=1}^rv_t\leqslant v$ and 
$\sum_{t=1}^rw_t\leqslant w$,
the direct sum yields a closed embedding
\begin{align}\label{can}
\oplus:\prod_{t=1}^r\frakM_0(v_t,w_t)\to\frakM_0(v,w).\end{align}
The relation with representation types is given by the following relation
$$\frakM_0(\tau)\oplus\frakM_0(\kappa)\subseteq\overline{\frakM_0(\tau\oplus\kappa)}.$$

\smallskip

For each representation type $\tau$ we write 
\begin{align*}
\M(\tau)&=\rho_0^{-1}(\frakM_0(\tau)),\\
\frakM(\tau)&=\pi^{-1}(\frakM_0(\tau))=\M_s(\tau)/\!\!/ G(v).
\end{align*}
If $\frakM(\tau)\neq\emptyset$ then the map $\pi$ restricts to a locally trivial fibration 
$\frakM(\tau)\to\frakM_0(\tau)$ such that, see \cite{CB01} and \cite[cor.~6.4]{CB03},
\begin{align}\label{ss}
\dim\frakM(\tau)\leqslant\dim\frakM(v,w)/2+\dim\frakM_0(\tau)/2.
\end{align}
In particular, if the map $\pi:\frakM(v,w)\to\frakM_0(v,w)$ is birational, then it is semismall.

\smallskip

The map $\pi$ may not be surjective. Its image is 
an irreducible closed subvariety of $\frakM_0(v,w)$ which is a union of
strata. For later purposes, we will need the following stronger statement. Let us call \textit{dimension type}
of a semisimple representation $z$ the sequence $d=(d_u)$ where $d_u$ is the total number of simple 
representations 
(counted with multiplicity) of dimension $u$ occuring in $z$. 
The representation type determines the dimension type but the 
converse is false since we lose the information of the multiplicity of each individual simple representation of a given 
dimension vector. Consider the locally closed subvariety 
$$\frakM_0(\!(d)\!)\subseteq \frakM_0(v,w)$$
parametrizing semisimple representations of dimension type $d$.

\smallskip

\begin{proposition}\label{P:impi} For any $v,$ $w$,  the image of $\pi$ is a union of strata $\frakM_0(\!(d)\!)$. 
\end{proposition}
\begin{proof} See the appendix.
\end{proof}

\smallskip

\subsubsection{The Bialynicki-Birula decomposition}\label{sec:BBQ}
Let $\gamma$ be a cocharacter of $T\times G(w)$. Composing it with the $T\times G(w)$-action we get a
$\bbG_m$-action $\bullet$ on $\frakM(v,w)$. We want to describe 
the fixed point locus $\frakM(v,w)^\bullet$ and the Bialynicki-Birula \emph{attracting variety}
\begin{align*}\frakL(v,w)^\bullet=\{z\in\frakM(v,w)\,;\,\exists\lim_{t\to 0}t\bullet z\}.\end{align*}
To do this, fix a cocharacter $\rho$ of $G(v)$. 
Since the $G(v)$-action on $\M_s(v,w)$ commutes with the $T\times G(w)$-action, we can view the product $\gamma\rho$ as
a cocharacter of $T\times G(v)\times G(w)$.
Let $L$ be the centralizer of $\rho$ in $G(v)$ and set
\begin{align}\label{274}
\begin{split}
P&=\{g\in G(v)\,;\,\exists\lim_{t\to 0}\rho(t)\,g\,\rho(t)^{-1}\},\\
U&=\{g\in G(v)\,;\,\lim_{t\to 0}\rho(t)\,g\,\rho(t)^{-1}=1\},\\
\M[\rho]&=\{z\in\M_s(v,w)\,;\,\gamma\rho(t)\cdot z=z\,,\,\forall t\},\\
\L[\rho]&=\{z\in\M_s(v,w)\,;\,\lim_{t\to 0}\gamma\rho(t)\cdot z\in\M[\rho]\},\\
\frakM[\rho]&=\big(G(v)\cdot\M[\rho]\big)\,/\!\!/\,G(v),\\
\frakL[\rho]&=\big(G(v)\cdot\L[\rho]\big)\,/\!\!/\,G(v).
\end{split}
\end{align}

\smallskip

\begin{proposition}\label{prop:1.10}
\hfill
\begin{itemize}[leftmargin=8mm]

\item[$\mathrm{(a)}$] The  maps $G(v)\times_L\M[\rho]\to G(v)\cdot\M[\rho]$ and $G(v)\times_P\L[\rho]\to G(v)\cdot\L[\rho]$
are invertible, 

\item[$\mathrm{(b)}$] $\frakM[\rho]$ is a sum of connected components of $\frakM(v,w)^\bullet$
such that $\frakM(v,w)^\bullet=\bigsqcup_\rho\frakM[\rho],$

\item[$\mathrm{(c)}$] 
$\frakL[\rho]=\{z\in\frakM(v,w)\,;\,\lim_{t\to 0}t\bullet z\in\frakM[\rho]\}$
and
$\frakL(v,w)^\bullet=\bigsqcup_\rho\frakL[\rho].$
\end{itemize}
\end{proposition}

\smallskip

\begin{proof} To prove the first claim of (a), observe that for two distinct
cocharacters $\rho$ and $\rho'$ the freeness of the $G(v)$-action on $\M_s(v,w)$ implies that $\M[\rho]\cap\M[\rho']=\emptyset$.
For the second claim we must check that for each $g\in G(v)$ we have 
$$\L[\rho]\cap(g\cdot\L[\rho])=
\begin{cases}\emptyset&\text{if}\ g\notin P,\\\L[\rho]&\text{else.}\end{cases}
$$
If $z$ and $g\cdot z$ belong to $\L[\rho]$, then the following limits exist in $\M_s(v,w)$
$$\lim_{t\to 0}\gamma\rho(t)\cdot z,\quad\lim_{t\to 0}(\rho(t)\,g\,\rho(t)^{-1})\cdot(\gamma\rho(t)\cdot z).$$ 
Since the $G(v)$-action on $\M_s(v,w)$ is free, this implies that $g\in P$.
Part (b) is well-known, see e.g. \cite[\S 4]{N01}.
Let us sketch a proof of (c). 
Put
$$\frakL'[\rho]=\{z\in\frakM(v,w)\,;\,\lim_{t\to 0}t\bullet z\in\frakM[\rho]\}.$$
We must check that $\frakL[\rho]=\frakL'[\rho]$. We have
$$\frakM[\rho]\subseteq\frakL[\rho]\subseteq\frakL'[\rho]$$
and the Bialynicki-Birula theorem implies that $\frakL'[\rho]$ is the disjoint union of affine space bundles 
$\frakL'_1,\dots,\frakL'_r$ over the connected 
components $\frakM_1,\dots,\frakM_r$ of $\frakM[\rho]$ with a 
contracting $\bbG_m$-action on each $\frakL'_i$ given by $\bullet$.
Set $\frakL_i=\frakL'_i\cap\frakL[\rho]$ for some $i=1,\dots r$.
It is not difficult to prove that $\frakL_i$ and $\frakL'_i$ have the same dimension.
Hence, since $\frakL_i$ is locally closed in $\frakL'_i$, it is a dense open subset.
Finally, since $\frakL_i$ contains $\frakM_i$ and is preserved by the contracting $\bbG_m$-action given by $\bullet$,
it equals $\frakL'_i$.

\end{proof}

\smallskip

\begin{remark}
We will actually prove later, see Corollary~\ref{cor:Mrhoconn}, that $\frakM[\rho]$ is connected for any $\rho$. In particular, the map $\frakL[\rho] \to \frakM[\rho]$ is an affine fibration.
\end{remark}

\smallskip

\subsection{Hecke correspondences}\label{S:Hecke}\hfill

\subsubsection{Basics}
Let $v_1$, $v_2$ and $w$ be dimension vectors.
Set $v=v_1+v_2$.
Fix an $I$-graded subspace $V_1\subseteq V$ with dimension vector $v_1$ and let
$P$ be the corresponding parabolic subgroup of $G(v)$.
Fix isomorphisms $\k^{v_1} \simeq V_1$ and $\k^{v_2} \simeq V/ V_1$.
Let us write 
\begin{align}\label{R1}
\begin{split}
\R_P&=\{\bar x\in \R(v)\,;\,\bar x(V_1)\subseteq V_1\}\oplus\Hom_I(w,v_1)\oplus\Hom_I(v,w),\\
\R_{s,P}&=\R_P\cap\R_s(v,w),\\
\H[v_1,v_2\,;\,w]&=\R_{P}\,\cap\,\M_s(v,w),\\
\frakh[v_1,v_2\,;\,w]&=\H[v_1,v_2\,;\,w]/\!\!/P.
\end{split}
\end{align}
Note that $\frakh[v_1,v_2\,;\,w]$ is indeed a geometric quotient. It is called a \emph{Hecke correspondence}. 
For each $(\bar{x}, \bar{a})\in\H[v_1,v_2\,;\,w]$ we write
\begin{align}\label{flag}
\bar{x}_1=\bar{x}|_{V_1}, \quad \bar{a}_{1}=\bar{a}|_{V_1},\quad\bar{x}_2=\bar{x}|_{V/V_1}.
\end{align}

\smallskip

\begin{proposition} \label{prop:hecke} \hfill
\begin{itemize}[leftmargin=8mm]
\item[$\mathrm{(a)}$] There is a closed embedding
$\frakh[v_1,v_2\,;\,w]\subseteq\frakM(v,w)\times\frakM(v_1,w)$,
\item[$\mathrm{(b)}$] 
$\frakh[v_1,v_2\,;\,w]$ is projective over $\frakM(v,w)$.
\end{itemize}
\end{proposition}

\smallskip

\begin{proof} 
We have
$$\frakh[v_1,v_2\,;\,w]=\big(G(v)\times_P\H[v_1,v_2\,;\,w]\big)/\!\!/G(v)$$
where
\begin{align}\label{HH}
\H[v_1,v_2\,;\,w]
&=\{z\in\M_s(v,w)\,;\,z(V_1\oplus W)\subseteq V_1\oplus W\}.
\end{align}
Thus, the assignment
$$\H[v_1,v_2\,;\,w]\to \M_s(v,w)\times\M_s(v_1,w),\quad
(\bar x,\bar a)\mapsto(\bar x,\bar a,\bar x_1,\bar a_1)$$ gives rise to a map
$\frakh[v_1,v_2\,;\,w]\to\frakM(v,w)\times\frakM(v_1,w).$
For each elements $(\bar x,\bar a)\in \M_s(v,w)$ and
$(\bar x_1,\bar a_1)\in \M_s(v_1,w)$ there is at most one embedding
$\phi_1:V_1\to V$ such that 
\begin{align*}\phi_1\circ\bar x_1=\bar x\circ\phi_1,\quad \phi_1\circ a_1=a,\quad a_1^*=a^*\circ\phi_1,
\end{align*}
because if there are two of them, say $\phi_1$ and $\phi'_1$, then
the $I$-graded subspace $\Im(\phi_1-\phi'_1)$ of $V$ is destabilizing for $(\bar x,\bar a)$, hence it is $\{0\}$.
This proves the part (a).
Part (b) follows from the projectivity of Grassmanians.
\end{proof}

\smallskip

For each $G(v_2)$-invariant locally closed subset 
$\SS\subseteq \M(v_2)$ we define
\begin{align} \label{heckeB} 
\begin{split}
\frakh[v_1,\SS\,;\,w]&=\H[v_1,\SS\,;\,w]/\!\!/P,\\
\H[v_1,\SS\,;\,w]&=\{(\bar x,\bar a)\in\H[v_1,v_2\,;\,w]\,;\,\bar x_2\in S\}.
\end{split}
\end{align}
We may abbreviate
\begin{align*}
\H[v_1,\SS]=\H[v_1,\SS\,;\,w],\quad
\frakh[v_1,\SS]=\frakh[v_1,\SS\,;\,w],\quad
\frakh[v_1,v_2]=\frakh[v_1,v_2\,;\,w], \quad\text{etc}.
\end{align*}

\smallskip

Recall that a subvariety $V$ of a symplectic manifold $X$ is \emph{isotropic} if the restriction of the symplectic form to 
the smooth locus of $V$ vanishes. Let $X^{\op}$ denote the manifold $X$ with the opposit symplectic form.

\smallskip

\begin{lemma}\label{lem:isotropic}
Let $\SS$ be an isotropic subvariety of the symplectic vector space $\R(v_2)$. 
Then $\frakh[v_1,\SS\,;\,w]$ is an isotropic subvariety of the symplectic manifold 
$\frakM(v,w)^{\op}\times\frakM(v_1,w)$.
\end{lemma}

\smallskip

\begin{proof} Consider the closed embedding, see \eqref{flag},
$$\phi:\R_{P}\to\R(v,w)\times \R(v_1,w),\quad(\bar x,\bar a)\to(\bar x,\bar a,\bar x_1,\bar a_1).$$
The variety $\frakM(v,w)^\op\times\frakM(v_1,w)$ is the symplectic reduction of the symplectic manifold
$$\R_s(v,w)^\op\times \R_s(v_1,w)$$ 
relative to the $G(v)\times G(v_1)$-action which is Hamiltonian.
The variety $\frakh[v_1,S]$ is the image in $\frakM(v,w)\times\frakM(v_1,w)$ of the $G(v)\times G(v_1)$-saturation of the set
\begin{align}\label{set}\big(\R_s(v,w)\times \R_s(v_1,w)\big)\cap\phi\big(\{(\bar x,\bar a)\in\R_P\,;\,\bar x_2\in S\}\big)
\end{align}
under the $G(v)\times G(v_1)$-action on $\R(v,w)^\op\times \R(v_1,w)$.
The symplectic form on $\frakM(v,w)^\op\times\frakM(v_1,w)$ is the unique symplectic form whose pullback to 
$\M_s(v,w)\times \M_s(v_1,w)$
equals the restriction of the symplectic form of $\R(v,w)^\op\times \R(v_1,w)$.
Therefore, to prove that $\frakh[v_1,\SS]$ is isotropic it is enough to check that \eqref{set}
is an isotropic subvariety of $\R(v,w)^\op\times \R(v_1,w)$.

\smallskip

According to  \S \ref{sec:QV}, the symplectic form on $\R(v,w)$ is given by
\begin{align*}
\omega\big((\bar X,\bar A)\,,\, (\bar Y,\bar B)\big)=\tr(XY^*-YX^*+AB^*-BA^*).
\end{align*}
Thus, if we have
$$X^*(V_1)\,,\,X(V_1)\,,\, A(W)\,,\,Y^*(V_1)\,,\,Y(V_1)\,,\, B(W)\subseteq V_1,$$ 
we deduce that
\begin{align*}
\omega\big((\bar X,\bar A)\,,\, (\bar Y,\bar B)\big)-\omega\big((\bar X_1,\bar A_1)\,,\, (\bar Y_1,\bar B_1)\big)=\tr(X_2Y_2^*-Y_2X_2^*),
\end{align*}
where $X_1\in\End(V_1)$ is the restriction of $X$, $X_2\in\End(V/V_1)$ is the induced operator, etc.
Since $\SS$ is an isotropic subvariety of $\R(v_2)$, we deduce that
\begin{align*}
\omega\big((\bar X,\bar A)\,,\, (\bar Y,\bar B)\big)-\omega\big((\bar X_1,\bar A_1)\,,\, (\bar Y_1,\bar B_1)\big)=0.
\end{align*}
\end{proof}



\smallskip

\begin{proposition} \label{prop:LAG} The Hecke correspondence 
$\frakh[v_1,\Lambda_{(v_2)}\,;\,w]$ is a closed Lagrangian
local complete intersection of the symplectic manifold $\frakM(v,w)^{\op} \times \frakM(v_1,w)$.
\end{proposition}

\smallskip

\begin{proof} By Lemma~\ref{lem:isotropic}, the variety $\frakh[v_1,\Lambda_{(v_2)}\,;\,w]$ is isotropic.
Since the $P$-action on $\H[v_1,\Lambda_{(v_2)}\,;\,w]$ is free, we deduce that
\begin{equation}
\begin{split}
\dim\H[v_1,\Lambda_{(v_2)}\,;\,w]&\leqslant \dim P + \dim\frakM(v,w)\dim \frakM(v_2,w)/2\\
&= v_1 \cdot v_1 + v_2 \cdot v_2 + v_1 \cdot v_2 + d_{v,w}/2 + d_{v_1,w}/2.
\end{split}
\end{equation}
Fix $I$-graded vector spaces $V_1$, $V_2$, $W$ with dimension vectors $v_1$, $v_2$, $w$.
For each $i$ we abbreviate $V_{1,i}=(V_1)_i$ and $V_{2,i}=(V_2)_i$.
By construction $\H[v_1,\Lambda_{(v_2)}\,;\,w]$ is an open subset of the zero fiber
of the map
$$\mu' : \M_{s}(v_1,w) \oplus Rep(\k Q^*,v_2) \oplus \Hom_I(V_2,W) \oplus  \bigoplus_{h \in \bar{\Omega}} \Hom(V_{2,h'}, V_{1,h''})
\to \Hom_I(v_2,v_1).$$
The domain of the above map being irreducible, using (\ref{E:dimquivervar}) we deduce that every 
irreducible component of
$\H[v_1,\Lambda_{(v_2)}]$ is of dimension at least
\begin{equation*}
\begin{split}
(d_{v_1,w} + v_1 \cdot v_1) &+ (v_2 \cdot v_2 - (v_2,v_2)_Q/2) + v_2 \cdot w + (2 v_1 \cdot v_2 -(v_1,v_2)_Q) -v_1 \cdot v_2\\
&=v_1 \cdot v_2 + d_{v,w}/2 + d_{v_1,w}/2.
\end{split}
\end{equation*}
It follows that the variety $\frakh[v_1,\Lambda_{(v_2)}\,;\,w]$ is a Lagrangian local complete intersection.
\end{proof}

\smallskip

\subsubsection{One vertex Hecke correspondences.} 
In this section we study  the Hecke correspondences in the particular 
case where $v_2$ is concentrated at a single vertex. 
Fix $i \in I$ and $l\in\bbN$ such that $v_2 = l \delta_i$. Write $q=q_i $.
We concentrate on the case $q>1$.
Recall that  
$$\Irr\big(\Lambda(v_2)\big)=\{\Lambda_{\nu}\,;\,\nu\vDash v_2\}.$$
Since the set $\Lambda_\nu$ is $G(v_2)$-invariant, the set $\frakh[v_1,\Lambda_\nu\,;\,w]$ is well-defined.
\smallskip

\begin{proposition}\label{prop:irrheckenu} If $q>1$ and
$\nu \vDash v_2$ then the Hecke correspondence
$\frakh[v_1,\Lambda_{\nu}\,;\,w]$ in 
$\frakM(v,w)^{\op} \times \frakM(v_1,w)$ is either Lagrangian and irreducible or empty.
\end{proposition}

\smallskip

\begin{proof} 
By Lemma~\ref{lem:isotropic}, the variety 
$\frakh[v_1,\Lambda_\nu\,;\,w]$ is isotropic, hence we have 
$$\dim\frakh[v_1,\Lambda_\nu\,;\,w] \leqslant d_{v,w}/2 + d_{v_1,w}/2.$$ 
We will show at the same time that 
$\frakh[v_1,\Lambda_\nu\,;\,w]$ is irreducible and that 
$$\dim\frakh[v_1,\Lambda_\nu\,;\,w] \geqslant d_{v,w}/2 + d_{v_1,w}/2.$$
Given $\nu=(\nu_1, \ldots, \nu_s)$, we fix a flag
$$F=(V_1=F_0\subsetneq F_1\subsetneq\cdots\subsetneq F_s=V)$$ 
of $I$-graded subspaces in $V$ of dimension
$$\dim F_i= v_1 + \nu_{\leqslant i}, \qquad i=1, \ldots, s.$$
Then, we have
$$\H[v_1,\Lambda_{W/V_1}\,;\,w]=
\{(\bar{x}, \bar{a}) \in \M_s(v,w)\;;\;  a(W) \subseteq F_0,\;x^*(F_i) \subseteq F_i, \; x(F_i) \subseteq F_{i-1}, \;
\forall i \geqslant 1\}.$$
Let $P\subset G(v)$ be the stabilizer of $V_1$ and
$P_\nu\subset P$ be the stabilizer of the flag $F$. 
By \eqref{Lambda}, the canonical map
$$\rho : P \times_{P_{\nu}} \H[v_1,\Lambda_{F/V_1}\,;\,w] \to \H[v_1,\Lambda_\nu\,;\,w]$$
is surjective and proper. It is an isomorphism over the open subset 
$\H[v_1,\Lambda_\nu^\circ\,;\,w],$ since for any
$z$ in this open set there exists a unique flag of type $\nu$ in $V/V_1$ satisfying the
semi-nilpotency condition with respect to $z$. Proposition~\ref{prop:irrheckenu} is a consequence of the following claims
\begin{itemize}
\item[(i)]
$\H[v_1,\Lambda_{F/V_1}\,;\,w]$ is irreducible of dimension $\geqslant d_{v,w}/2 + d_{v_1,w}/2+\dim P_{\nu}$, 
\item[(ii)]
$\H[v_1,\Lambda_\nu^\circ\,;\,w]$ is non-empty.
\end{itemize}

\smallskip

We begin with the dimension estimate in (i).
Fixing isomorphisms $F_j \simeq \k^{v_1+\nu_{\leqslant j}}$ for $j \geqslant 0$ we obtain a closed embedding 
$$\H[v_1,\Lambda_{F/V_1}\,;\,w] \to \prod_{i \geqslant 0} \M_s(v_1 + \nu_{\leqslant i},w), \quad 
(\bar{x}, \bar{a}) \mapsto (\bar{x}|_{F_j}\,,\, \bar{a}|_{F_j})_j.$$
For each $u \in \mathbb{N}^I$ and $n \in \mathbb{N}$ we set
\begin{equation*}
\begin{split}
\H[u,n]=
\{((\bar{x},\bar{a}),(\bar{y},\bar{b})) \in \M_s(u,w) \times \M_s(u+n\delta_i,w)\;;\; \bar{y}|_{\k^u}=\bar{x},\,
 \bar{b}|_{\k^u}=\bar{a},\, y(\k^{u+n\delta_i}) \subseteq \k^u\}.
\end{split}
\end{equation*}
It is an open subset of the zero locus of $(u + n\delta_i) \cdot u$ equations in a vector space of dimension 
$$2u \cdot w + n\delta_i \cdot w + (2u \cdot u - (u,u)) + 
(2n\delta_i \cdot u -(n\delta_i,u)) + n\delta_i \cdot n\delta_i -\langle n\delta_i, n\delta_i \rangle.$$
By \eqref{E:dimquivervar}, each irreducible component of $\H[u,n]$ is of dimension
$$\geqslant \dim P_{u,n} + d_{u,w}/2 + d_{u+n\delta_i,w}/2 $$
where $P_{u,n}$ is the parabolic subgroup in $G(u+n\delta_i)$ of type $(u,n\delta_i)$. 
Since the set $\M_s(u,w) \times \M_s(u + n\delta_i,w)$
is irreducible, every irreducible component of $\H[u,n]$ is of codimension
$$\leqslant d_{u,w}/2 + d_{u+n\delta_i,w}/2 + u \cdot n\delta_i$$
in $\M_s(u,w) \times \M_s(u + n\delta_i,w)$. We have
$$\H[v_1,\Lambda_{F/V_1}\,;\,w]=\bigcap_{j \geqslant 0} \M_s(v_1,w) \times \cdots \times \H(v_1+\nu_{\leqslant j}, \nu_{j+1}) \times \cdots \times \M_s(v,w).$$
Hence every irreducible component of $\H[v_1,\Lambda_{F/V_1}\,;\,w]$ is of codimension
$$\leqslant\sum_{j \geqslant 0} \Big( d_{v_1+\nu_{\leqslant j},w} /2+ d_{v_1+\nu_{\leqslant j+1},w}/2 + (v_1+ \nu_{\leqslant j}) \cdot (\nu_{j+1}\delta_i) \Big)$$
in $\prod_j \M_s(v_1+ \nu_{\leqslant j},w)$ and finally of dimension
\begin{align}\label{E:lower}
\begin{split}
&\geqslant d_{v_1,w}/2+d_{v_1 + l\delta_i,w}/2 + v_1 \cdot v_1 + (v_1+l\delta_i) \cdot (v_1+l\delta_i) -\sum_j (v_1+\nu_{\leqslant j})\cdot (\nu_{j+1}\delta_i)\\
&= d_{v_1,w}/2+ d_{v_1 + l\delta_i,w}/2  + \dim P_{\nu}.
\end{split}
\end{align}

\smallskip

Let us turn to the irreducibility statement in (i). For this, we stratify $\H[v_1,\Lambda_{F/V_1}\,;\,w]$. 
We use Crawley-Boevey's trick and view elements of $\M(v,w)$ as representations of the preprojective 
algebra 
$\tilde\Pi$ in \S\ref{S:CBtrick}. 
The stratification we want is obtained in three steps.

\smallskip

\emph{$\mathbf{Step\ 1 :}$  we stratify $\M_s(v_1,w)$}.
Given $z\in\M_s(v_1,w)$, let 
$V^z_1 \oplus W \subseteq V_1 \oplus W$ be the smallest $I$-graded $z$-stable subspace containing 
$\bigoplus_{j \neq i} V_{1,j} \oplus W$.
We have an exact sequence of $\tilde\Pi$-modules
$$0 \to V^z_1 \oplus W \to V_1 \oplus W \to V_1/V^z_1 \to 0.$$
Note that $V_1/V^z_1$ is a representation of $\tilde\Pi$ which is supported at the vertex $i$. Set
$$\M_s(v_1,w)^{n,i}=\{ z \in \M_s(v_1,w)\;;\; \codim_{V_1} V^z_1 =n\delta_i\}.$$
We have $\M_s(v_1,w)=\bigsqcup_{n\geqslant 0} \M_s(v_1,w)^{n,i}$ with $\M_s(v_1,w)^{0,i}$ being open (and possibly empty).

\smallskip

\begin{lemma}\label{L:vprimeun}Let $z\in \M_s(v_1,w)$ and let $z'$ be any $\tilde\Pi$-module supported at the vertex $i$. 
Then, we have
$\Hom_{\tilde\Pi}(z,z')=\Hom_{\tilde\Pi}(z|_{V_1/V^z_1}\,,\,z').$
\end{lemma}

\smallskip

\begin{proof}
Any $\tilde\Pi$-morphism $z \to z'$ must vanish on $\bigoplus_{j \neq i} V_{1,j} \oplus W$, 
hence on $V^z_1 \oplus W$.
\end{proof}

\smallskip

\emph{$\mathbf{Step\ 2 :}$ we stratify $\M(v_2)$}.
We begin with a general construction. Let $Q$ be any quiver and $M$ a $\Pi$-module. 
Then, there exists a unique minimal submodule $N \subset M$ such that $M/N$ is semi-nilpotent. 
Indeed, if $M/N_1$ and $M/N_2$ are semi-nilpotent, then from the exact sequence
$$0 \to N_1/(N_1 \cap N_2) \to M/(N_1 \cap N_2) \to M/N_1 \to 0$$
and the fact that the subcategory of semi-nilpotent representations of $\Pi$ is closed under subobjects, 
quotients and extensions, it follows that $M/(N_1 \cap N_2)$ is also semi-nilpotent. We denote the minimal submodule $N$ above 
by $\corad(M)$ and call it the \emph{coradical} of $M$. The next statement is obvious.

\smallskip

\begin{lemma}\label{L:corad}Let $M,$ $L$ be $\Pi$-modules, with $L$ being semi-nilpotent. Then
\begin{itemize}[leftmargin=8mm]
\item[$\mathrm{(a)}$]
 $\corad(\corad(M))=\corad(M)$,
\item[$\mathrm{(b)}$]
 $\Hom_{\Pi}(M,L)=\Hom_{\Pi}(M/\corad(M),L)$.
\end{itemize}
In particular, we have $\Hom_{\Pi}(\corad(M),L)=\{0\}$.
\qed
\end{lemma}

\smallskip

Let us now consider the special case of the one vertex quiver with $q$ loops with $q >1$. 
Recall that $v_2=l\delta_i$. For any integer $n \in \mathbb{N}$ with $n \leqslant l$ we
set
$$\M(v_2)^n=\{\bar{x} \in \M(v_2)\;;\; \codim(\corad(\bar{x}))=n\}.$$
We have a stratification 
$$\M(v_2)=\bigsqcup_{n=0}^l \M(v_2)^n$$
with $\M(v_2)^0$ being a non-empty open subset. 
Because $\M(v_2)$ is irreducible by Lemma~\ref{L:muirr} and of dimension
$(2q-1)l^2+1$, the same holds for $\M(v_2)^0$.

\smallskip

\emph{$\mathbf{Step\ 3 :}$ we stratify $\H[v_1,\Lambda_{F/V_1}\,;\,w]$}.
For each $n_1 \leqslant n_2$ we put
$$\M_s(v_1,w)^{n_1,n_2}=\{z=(\bar{x},\bar{a}) \in \M_s(v_1,w)^{n_2,i}\;;\; 
\bar{x}|_{V_1/V^z_1} \in \M(n_2)^{n_1}\}.$$
Under the morphism which assigns to $z$ its associated graded
$$\H[v_1,\Lambda_{F/V_1}\,;\,w] \to \M_s(v_1,w) \times \M(v_2),$$
the stratification 
$\M_s(v_1,w)=\bigsqcup_{\,n_1\leqslant n_2} \M_s(v_1,w)^{n_1,n_2}$ pulls back to a stratification
\begin{equation}\label{E:strataH}
\H[v_1,\Lambda_{F/V_1}\,;\,w] =\bigsqcup_{n_1\leqslant n_2} \H[v_1,\Lambda_{F/V_1}\,;\,w]^{n_1,n_2}.
\end{equation}

\smallskip

The next step is to compute the dimension of the strata in \eqref{E:strataH}. 
We begin by the following general result.

\smallskip

\begin{lemma}\label{L:CBext} Let $v_1,$ $v_2$ be dimension vectors of a quiver $Q$ with $v=v_1+v_2$. 
Let $\rho$ be the obvious map
$\{\bar{x} \in \M(v)\;;\; \bar{x}(V_1) \subseteq V_1\}\to\M(v_1) \times \M(v_2).$
\begin{itemize}[leftmargin=8mm]
\item[$\mathrm{(a)}$]
The fiber of $\rho$
over  $(\bar x_1,\bar x_2)$ is an affine space of dimension 
$$v_1 \cdot v_2-(v_2\,,\,v_1)_Q + \dim\Hom_\Pi(\bar x_1,\bar x_2).$$
\item[$\mathrm{(b)}$]
The restriction of $\rho$ to the preimage of the open set 
$$\{(\bar{x}_1, \bar{x}_2) \in \M(v_1) \times \M(v_2)\;;\; \Hom_\Pi(\bar{x}_1, \bar{x}_2)=0\}$$
is an affine space bundle.
\end{itemize}
\end{lemma}

\smallskip

\begin{proof} The first statement is \cite[lem~5.1]{CB03} combined with \eqref{E:prepro1} and \eqref{E:prepro}. 
Let us give more details. Fix a vector space splitting $V=V_1 \oplus V_2$. 
The fiber of $\rho$ over $(\bar x_1,\bar x_2)$ is then identified with the kernel of the linear map
\begin{equation*}
\begin{split}
\gamma~:\,\bigoplus_{h \in \bar{\Omega}}\Hom\big(V_{2,h'}\,, V_{1,h''}\big) &\to \bigoplus_i \Hom\big(V_{2,i}\,, V_{1,i}\big)\\
(u_h\,, u_{h^*}) &\mapsto \sum_{h} (x_{1,h}\,u_{h^*} + u_h\,x_{2,h^*}-x_{1,h^*}\,u_{h} - u_{h^*}\,x_{2,h}).
\end{split}
\end{equation*}
By \cite[\S 1]{CBex}, see also \S \ref{sec:6.1}, any finite-dimensional
representation $N$ of $\Pi$ of dimension $v$ has a canonical projective resolution
$$\xymatrix{ 0\ar[r] & P'' \ar[r]^-{f} & P' \ar[r]^-{g} & P \ar[r] & N \ar[r] & 0}$$
with
$$P'=\bigoplus_{h \in \bar{\Omega}} P_{h''} \otimes V_{h'}\,,\quad P =P''=\bigoplus_i P_i \otimes V_{i}.$$
Considering the above resolution for $N=\bar x_2$, we have
$$\Hom_\Pi(P',\bar x_1)=\bigoplus_{h \in \bar{\Omega}}\Hom\big(V_{2,h'}\,, V_{1,h''}\big) , \quad 
\Hom_\Pi(P'',\bar x_1)=\bigoplus_i \Hom\big(V_{2,i}\,, V_{1,i}\big).$$
Hence $\gamma$ is identified with the map $\Hom_\Pi(P',\bar x_1) \to \Hom_\Pi(P'',\bar x_1)$ induced by $f$.
So there is a short exact sequence
$$\xymatrix{0 \ar[r] & \Hom_\Pi(\bar x_2,\bar x_1) \ar[r] & \Hom_\Pi(P,\bar x_1) \ar[r] & \Ker(\gamma) \ar[r] & \Ext^1_\Pi(\bar x_2,\bar x_1) \ar[r] & 0}$$
Since $\dim\Hom_\Pi(P,\bar x_1)=v_1 \cdot v_2$ we deduce that 
$$\dim\Ker(\gamma)=\dim\Ext^1_\Pi(\bar x_2,\bar x_1) - \dim\Hom_\Pi(\bar x_2,\bar x_1) + v_1 \cdot v_2.$$ 
The first part of the proposition is now a consequence of \eqref{E:prepro1} and \eqref{E:prepro}. The second statement
follows from the fact that $U$ is precisely the set of points of $\M(v_1) \times \M(v_2)$ for which the vector bundle morphism $\gamma$ is 
of maximal rank. 
\end{proof}

\smallskip

Now, we can prove the following lemma.

\smallskip

\begin{lemma}
The stratum $\H[v_1,\Lambda_{F/V_1}\,;\,w]^{n_1,n_2}$ is either empty or of dimension equal to
$$d(n_1,n_2):=d_{v_1,w} /2+ d_{v,w}/2 + \dim P_{\nu} -w \cdot n_2 +
(n_2\,,\, v_1-n_2) + (n_1\,,\,n_2-n_1/2) + \delta_{n_1<n_2}.$$
\end{lemma}

\begin{proof} Let us fix $n_1 \leqslant n_2$. 
We refine the flag 
$$F=(V_1=F_0 \subset F_1\subset \cdots\subset F_s=V)$$
by adding a pair of subspaces $F_{-2} \subseteq F_{-1} \subseteq V_1$ such that 
$\codim_{V_1}(F_{-j})=n_j\delta_i$ for $j=1,2$.
Then, we set
\begin{equation*}
\H_{F}=\{z\in \H[v_1,\Lambda_{F/V_1}\,;\,w]^{n_1,n_2}\;;\; V^{z_1}_1=F_{-2},\; \corad(V_1/F_{-2})=F_{-1}/F_{-2}\}
\end{equation*}
where $z_1=z|_{V_1\oplus W}$.
A pair $z=(\bar{x},\bar{a})$ in $\M_s(v,w)$ belongs to $\H_F$ if and only if it stabilizes the 
(refined) flag $W$ and the following 
conditions are satisfied
\begin{equation}\label{E:defHbullet}
\begin{split}
z|_{F_{-2}\oplus W} &\in \M_s(v_1-n_2\delta_i,w)^0, \quad  
\bar{x}|_{ F_{-1}/F_{-2}} \in \M(n_2\delta_i-n_1\delta_i)^0,\\
\bar{x}|_{F_0/F_{-1}} &\in \Lambda(n_1\delta_i), \qquad\qquad\qquad
x|_{F_j/F_{j-1}}=0,\; \;\forall j>1.
\end{split}
\end{equation}
Let $P_F\subset G(v_1)$ be the stabilizer of the flag $F_{-2} \subseteq F_{-1}\subseteq V_1$. We have 
$$\H[v_1,\Lambda_{F/V_1}\,;\,w]^{n_1,n_2}=G(v_1) \times_{P_F} \H_{F}.$$
Let $\H'_F$ be defined as $\H_{F}$ but without the stability condition for $z$ :
it is the set of elements $z=(\bar{x},\bar{a})$ in $\M(v,w)$ satisfying \eqref{E:defHbullet}. 
Finally, consider the flag 
$$F/F_{-1}=(V_1/F_{-1}=F_0/F_{-1}\subseteq F_1/F_{-1}\subseteq\cdots\subset F_s/F_{-1}=V/F_{-1})$$ 
and set
$$\Lambda'_{F/F_{-1}}=
\{\bar{x} \in \Lambda(n_1\delta_i+l\delta_i)\;;\; \bar{x}(F_0/F_{-1}) \subseteq F_0/F_{-1}\,,\,  
x(F_j/F_{-1}) \subseteq F_{j-1}/F_{-1}\,,\,\forall j >1\}.$$ 
By Lemmas~\ref{L:vprimeun} and \ref{L:corad}, for all $z=(\bar x,\bar a)$ in $\H'_{W}$ we have
$$\Hom(z|_{F_{-2}}\,,\, \bar{x}|_{V/F_{-2}})= \Hom(\bar{x}|_{F_{-1}/F_{-2}}\,,\, \bar{x}|_{V/F_{-1}})=\{0\}.$$
Therefore by Lemma~\ref{L:CBext} the map
$$p_{n_1,n_2} : \H'_F \to \M_s(v_1-n_2\delta_i,w)^0 \times \M(n_2\delta_i-n_1\delta_i)^0 \times 
\Lambda'_{F/F_{-1}}$$
is an affine fibration of rank equal to
\begin{equation}\label{E:equa1}
\begin{split}
\dim(p_{n_1,n_2})=&- \left(l\delta_i+n_1\delta_i\,,\,(v_1-n_1\delta_i,1)\right)_{\tilde{Q}}-
(n_2\delta_i-n_1\delta_i\,,\, (v_1-n_2\delta_i, 1))_{\tilde{Q}} \\
&+ (v_1-n_1\delta_i) \cdot (l\delta_i+n_1\delta_i) + (v_1-n_2\delta_i) \cdot (n_2\delta_i-n_1\delta_i)\\
=&-(l\delta_i+n_1\delta_i\,,\, v_1-n_1\delta_i) - (n_2\delta_i-n_1\delta_i\,,\, v_1-n_2\delta_i) +w \cdot (l+n_2\delta_i) \\
&+ (v_1-n_1\delta_i) \cdot (l\delta_i+n_1\delta_i )+ (v_1-n_2\delta_i) \cdot (n_2\delta_i-n_1\delta_i),
\end{split}
\end{equation}
see (\ref{E:eulerformqtilde}).
On the other hand, we have
\begin{align}\label{E:equa2}
\begin{split}
&\dim\M_s(v_1-n_2\delta_i\,,\,w)^0=-(v_1-n_2\delta_i\,,\, v_1-n_2\delta_i) + 2w \cdot (v_1-n_2\delta_i)\\ 
&\hspace{4.2cm} + (v_1-n_2\delta_i) \cdot (v_1-n_2\delta_i),\\
&\dim\M(n_2\delta_i-n_1\delta_i)^0
=\begin{cases}(n_2-n_1)^2-(n_2\delta_i-n_1\delta_i\,,\,n_2\delta_i-n_1\delta_i)  +1\!&\!\text{if}\, n_2 > n_1\\ 0&\text{if}\, n_2=n_1\end{cases}
\end{split}
\end{align}

\smallskip

\begin{lemma}\label{L:lambdaprime} $\Lambda'_{F/F_{-1}}$ is pure of dimension 
$q(n_1+l)^2 -\sum_{a<b} \nu_a\nu_b - n_1l$.
\end{lemma}

\smallskip

\begin{proof}
There is a natural projection map $\Lambda'_{F/F_{-1}} \to \Lambda(n_1\delta_i)$. 
Thus, the partition $\Lambda(n_1\delta_i)=\bigsqcup_{\mu} \Lambda_{\mu}^{\circ}$ yields a partition 
$$\Lambda'_{F/F_{-1}}=\bigsqcup_{\mu \,\vDash\, n_1} \Lambda'_{F/F_{-1},\,\mu}.$$
Fix a composition $\mu$ 
and a flag of subspaces $F_0/F_{-1}=F'_0 \supseteq F'_{-1} \supseteq \cdots $ of type 
$\mu$ in $F_0/F_{-1}$. We have
$$\Lambda'_{F/F_{-1},\,\mu}=GL(n_1) \times_{P_{\mu}} \Lambda_{F''}$$
where $P_{\mu}$ is the stabilizer of  $F'_{\leqslant 0}$ in $GL(n_1) \subset GL(n_1+l)$ and $F''$ is the flag 
$$\cdots \subseteq F'_{-1} \subseteq F'_0 \subseteq F_1 \subseteq \cdots.$$
 It follows that 
$$\dim\Lambda'_{F/F_{-1},\,\mu} = \dim\Lambda_{F''} + \dim(GL(n_1)/P_{\mu}).$$
Recall that $F''$ is a flag of type $\gamma=(\mu,\nu)$, the concatenation of $\mu$ and $\nu$.
From the equality $$\dim\Lambda_{F''}=\dim\Lambda(n_1\delta_i+l\delta_i) - \sum_{a<b}\gamma_a\gamma_b,$$
we deduce that 
$$\dim\Lambda'_{F/F_{-1},\,\mu}=q(n_1+l)^2 -\sum_{a<b} \nu_a\nu_b - n_1l.$$ The lemma is proved.
\end{proof}

\smallskip

We are finally ready to compute the dimension of the stratum $\H[v_1,\Lambda_{F/V_1}\,;\,w]^{n_1,n_2}$.
Since $\H'_F$ is of pure dimension by 
Lemma~\ref{L:lambdaprime}, the subset $\H_F$ is either empty or of the same dimension. 
Thus $\H[v_1,\Lambda_{F/V_1}\,;\,w]^{n_1,n_2}$ is either empty or of dimension equal to
$$\dim\H'_F + \dim(G(v_{1})/P_F) = 
\dim\H'_F+(v_{1}-n_2\delta_i) \cdot n_2\delta_i + (n_2\delta_i-n_1\delta_i) \cdot n_1\delta_i$$
So a straightforward computation using \eqref{E:equa1}, \eqref{E:equa2} and 
Lemma~\ref{L:lambdaprime} shows
that the stratum $\H[v_1,\Lambda_{F/V_1}\,;\,w]^{n_1,n_2}$ is either empty or of dimension equal to
$d(n_1,n_2).$
\end{proof}

\smallskip

Now, we can finish the proof of Proposition~\ref{prop:irrheckenu}.
For all $n_1$, $n_2$ we have
$$d(n_1,n_2) \leqslant (d_{v_1,w} + d_{v,w})/2 + \dim(P_{\nu}),$$
with equality if and only if $n_1=n_2=0$.
Indeed, we have $(n_2\,,\,v_1-n_2) <0$ as soon as $0 < n_2 < v_{1,i}$.
Further, if $n_2=v_{1,i}$ then $w \cdot n_2 +(n_2\,,\,v_1-n_2) <0$ 
as soon as $v_{1,j} >0$ for some $j$ adjacent to $i$ or $w_i>0$. 
Finally, if $v_{1,j}=0$ for all adjacent $j$ and $w_i=0$ then 
$\M_s(v_1,w)$ is empty. 

\smallskip

Hence the set $\H[v_1,\Lambda_{F/V_1}\,;\,w]^{0,0}$ is the only stratum of dimension 
$$\geqslant d_{v_1,w} /2+ d_{v,w}/2 + \dim P_{\nu}.$$ Combining this with the lower dimension estimates for 
$\H[v_1,\Lambda_{F/V_1}\,;\,w]$ in
\eqref{E:lower}, we deduce that $\H[v_1,\Lambda_{F/V_1}\,;\,w]^{0,0}$ is non-empty if 
$\H[v_1,\Lambda_{F/V_1}\,;\,w]$ is non-empty. 
Observe also 
that $\H[v_1,\Lambda_{F/V_1}\,;\,w]^{0,0}$ is irreducible. 
These facts imply that $\H[v_1,\Lambda_{F/V_1}\,;\,w]$ is irreducible of dimension
equal to $$d_{v_1,w}/2 + d_{v,w}/2 + \dim P_{\nu}.$$ 
The fact that $\H[v_1,\Lambda^\circ_{\nu}\,;\,w]$ is non-empty follows from the
definition of $\H[v_1,\Lambda_{F/V_1}\,;\,w]^{0,0}$. Proposition~\ref{prop:irrheckenu} is proved.
\end{proof}

\smallskip

\subsubsection{Convolution of one vertex Hecke correspondences}

Let us now state a useful corollary of the previous proposition concerning compositions of Hecke correspondences.
Let $\star$ denote the convolution product of correspondences as in \cite{CG}. 
Set $v=v_1+v_2$ with $v_2=l\delta_i$ and consider a fixed composition $\nu=(\nu_1,\nu_2,\dots,\nu_r)$ of $v_2$.

\smallskip

\begin{proposition}\label{prop:comp}
Let $q>1$. The following relation holds in 
$H_{d}^T\big(\frakM(v,w)\times\frakM(v_1,w)\big)$
$$[\,\frakh[v_1,\Lambda_\nu\,;\,w]\,]=
[\,\frakh[v-\nu_1,\Lambda_{(\nu_1)}\,;\,w]\,]\star\cdots\star[\,\frakh[v_1,\Lambda_{(\nu_r)}\,;\,w]\,].$$

\end{proposition}

\smallskip

\begin{proof} To simplify we assume that $r=2$.
We use the same notation as in the proof of Proposition \ref{prop:irrheckenu}.
In particular, we have the parabolic subgroups $P_\nu\subseteq P$ in $G(v)$, where 
$P_\nu$ fixes a flag $F=(V_1\subsetneq F_1\subsetneq V)$ of type $(v_1,v_1+\nu_2,v)$.
From Proposition \ref{prop:LAG}  we deduce that the varieties
$$\frakh[v-\nu_1,\Lambda_{(\nu_1)}\,;\,w]\times\frakM(v_1,w),\quad
\frakM(v,w)\times\frakh[v_1,\Lambda_{(\nu_2)}\,;\,w]$$ 
are local complete intersections and are
dimensionally transverse in $$\frakM(v,w)\times\frakM(v_1+\nu_2,w)\times\frakM(v_1,w),$$ with intersection equal to
the categorical quotient
$\H(v_1,\Lambda_\nu\,;\,w)/\!\!/P_\nu$.  We deduce that
\begin{align}\label{intersection}
[\frakh[v-\nu_1,\Lambda_{(\nu_1)}\,;\,w]\times\frakM(v_1,w)]\,\cap\,
[\frakM(v,w)\times\frakh[v_1,\Lambda_{(\nu_2)}\,;\,w]]=
[\H[v_1,\Lambda_\nu\,;\,w]/\!\!/P_\nu].
\end{align}
Since $\frakh[v_1,\Lambda_\nu\,;\,w]=\H[v_1,\Lambda_\nu\,;\,w]/\!\!/P$,
the varieties $\H[v_1,\Lambda_\nu\,;\,w]/\!\!/P_\nu$ and $\H[v_1,\Lambda_\nu\,;\,w]/\!\!/P$ are both irreducible
by Proposition \ref{prop:irrheckenu}.
Further, they have the same dimension by \eqref{intersection} and Proposition \ref{prop:irrheckenu}.
Since for any
$z$ in the open subset 
$\H[v_1,\Lambda_\nu^\circ\,;\,w],$ there exists a unique flag of type $\nu$ in $V/V_1$ satisfying the
semi-nilpotency condition with respect to $z$,
the obvious map 
$$f\,:\,\H[v_1,F/V_1\,;\,w]/\!\!/P_\nu\to \H[v_1,\Lambda_\nu\,;\,w]/\!\!/P=\frakh[v_1,\Lambda_{\nu}\,;\,w]$$ 
is generically one to one because the set $\Lambda_\nu^\circ$ is dense in $\Lambda_\nu$ by
Proposition \ref{prop:geomirrlambda}. So the proposition
 follows from definition of the convolution product which yields
$$[\frakh[v_1+\nu_2,\Lambda_{(\nu_1)}\,;\,w]]\star
[\frakh[v_1,\Lambda_{(\nu_2)}\,;\,w]]=f_*([\H[v_1,\Lambda_\nu\,;\,w]/\!\!/P_\nu]).$$
\end{proof}

\medskip

\section{The homology of the semi-nilpotent variety}\label{sec:irr}\hfill

In this section $\k$ is any algebraically closed field, except in subsections \S\S\ref{sec:LL}, \ref{sec:TF}
where we assume that $\k=\bbC$.

\subsection{The Lagrangian quiver variety}\label{sec:L}\hfill

\subsubsection{Definition}

For $\flat=0$ or $1$, and for any dimension vectors $v,$ $w$, we set
\begin{align}\label{Z}
\L_s^\flat(v,w)=\R_s(v,w)\cap\Big(\Lambda^\flat(v)\times\{0\}\times\Hom_I(v,w)\Big).
\end{align}
The \emph{Lagrangian quiver variety} is the geometric quotient
\begin{align}\label{ZZ}\frakL^\flat(v,w)=\L_s^\flat(v,w)/\!\!/G(v).\end{align}
The closed embedding $\Lambda^1(v)\subseteq\Lambda^0(v)$ yields closed embeddings
\begin{align*}
&\L_s^1(v,w) \subseteq \L_s^0(v,w) \subset \M_s(v,w),\\
&\frakL^1(v,w)\subseteq \frakL^0(v,w) \subseteq \frakM(v,w).
\end{align*}
We use another description of $\frakL^\flat(v,w)$.
Consider the $\bbG_m$-action $\circ$ on $\frakM(v,w)$ given by 
\begin{align}\label{action1}\begin{split}
t\circ(\bar x,\bar a)&=(t^{-1}x,x^*,t^{-1}a,a^*).\end{split}\end{align}
Let $\frakM(v,w)^\circ$ be the fixed points locus 
and $\frakL(v,w)^\circ$ be the attracting variety. The following is proved in \cite[prop. 3.1, 3.2]{BSV}.

\smallskip

\begin{proposition}\label{prop:TL}\hfill
\begin{itemize}[leftmargin=8mm]
\item[$\mathrm{(a)}$] 
$\frakL^0(v,w)=\frakL(v,w)^\circ$,
\item[$\mathrm{(b)}$]
there is a $\bbG_m$-action $\bullet$ on $\frakM(v,w)$ such that $\frakL^1(v,w)=\frakL^1(v,w)^\bullet$,
\item[$\mathrm{(c)}$] 
$\frakL^\flat(v,w)$ is a closed Lagrangian subvariety of $\frakM(v,w)$ for $\flat=0,1$.
\end{itemize}
\qed
\end{proposition}

\smallskip

The $\bullet$-action can be made explicit as in  \cite[\S 3.3.2]{BSV}, but we won't need it here. Observe that all the 
irreducible components of $\frakL^1(v,w)$ are also irreducible components of $\frakL^0(v,w)$.

\smallskip

The $\circ$-action is associated with the cocharacter 
$\gamma=\xi^{-1}$ of $T\times G(w)$ in \eqref{action0}.
Given a cocharacter $\rho$ of $G(v)$, we define the subsets
$\frakM[\rho]$ and $\frakL[\rho]$ of $\frakM(v,w)$ 
as in \S \ref{sec:BBQ}. Hence,
\begin{align}\label{ATT}
\frakL^0(v,w)=\bigsqcup_\rho\frakL[\rho],\quad
\frakL[\rho]=\{z\in\frakM(v,w)\,;\,\lim_{t\to 0}t\circ z\in\frakM[\rho]\},
\end{align}
and the closure of $\frakL[\rho]$ is a sum of irreducible components of $\frakL^0(v,w)$ 
by \cite[prop. 3.5]{BSV}.

 \smallskip
 
 \subsubsection{The case of one vertex quivers}
Assume that $Q=\J(i,q)$ with $q>1$.
Then, we have $\Lambda^1=\Lambda^0$.
We abbreviate $\frakL=\frakL^\flat$. 
Given a dimension vector $v=l\,\delta_i$, we fix a composition $\nu=(\nu_1,\dots,\nu_r)$ 
of $v$ and an increasing flag $W$ of type $\nu$.
Let $\rho_\nu$ be the cocharacter of $G(v)$ which preserves the flag $W$ and has the weight $1-p$ on
$W_p/W_{p-1}$ for each $p=1,\dots,r.$
We write
\begin{align}\label{rho-nu}\rho_\nu(t)=\diag\big(1^{(\nu_1)}\,,\,(t^{-1})^{(\nu_2)}\,,\dots,\,(t^{1-r})^{(\nu_r)}\big).
\end{align}
Assume that $w=v$. By Propositions \ref{prop:1.10} and \ref{prop:TL}
the variety $\frakL(v,v)$ is the union of
the Zariski closure of the attracting sets for the $\circ$-action.
Each of these closed set is a finite union of irreducible Lagrangian subvareties of $\frakM(v,v)$.

\smallskip

\begin{lemma}\label{claim:2.10} Assume that $w=v=l\delta_i$. For each cocharacter $\rho$ of $G(v)$ we have
$$\frakL[\rho]\neq\emptyset\iff\exists\,\nu\vDash  v \ 
\text{such\ that}\ \rho=\rho_\nu$$
\end{lemma}

\begin{proof}
Fix any cocharacter $\rho$ of $G(v)$.
Up to the conjugation by an element of $G(v)$, we may choose some integers $v'_k\geqslant 0$ such that
$$\rho(t)=\diag\big(\dots\,,\,t^{(v'_{0})}\,,\,1^{(v'_{1})}\,,\,(t^{-1})^{(v'_{2})}\,,\,\dots\big).$$
Set $v'=(v'_k)$. Set also $w'=(w'_k)$ with $w'_k=l\,\delta_{k,1}$. 
The \emph{local quiver variety} $\frakM(v',w')^\#$ 
associated with the quiver $Q$ and the dimension vectors $v'$ and $w'$
is the geometric quotient by the action of the group $G(v)\times G(v)$ of the set of 
semistable tuples of linear maps $(\bar x,\bar a)$ such that
\begin{itemize}
\item $x\in\Hom_\bbZ(V,V[-1])$ and $x^* \in\End_\bbZ(V)$,
\item $a=a_1\in\Hom(W_1,V_{0})$ and $a^*=a^*_1\in\Hom(V_1,W_1),$
\item $V_k\subseteq V$ is the kernel of $\rho(t)-t^{1-k}$,
\item $W_1=W.$
\end{itemize}
It is empty unless there is an integer $r$ and a composition $\nu\vDash v$ as above such that
$v'_k=\nu_k$ if $k\in[1,r]$ and $v'_k=0$ else.
Further,  we have an isomorphism $\frakM(v',w')^\#\simeq\frakM[\rho].$ 
This yields the direct part of the equivalence.

\smallskip

To prove the reverse implication, it remains to check that $\frakM[\rho_\nu]\neq\emptyset$ for all $\nu$.
The assignment $\bar x\mapsto[\bar x,0,1]$ 
yields an open immersion of $G(v)$-varieties 
\begin{align}\label{alpha1}\alpha:\Lambda(v)\to\frakL(v,v).\end{align}
We claim that we have
\begin{align}\label{subset1}
\Lambda^\circ_{\nu}=\alpha^{-1}(\frakL[\rho_\nu]),\end{align}
proving the lemma.
In order to prove \eqref{subset1}, we fix an increasing flag $W$ of $I$-graded vector spaces in $V$ of type $\nu$
and we consider the cocharacter $\rho_\nu$ as above.
For each element $\bar x\in\Lambda(v)$ we have
\begin{align*}
\bar x\in\Lambda_W
\Longrightarrow\,\,&\exists\lim_{t\to 0}t\circ\rho_\nu(t)\cdot(\bar x,0,1)\in\M(v,v).
\end{align*}
Let $(\bar y,0,b^*)$ denote the limit above.
We must prove that 
$$\bar x\in\Lambda^\circ_W\Rightarrow(\bar y,0,b^*)\in\M[\rho_\nu].$$
It is enough to check that if $\bar x\in\Lambda^\circ_W$ then $(\bar y,0,b^*)$ is semistable.
Since $\bar x\in\Lambda_W$, taking the associated graded of the representation $(\bar x, 0, 1)$ relative to the flag $W$
yields the representation $(\bar y,0,b^*)$.
Hence, the later can be viewed as a representation in $\M^\#(v,w)$ 
where the dimension vectors $v$, $w$ are given by
$$i\in[1,r]\Rightarrow v_i=\nu_i,\ i\notin[1,r]\Rightarrow v_i=0,\
w=l\,\delta_{1}.$$
Since $\bar x\in\Lambda^\circ_W$,  the subspace $W_{p}\subseteq V$ is the annihilator of $A_{\bar x}(\geqslant\!p)$. 
Hence any nonzero subspace in $V\oplus W$ which is stable by $\bar y$
must have a nonzero component in $W$.
Hence $(\bar y,0,b^*)$ is semistable.
\end{proof}

\smallskip

\subsection{The homology of the Lagrangian quiver variety} \hfill\\




For each dimension vector $w$ we fix a closed subgroup $H(w)\subset G(w)$.

\begin{proposition}\label{prop:L} 
Let $\k=\bbC$ and $\flat$ be either 0 or 1. If $X=\frakM(v,w)$ or $\frakL^\flat(v,w)$ then
\hfill
\begin{itemize}[leftmargin=8mm]
\item[$\mathrm{(a)}$] $H_*^{T\times H(w)}(X)$ is pure and even,
\item[$\mathrm{(b)}$] the $T\times H(w)$-variety $X$ is equivariantly formal,
\item[$\mathrm{(c)}$] $\cl : \A_*^{T\times H(w)}(X)\to H_{2*}^{T\times H(w)}(X)$ is an isomorphism.
\end{itemize}
\end{proposition}

\smallskip

\begin{proof}
The proof is similar to the proof of \cite[thm.~7.3.5]{N01}, replacing equivariant K-theory by equivariant Borel-Moore homology or Chow groups.
Since our setting differs slightly from loc. cit., let us recall briefly the main arguments.

\smallskip

First, recall the Bialynicki-Birula decomposition. 
Let $X$ be a $\k$-variety with a $\bbG_m$-action which embeds equivariantly in a projective space with a diagonalizable 
$\bbG_m$-action. 
Assume that the $\bbG_m$-fixed point locus $X^{\bbG_m}$ is contained
in the regular locus of $X$. Then we have a partition $X=\bigsqcup_{s}X_{s}$ into locally closed subsets with affine space bundles 
$p_s\,:\,X_s\to F_s$ such that the $F_s$'s are the connected components of $X^{\bbG_m}$.
Further, there is a filtration $(X_{\leqslant s})$ of $X$ and an ordering of the components $F_s$ such that 
$X_s=X_{\leqslant s}\setminus X_{<s}$. This filtration depends on the choice of the equivariant embedding of $X$ 
in a projective space with a diagonalizable $\bbG_m$-action.

\smallskip

First, we prove the proposition for $X=\frakM(v,w)$.

\smallskip

Consider the actions of $\bbG_m\times\bbG_m$ and $\bbG_m$ on $\frakM(v,w)$ given by 
\begin{align}\label{action2}\begin{split}
(t,t^*)\triangledown(\bar x,\bar a)&=(tx,t^*x^*,ta,t^*a^*),\\ 
t\blacktriangledown(\bar x,\bar a)&=(t^a,t^b)\,\triangledown(\bar x,\bar a)\end{split}\end{align}
for some positive integers $a,$ $b$.
One checks (for instance using the known generators of $\mathrm{k}[\frakM_0(v,w)]$) that the fixed points locus $\frakM_0(v,w)^\triangledown$ is $\{0\}$.
Thus $\frakM(v,w)^\triangledown$ is smooth and projective.
Let $\frakM_1,\dots,\frakM_N$ be its connected components.
Choose $a$, $b$ such that the fixed points locus $\frakM(v,w)^\blacktriangledown$  coincides with $\frakM(v,w)^\triangledown.$
Then, the Bialynicki-Birula decomposition yields a partition of
$\frakM(v,w)$ by smooth locally-closed stable subvarieties $\frakL_1,\frakL_2,\dots,\frakL_N$ 
such that $$\frakL_s=\{z\in\frakM(v,w)\,;\,\lim_{t\to 0}t\blacktriangledown z\in \frakM_s\},$$
and it also yields equivariant affine space bundles $p_s:\frakL_s\to \frakM_s$ such that
$\frakL_1\sqcup\frakL_2\sqcup\cdots\sqcup\frakL_s$ is closed for each $s=1,\dots, N$.
We say that $\frakL_1,\dots,\frakL_N$ form an \emph{$\alpha$-partition} of $\frakM(v,w)$.

\smallskip

By \cite[lem.~7.1.3, 7.1.4]{N01} we must check that the properties (a), (b), (c) 
above hold for the $T\times H(w)$-varieties $\frakM_1,\frakM_2,\dots,\frakM_N$.
The proof in loc.~cit.~is done for equivariant algebraic K-theory and equivariant topological K-theory, but it applies also to equivariant Chow groups
and equivariant Borel-Moore homology.

\smallskip

To prove this we apply \cite[prop.~7.2.1]{N01}. For each $s=1,\dots,N$ the $T\times H(w)$-variety $\frakM_s$ is smooth and projective.
Hence, it is enough to prove that the diagonal $\Delta_s$ of $\frakM_s\times \frakM_s$ satisfies the Kunneth property.
Since $\frakM_s$ is a particular case of graded quiver varieties, 
this follows from a decomposition of the diagonal $\Delta_s$ as in \cite[\S 7.3]{N01}.

\smallskip

Now, we prove the proposition for $X=\frakL^0(v,w)$.
By \cite[prop.~3.3]{BSV}, the Bialynicki-Birula decomposition yields an $\alpha$-partition 
$\frakL^0(v,w)=\bigsqcup_\kappa X_\kappa$ 
with affine space bundles
$X_\kappa\to V_\kappa$ and $\alpha$-partitions $V_\kappa=\bigsqcup_s U_{\kappa,s},$ 
where $s$ runs over a subset $\pi_\kappa$ of $\{1,2,\dots, N\},$
with affine space bundles $U_{\kappa,s}\to F_{s}.$
Applying \cite[lem.~7.1.3, 7.1.4]{N01} successively to the $\alpha$-partitions $V_\kappa=\bigsqcup_s U_{\kappa,s}$ 
and $\frakL^0(v,w)=\bigsqcup_\kappa X_\kappa$,
the properties (a), (b), (c) in the proposition follows from  the analogous properties for the $T\times H(w)$-varieties 
$\frakM_1,\frakM_2,\dots,\frakM_N$,
that we have already checked.

\smallskip

The proof in the case of $X=\frakL^1(v,w)$ is exactly the same, using \cite[prop. 3.4]{BSV}. \end{proof}

\smallskip

\subsection{Purity}\label{sec:LL}\hfill\\

In this section we prove Theorem \ref{thm:1.1}.
In particular, we assume that $\k=\bbC$.
Part (a) is proved in \cite{B14} in the case $\flat=1$.
The case of $\flat=0$ is similar.
Now, we concentrate on the other claims.
Fix any dimension vectors $v$ and $w$ with $v\leqslant w$, i.e.,
with $w-v\in\bbN^I$.
Then, let $\Mon_I(v,w)$ be the set of injective maps in $\Hom_I(v,w)$.
Since the stability condition is open, by \eqref{Z} we have open immersions
\begin{align}\label{diag4}
\xymatrix{\Lambda^\flat(v)\times\Mon_I(v,w)\ar@{^{(}->}[r]^-\alpha\ar[r]\,\,&\L^\flat_{s}(v,w)\,\ar@{^{(}->}[r]^-\beta&\Lambda^\flat(v)\times\Hom_I(v,w)
}
\end{align}
such that $\beta(\bar x,\bar a)=(\bar x,a^*)$ and
$\alpha(\bar x, a^*)=(\bar x, 0,a^*)$.
To simplify the notation we write $G=G(v)$ and $H=G(w).$

\smallskip

The group $T\times G$ acts on $\Lambda^\flat(v)$ as in \S \ref{sec:sn}, hence
$T\times G\times H$ acts via the projection 
\begin{align*}\rho_1\,:\,T\times G\times H\to T\times G.\end{align*}
The group $G\times H$ acts on $\Hom_I(v,w)$ by
$(g,h)\cdot(a^*)=h\,a^*g^{-1},$
the group $T$  so that
the factor $(\bbG_m)^\Omega$ in \eqref{T} acts trivially and the factor 
$\bbG_{m}$ acts by dilatation.
We equip the product $\Lambda^\flat(v)\times\Hom_I(v,w)$ with the 
diagonal $T\times G\times H$-action.
The group $T\times H$ acts on $\frakL^\flat(v,w)$ as in \S \ref{sec:QV}.
We define the following maps :

\begin{itemize}

\item{} The functoriality with respect to $\rho_1$ is an $M^*_{T\times G}$-linear map
$$\rho_1^\sharp\,:\,M_*^{T\times G}(\Lambda^\flat(v))\to M_*^{T\times G\times H}(\Lambda^\flat(v)).$$

\item{} The pullback by the first projection
$p_1\,:\,\Lambda^\flat(v)\times\Hom_I(v,w)\to\Lambda^\flat(v)$
gives an $M^*_{T\times G\times H}$-linear map 
$$p_1^*\,:\,
M_*^{T\times G\times H}(\Lambda^\flat(v))\to M_{*+\varepsilon\,v\cdot w}^{T\times G\times H}(\Lambda^\flat(v)\times\Hom_I(v,w)).$$

\item{} The pullback by $\beta$ gives an $M^*_{T\times G\times H}$-linear map
$$\beta^*\,:\,M_{*}^{T\times G\times H}(\Lambda^\flat(v)\times\Hom_I(v,w))
\to M_*^{T\times G\times H}(\L^\flat_s(v,w)).$$

\item{} The quotient by the $G$-action is a $T\times H$-equivariant principal $G$-bundle
$\L^\flat_s(v,w)\to\frakL^\flat_s(v,w)$. Hence, we have an $M^*_{T\times H}$-linear descent isomorphism
\begin{align*}
f^\heartsuit:M_{*+\varepsilon\,v\cdot v}^{T\times G\times H}(\L^\flat_s(v,w))\to M_*^{T\times H}(\frakL^\flat(v,w)).
\end{align*}

\end{itemize}
Composing $\rho_1^\sharp$, $p_1^*$, $\beta^*$ and $f^\heartsuit$ we get a map
\begin{align}\label{p-beta}
M_*^{T\times G}(\Lambda^\flat(v))\to M_{*+\varepsilon\, v\cdot(w-v)}^{T\times H}(\frakL^\flat(v,w)).
\end{align}

\smallskip

\smallskip

Now, we set $v=w$. Hence \eqref{diag4} reads
\begin{align*}
\xymatrix{\Lambda^\flat(v)\times G(v)\ar@{^{(}->}[r]^-\alpha\ar[r]\,\,&\L_{s}^\flat(v,v)\,\ar@{^{(}->}[r]^-\beta&\Lambda^\flat(v)\times\frakg(v)
}\end{align*}
and the $T\times G\times H$-action on $\Lambda^\flat(v)\times\frakg(v)$
preserves the open subsets $\Lambda^\flat(v)\times G(v)$ and $\L_{s}^\flat(v,v)$.
We define the following maps :

\begin{itemize}

\item{} The pullback by $\alpha$ gives a map
$$\alpha^*:M_*^{T\times G\times H}(\L^\flat_s(v,v))\to M_*^{T\times G\times H}(\Lambda^\flat(v)\times G(v)).$$

\item{} The $T\times H$-equivariant principal $G$-bundle 
$$q:\Lambda^\flat(v)\times G(v)\to\Lambda^\flat(v),\quad (\bar x,g)\mapsto g\cdot\bar x,$$
where $T\times H$ acts on the right hand side as in \S \ref{sec:sn},
yields a descent isomorphism
$$q^\heartsuit:
M_*^{T\times G}(\Lambda^\flat(v))\to M_{*+\varepsilon\,\v\cdot v}^{T\times G\times H}(\Lambda^\flat(v)\times G(v))$$
which intertwines the $M^*_{T\times G}$-action on the left hand side with the $M^*_{T\times H}$-action on the right hand side
under the isomorphism $M^*_{T\times G}=M^*_{T\times H}$ given by the identification $G=H$.

\item{} The first projection is a $T\times G$-equivariant principal $H$-bundle 
$q_1:\Lambda^\flat(v)\times G(v)\to\Lambda^\flat(v).$
Hence we have an $M^*_{T\times G}$-linear isomorphism
$$q_1^\heartsuit:M_*^{T\times G}(\Lambda^\flat(v))\to M_{*+\varepsilon\,v\cdot v}^{T\times G\times H}(\Lambda^\flat(v)\times G(v)).$$

\end{itemize}
Composing $(f^\heartsuit)^{-1}$, $\alpha^*$ and $(q^\heartsuit)^{-1}$ we get a map
\begin{align}\label{alpha}M^{T\times H}_*(\frakL^\flat(v,v))\to M_*^{T\times G}(\Lambda^\flat(v))\end{align}
which intertwines the $M^*_{T\times H}$-action on the left hand side with the $M^*_{T\times G}$-action on the right hand side.

\smallskip

Note that the map \eqref{p-beta} is $M^*_T$-linear but is not $M^*_{T\times G}$-linear in general.
More precisely, we consider a new $M^*_{G}$-action on $M_*^{T\times H}(\frakL^\flat(v,v))$ defined as follows.
Let $\calV_i$ be the \emph{universal} 
$T\times H$-equivariant vector bundle on $\frakL^\flat(v,v)$, which is obtained by restricting the
universal bundle on $\frakM(v,v)$, see \eqref{bundle} below. 
The $\bbQ$-algebra $M^*_G$ is spanned by the classes $\ch_l(\theta_i)$ with $i\in I$ and $l\in\bbN$, of the tautological
$G$-equivariant bundles $\theta_i$ on $\{\bullet\}$. We define the new action so that the element $\ch_l(\theta_i)$ acts on
$M_*^{T\times H}(\frakL^\flat(v,v))$ by the cap product with the cohomology class $\ch_l(\calV_i)$.
Then, the map \eqref{p-beta} is $M^*_{T\times G}$-linear relative to the new $M^*_{T\times G}$-action on the right hand side.

\smallskip

\begin{lemma}\label{lem:inj-surj} 
If $v=w$ then the map \eqref{p-beta} is injective while \eqref{alpha} is surjective.
\end{lemma}

\smallskip

\begin{proof}
Assume that $v=w$. We abbreviate $\Lambda=\Lambda^\flat(v)$ and $\frakg=\frakg(v)$.
The map $i=\alpha\circ\beta$ is the obvious inclusion $\Lambda\times G\to \Lambda\times\frakg$.
The composed map 
\begin{align}\label{CM}
\begin{split}
\xymatrix{
M_*^{T\times G\times H}(\Lambda)\ar[r]^-{p_1^*}&
M_{*+\varepsilon\dim G}^{T\times G\times H}(\Lambda\times\frakg)\ar[r]^-{i^*}&
M_{*+\varepsilon\dim G}^{T\times G\times H}(\Lambda\times G)\ar[d]^-{(q_1^\heartsuit)^{-1}}\\
M_*^{T\times G}(\Lambda)\ar[u]^-{\rho_1^\sharp}&&M_*^{T\times G}(\Lambda)}
\end{split}
\end{align}
is equal to the identity.
Now, we have $i^*=\alpha^*\circ\beta^*$.
We deduce that \eqref{alpha} is surjective and \eqref{p-beta} is injective.
\end{proof}

\smallskip

Theorem \ref{thm:1.1} follows from
Proposition \ref{prop:L} with $X=\frakL^\flat(v,v)$ and from Lemma \ref{lem:inj-surj}.
More precisely, to prove (b) we use the fact that if we have an 
epimorphism $H\to H'$ of MHS such that $H$ is pure, then $H'$ is also pure,
see, e.g., \cite[\S 3.1]{PS08}, and to prove (c) we use the naturality of the cycle map with respect to pullback by open 
immersions.
To prove part (d) we must check that
$M_*^{T\times G(v)}(\Lambda^\flat(v))$ is free as an $M^*_{T}$-module.
Note that the composed map \eqref{CM} is the identity and it factorizes through the $M^*_{T}$-module
$M_*^{T\times H}(\frakL^\flat(v,v))$, which is free by Proposition \ref{prop:L}.
So the theorem is proved.

\smallskip

\begin{remark}
Fix a degree $j$.
If $v\leqslant w$ then we consider the composed map
$$\xymatrix{M_j^{T\times G}(\Lambda)\ar[r]^-{p_1^*}&
\M_{j+\varepsilon\,v\cdot w}^{T\times G}
(\Lambda\times \Hom_I(v,w))\ar[r]^-{i^*}
&\M_{j+\varepsilon\,v\cdot w}^{T\times G}
(\Lambda\times \Mon_I(v,w)).}$$
The excision long exact sequence implies that the restriction $i^*$
is invertible for degree reasons if $w$ is large enough.
Since $p_1^*$ is invertible, we deduce that the map
$$\beta^*\circ p_1^*:M_j^{T\times G}(\Lambda)\to\M_{j+\varepsilon\,v\cdot w}^{T\times G}(\L_s^\flat(v,w))$$
is also injective if $w$ is larger than a constant depending on $j$ and $v$.
\end{remark}

\medskip

\subsection{Torsion-freeness}\label{sec:TF}\hfill\\

\smallskip

Set $\Bbbk=H^*_{T}$ and
$\Bbbk[v]=H^*_{T\times G(v)}$. Let $K$ be the fraction field of $\Bbbk$.
Recall the one-parameter subgroups $\theta$, $\theta^*$ of $G_{\Omega}$ defined in \S 3.3. 
For each $z\in\bbC^\times$, the element $\theta(z)$ acts by multiplication by 
$z$ on the summand $Rep(\k Q,v)$ in $\R(v)$ and trivially on $Rep(\k Q^*,v)$, 
while $\theta^*(z)$ acts by multiplication by $z$ on the summand $Rep(\k Q^*,v)$ and trivially on $Rep(\k Q,v)$.
We consider also the one-parameter subgroup $\theta-\theta^*$, such that 
$(\theta-\theta^*)(z)=\theta(z)\theta^*(z)^{-1}$.
In this section we prove the following result.

\smallskip

\begin{proposition}\label{prop:torsionfree} 
Let $Q$ be an arbitrary quiver, $\flat \in \{0,1\}$ and assume that the torus $T$ contains
the one-parameter subgroups $\theta$ and $\theta^*$.
Then for any dimension vector $v$ of $Q$, the $\Bbbk[v]$-module $H^{T \times G(v)}_*(\Lambda^{\flat}(v))$ is torsion-free.
\end{proposition}

\smallskip

\begin{proof} The proof proceeds through several reductions. 
We'll abbreviate $\gamma=\theta-\theta^*$. First, observe that the $\Bbbk[v]$-module
$$H_*^{T \times G(v)}(\Lambda^{\flat}(v)^{\gamma})=H^{T \times G(v)}_*(\{0\})$$ is torsion free. 
Since $T$ contains
the one-parameter subgroup $\gamma$, we can consider the ideal 
$$I_{\gamma} \subset\Bbbk[v]$$
which is the vanishing ideal 
of $\text{Lie}(\gamma(\bbG_m))$ in $\text{Lie}(T\times G(v))$. 
Let $T(v)\subset G(v)$ be a maximal torus and let $W(v)$ be the Weyl group of the pair $(G(v),T(v))$.
There are isomorphisms of $\Bbbk[v]$-modules, see, e.g., \cite[prop.~6]{EG98},
$$H^{T \times G(v)}_*(\Lambda^{\flat}(v)^{\gamma}) \simeq 
H^{T \times T(v)}_*(\Lambda^{\flat}(v)^{\gamma})^{W(v)},
\quad
H_*^{T \times G(v)}(\Lambda^{\flat}(v))\simeq H_*^{T \times T(v)}(\Lambda^{\flat}(v))^{W(v)}.$$
Hence, by \cite[thm.~6.2]{GKM}, the pushforward is an isomorphism of localized modules
$$H^{T \times G(v)}_*(\{0\})_{I_{\gamma}} =H^{T \times G(v)}_*(\Lambda^{\flat}(v)^{\gamma})_{I_{\gamma}} \simeq H_*^{T \times G(v)}(\Lambda^{\flat}(v))_{I_{\gamma}}.$$
It is therefore enough to show that the obvious map below is injective
$$ H_*^{T \times G(v)}(\Lambda^{\flat}(v)) \to  H_*^{T \times G(v)}(\Lambda^{\flat}(v))_{I_{\gamma}}.$$ 
Equivalently, we must prove 
that $H_*^{T \times G(v)}(\Lambda^{\flat}(v))$ is torsion free over the set $S_\gamma$ given by
$$S_{\gamma} = \Bbbk[v] \backslash I_{\gamma}.$$
By Theorem~\ref{thm:1.1}(d) there is an embedding
\begin{align}\label{00}
H_*^{T \times G(v)}(\Lambda^{\flat}(v)) \hookrightarrow H_*^{T \times G(v)}(\Lambda^{\flat}(v)) \otimes K,
\end{align}
while
by Proposition~\ref{prop:ratform}(b) below we have
\begin{align}\label{11}
H^{T \times G(v)}_*(\Lambda^{\flat}(v)) \otimes K \simeq H^{T \times G(v)}_*(\M(v)) \otimes K
\end{align}
if $\theta\subset T$.
Thus it is enough to show that $H^{T \times G(v)}_*(\M(v)) \otimes K$ is $S_\gamma$-torsion free. 

\smallskip

We will now use the following result of B. Davison. Put
$$E(v)=\{(x,u) \in Rep(\k Q^*,v) \times \frakg(v)\;;\; u \in End_{\k Q^*}(x)\}.$$
There is a natural action of $T \times G(v)$ on $E(v)$ defined as follows~: a tuple $(z_h,z) \in T$ acts by $z/z_h$ on $x_{h^*}$ and by $z^{-1}$ on $u$.  In particular, we have 
$$
\theta(z) \cdot (x,u)=(x,z^{-1}u),\qquad
\theta^*(z) \cdot (x,u)=(zx,z^{-1}u),\qquad\forall(x,u)\in E.$$
By \cite[eq. (29), (30)]{Davisondim}, 
there is an isomorphism of $\Bbbk[v]$-modules
\footnote{Note that \cite{Davisondim} does not consider $T$-equivariant Borel-Moore homology  and deals with stacks of 
nilpotent representations. As explained to us by B. Davison, 
the proof works verbatim in our situation since the potential $w$ 
involved in the dimensional reduction process is $T$-invariant, and the present setting of stacks of (not necessarily nilpotent) 
representations is actually simpler than that considered in \cite{Davisondim}.}
\begin{align}\label{22}
H^{T \times G(v)}_*(\M(v)) \simeq H^{T \times G(v)}_{* + o(v)} (E(v))
\end{align}
where $o(v)=-2\langle v, v\rangle$. The same argument as in Proposition~\ref{prop:ratform}(b) shows that
\begin{align}\label{33}
H_*^{T \times G(v)}(E(v)) \otimes K \simeq H^{T \times G(v)}_*(N(v)) \otimes K
\end{align}
 if $\theta\subset T$ or $\theta^*\subset T$, where
$$N(v)=\{(x,u) \in E(v)\;;\; u\;\text{is\;nilpotent}\}.$$
Combining the isomorphisms \eqref{00}-\eqref{33}, if $\theta,\theta^*\subset T$ we get an inclusion
$$H^{T \times G(v)}_*(\Lambda^{\flat}(v)) \hookrightarrow H^{T \times G(v)}_*(N(v)) \otimes K.$$
Proposition~\ref{prop:torsionfree} will thus be proved once we show that 
$H^{T \times G(v)}_*(N(v))$ has no $S_\gamma$-torsion. 

The advantage of $N(v)$ over $\Lambda^{\flat}(v)$ or $\M(v)$ is that it carries a natural stratification (by Jordan types) with good 
properties. More precisely, let $(x,u) \in N(v)$ and let $l$ be the index of nilpotence of $u$. Consider the $\k Q^*$-modules 
$F_i= \Im(u^{i-1})/ \Im(u^i)$ for $i=1, \ldots, l$ and the chain of epimorphisms of $\k Q^*$-modules
$$F_1 \twoheadrightarrow F_2 \twoheadrightarrow \cdots \twoheadrightarrow F_l \twoheadrightarrow F_{l+1}=0$$
induced by $u$. Let $v_k$ be the dimension vector  of the kernel of the map $F_k \to F_{k+1}$. 
Note that $v=\sum_{k=1}^l k v_k$ and that some of the $v_k$ may be zero, but $v_l \neq 0$. 
We define the \emph{Jordan type} of $(x,u)$ as the tuple of dimension vectors $J(x,u)$ given by 
$$J(x,u)=(v_1,v_2,\dots,v_l).$$ It only depends on $u$. For any such a tuple $J=(v_1,v_2,\dots,v_l)$ we put $$N(J)=
\{(x,u) \in N(v)\;;\; J(x,u)=J\},$$ obtaining in this way a partition into locally closed subsets
\begin{equation}\label{eq:Jordanstrat}
N(v)=\bigsqcup_{J} N(J)
\end{equation}
where the union ranges over all tuples $J$ such that $\sum_k k v_k=v$. 
First, we consider each stratum $N(J)$ separately of $N(v)$.
Following \cite[$\S 5$]{MozSch}, we consider the morphism of stacks
\begin{align*}
\pi_{J}~: N(J)\,/\,T \times G(v) &\to \prod_k Rep( \k Q^*, v_k) \,/ \,T \times G(v_k),\\
(x,u) &\mapsto (\Ker(F_k \to F_{k+1}))_k
\end{align*}
where the product of stacks on the right hand side is taken over the classifying stack  $\bullet\, / \,T$. 
By \cite[prop. 5.1]{MozSch}, the map $\pi_{J}$ is a composition of stack vector bundles.
Note that although
\cite[prop. 5.1]{MozSch}
 is written for coherent sheaves on smooth projective curves, the argument is valid for (smooth) moduli stacks of
 objects in arbitrary abelian category of homological dimension at most one. As a consequence, the map $\pi_J$ induces an 
 isomorphism of $\Bbbk[v]$-modules
\begin{equation}\label{E:torsionfree1}
\pi_{J}^*~: \bigotimes_k H^{T \times G(v_k)}_*(Rep(\k Q^*, v_k)) \simeq H^{T \times G(v)}_*(N(J))  
\end{equation}
where the tensor product on the left hand side is taken over $\Bbbk$ and the $\Bbbk[v]$-module structure (on the left hand side) 
comes from the restriction from $G(v)$ to the stabilizer subgroup of an element $u \in \frakg(v)$ of the conjugacy class of nilpotent 
elements of Jordan type $J$, because 
$$\text{Stab}_{G(v)}u / \rad(\text{Stab}_{G(v)}u) \simeq \prod_k G(v_k).$$ From (\ref{E:torsionfree1}) it 
follows that 
\begin{align}\label{LL}H_*^{T \times G(v)}(N(J)) \simeq \bigotimes_k \Bbbk[v_k]\end{align} and in particular 
$H_*^{T \times G(v)}(N(J))$ is pure and even. 

The purity implies that the excision long exact sequences 
associated to the partition of $N(v)$ into Jordan strata $N(J)$ splits into short exact sequences,  inducing a filtration on 
$H_*^{T \times G(v)}(N(v))$ whose 
associated graded is $\bigoplus_{J} H^{T \times G(v)}_*(N(J))$. Therefore, to prove that $H^{T \times G(v)}_*(N(v))$ has no 
$S_\gamma$-torsion, it is enough to check that each $H^{T \times G(v)}_*(N(J))$ has no $S_\gamma$-torsion.

To do this, we use \eqref{LL} again. 
Fix an element $u \in \frakg(v)$ of the conjugacy class of nilpotent 
elements of Jordan type $J$. We have 
$$N(J)=(T\times G(v))\times_{Stab}(N(J)\cap (Rep(Q^*,v)\times\{u\}))$$
where $Stab$ stands for the stabilizer of $u$ in $T \times G(v)$. The argument above implies that
$$H^{T\times G(v)}_*(N(J))=H^{T\times G(v)}_*((T\times G(v))/ Stab) = H^*_{Stab}.$$ 
The $H^*_{T \times G(v)}$-module structure on the right hand side is induced by the restriction 
$H^*_{T \times G(v)} \to H^*_{Stab}$ relative to the inclusion $Stab\subset T\times G(v)$. It is thus enough to show that no element of $S_\gamma$ is mapped to zero by this restriction, equivalently that the kernel of this restriction map is contained in $I_\gamma$. But this follows from the fact that $\gamma \subset Stab$.
This completes the proof of Proposition~\ref{prop:torsionfree}.
\end{proof}

\smallskip

\begin{remark} 
a) In this paper we only use Proposition~\ref{prop:torsionfree} in the particular case of the Jordan quiver
(in Theorem~\ref{thm:gen}). In this case, a more direct proof can be given which does not use \cite{Davisondim}.
The proposition is important for further applications. For instance, coupling it with Theorem~\ref{thm:gen} and the methods of \cite{SV13b} or \cite{YZ14}, we get a combinatorial realization of the cohomological Hall algebras $\Y^\flat$ to be introduced in the next section. Namely, they are subalgebras of shuffle algebras generated by some explicit elements. We will not need this here.\\
b) A (wrong) previous version of Proposition~\ref{prop:torsionfree} involved only the condition that $\theta \subset T$. In fact, one can prove that under this sole condition the map to the shuffle algebra is \textit{not} injective, see Appendix~\ref{sec:non-injectivity}.
\end{remark}

\medskip

\subsection{Kirwan Surjectivity for fixed point quiver varieties}\label{sec:KS}\hfill\\

\medskip

Fix dimension vectors $v,w \in \mathbb{N}^I$. Consider the $\mathbb{G}_m$-action on $\frakM(v,w)$ associated to a cocharacter $\sigma=\sigma_T \times \sigma_{G(w)}$ of $ T\times G(w)$. The cocharacter $\sigma_{G(w)}$ 
yields a weight space decomposition $W=\bigoplus_l W^{(l)}$ of $I$-graded vector spaces. The cocharacter $\sigma_T$ yields a 
$\mathbb{G}_m$-action on $\R(v,w)$ such that 
$$\sigma_T(t) \cdot (x_h)=t^{d(h)}x_h, \quad \sigma_T(t) \cdot (x^*_h)=t^{d^*(h)}x^*_h, \quad \sigma_T(t) \cdot (a_i)=a_i, \quad \sigma_T(t) \cdot (a^*_i)=t^{d}a^*_i$$
for some integers $d, d(h), d^*(h)$ with $ h \in \Omega$ such that
$d(h)+d^*(h)=d$ for any $h \in \Omega$.

\smallskip

 Next, let $\rho$ be a cocharacter of $G(v)$. Let $V=\bigoplus_{l \in \mathbb{Z}} V^{(l)}$ be the decomposition in weight spaces according to $\rho$. We will set 
$$V^{(\geqslant l)}=\bigoplus_{k \geqslant l} V^{(k)}, \quad v_{i,l}=\dim(V^{(l)} \cap V_i), \quad 
v^{\bullet}=(v_{i,l})$$
and use similar notations for $W$.

\smallskip

 Define $\M[\rho],$ $ \frakM[\rho]$ as in \S \ref{sec:BBQ} (with respect to $\sigma$).
There is an action of the group $T \times G(v^{\bullet}) \times G(w^\bullet)$ on $\M[\rho]$ and there is a geometric quotient $\frakM[\rho]=\M[\rho]/\hspace{-.05in}/ G(v^{\bullet})$. By construction, $\frakM[\rho]$ carries a collection of $T$-equivariant tautological bundles $\mathcal{V}_{i,l}$, $\mathcal{W}_{i,l}$ whose Chern classes generate an action of $H^*_{T \times G(v^\bullet) \times G(w^\bullet)}$. The version of Kirwan surjectivity which we will use (in the particular case of the $\circ$ action) is the following one.

\smallskip

\begin{theorem}\label{thm:Kirwangraded} Assume that $d \neq 0$. Then the following map is surjective
$$M^*_{T \times G(v^{\bullet}) \times G(w)} \to M^*_{ T \times G(w)}(\frakM[\rho]),\quad c \mapsto c \cdot 1.$$
\end{theorem}
\begin{proof} A proof is given in the Appendix for the case of the $\circ$ action. 
The case of a general cocharacter $\sigma$ with $d \neq 0$ can be proved in a similar way.
\end{proof}

\smallskip

\begin{corollary}\label{cor:Mrhoconn} 
For any pair of cocharacters $\sigma,\rho$ of $T \times G(w)$ and $G(v)$ respectively
for which $d \neq 0$, the variety $\frakM[\rho]$ is connected. 
\end{corollary}
\begin{proof} By Theorem~\ref{thm:Kirwangraded}, the cohomology group 
$H^0(\frakM[\rho])$ is one-dimensional.\end{proof}

\medskip

\section{The algebra $\Y^\flat$}
\label{sec:product}

\medskip

Until the end of the paper we assume that  $\k=\bbC$.
For each dimension vector $v$ we fix a closed subgroup $H(v)\subset G(v)$.
Set $M^*=\A^*$ or $H^*,$ and
$$ \Bbbk\langle v\rangle=M^*_{T\times H(v)},\quad
\Bbbk=M^*_{T},\quad 
\Bbbk[v]=M^*_{T\times G(v)}.$$
Let $K$ be the fraction field of $\Bbbk$.
We abbreviate
$$\otimes=\otimes_\Bbbk,\quad\Hom=\Hom_\Bbbk,\quad (\bullet)^\vee=\Hom(\bullet,\Bbbk).$$

\smallskip

\subsection{Definition}\label{sec:def-product}\hfill\\

Consider the abelian group $\bbX=\bbZ^I\times \bbZ$ and the sub-semigroup $\bbX_+=\bbN^I\times(-\bbN)$.
Let $\Y^\flat$ be the $\bbX$-graded $\Bbbk$-module given by
$$\Y^\flat=\bigoplus_{v,k}\Y^\flat(v,k),\quad
\Y^\flat(v,k)=M_{k+\varepsilon d_v}^{T\times G(v)}(\Lambda^\flat(v))
=M_{k-\varepsilon (v,v)/2}^{T}(\Lambda^\flat(v)\,/\,G(v)).$$
For each dimension vector $v$ the $\Bbbk$-submodule of $\Y^\flat$ given by
$$\Y^\flat(v)=\bigoplus_{k}\Y^\flat(v,k)$$  has a canonical $\Bbbk[v]$-module stucture given by the cap-product $\cap$.
We define a $\Bbbk$-algebra structure on $\Y^\flat$.
The multiplication is given by correspondences.
The same multiplication was first defined in \cite{SV13a}, \cite{SV13b} for the algebra $\Y$ in the particular case where $|I|=1$.
It differs from the usual convolution product on the Borel-Moore homology of the moduli stack of quivers, because it uses a pullback
by a non l.c.i.~morphism. This pullback is defined as a a refined pullback. We could as well view it as an ordinary pullback
relative to some 
\emph{virtual fundamental classes} which compensate the singularities of the spaces.
This multiplication was later generalized to $\Y$ and arbitrary quivers in \cite{YZ14}.

\smallskip

To avoid any confusion we may write $\Y^\flat_H$ or $\Y^\flat_A$  for $\Y^\flat$ to distinguish between the algebras associated with $M_*=H_*$
and $M_*=A_*$.

\smallskip

First, we introduce some notation.
We write 
\begin{align}\label{f5}
L=G(v_1)\times G(v_2).
\end{align}
Fix an $I$-graded vector subspace $V_1\subseteq V$. Let
\begin{align}\label{f5b}
P,\quad U,\quad \fraku
\end{align}
be the stabilizer of $V_1$ in $G(v)$, its unipotent radical and the Lie algebra of the unipotent radical.
Following \eqref{R1}, we set
\begin{align}\label{f6}
\begin{split}
&\R_L=R(v_1)\times R(v_2),\\
&\Lambda^\flat_L=\Lambda^\flat(v_1)\times\Lambda^\flat(v_2),\\
&\R_P=\{\bar x\in \R(v)\,;\,\bar x(V_1)\subseteq V_1\},\\
&\Lambda^\flat_P=\R_P\cap\Lambda^\flat(v).
\end{split}
\end{align}

\smallskip

A closed point of $G(v)/U$ is identified with a pair of maps $\phi_1,$ $\phi_2$ 
which give an exact sequence of $I$-graded vector spaces
\begin{align}\label{seq1}
0\longrightarrow\,V_1\,{\buildrel\phi_1\over\longrightarrow}\,V\,{\buildrel\phi_2\over\longrightarrow}\,V_2\,\longrightarrow 0.
\end{align}
Thus, the set of triples $(\bar x,\phi_1,\phi_2)$ where $\bar x\in \R(v)$ and $\phi_1,$ $\phi_2$ are as above with
$\Im(\phi_1)$ stable by $\bar x$ is identified with the fiber product
$$G(v)\times_U \R_P.$$
For each such triple $(\bar x,\phi_1,\phi_2)$ the representation $\bar x$ in $V$ induces representations $\bar x_1$ in $V_1$ 
and $\bar x_2$ in $V_2$, yielding the diagram 
\begin{align}\label{seq2}
\R_L\,{\buildrel q\over \longleftarrow}\, G(v)\times_U \R_P\,{\buildrel p\over \longrightarrow} \,\R(v),
\end{align}
where the maps $p,q$ are given by $p(\bar x,\phi_1,\phi_2)=\bar x$ and $q(\bar x,\phi_1,\phi_2)=\bar x_1\oplus\bar x_2$.
The map $q$ is smooth while $p$ is proper.
The group $G(v)\times L$ acts in the obvious way on the variety
$G\times_U \R_P$.
The diagram \eqref{seq2} yields the following diagram of Artin stacks
\begin{align}\label{seq3}
\Lambda^\flat_L/L
\,{\buildrel q\over \longleftarrow}\, \Lambda^\flat_{P}/P\,{\buildrel p\over \longrightarrow}\, \Lambda^\flat(v)/G(v).
\end{align}

\smallskip

In the diagram \eqref{seq3} the map $q$ is not smooth. So, the pullback homomorphism $q^*$ is not well-defined.
In \cite{SV13a} this pullback is replaced by a refined pullback.
Let $\fraku$ be the Lie algebra of $U$.
The map $q$ in \eqref{seq3} factorizes as follows.
Consider the maps $c\,:\,\Lambda^\flat_L\to \fraku\times \R_L$ and $q_3\,:\,\R_P\to \fraku\times \R_L$  given by
\begin{align*}
c(\bar x_1\oplus\bar x_2)&=(0\,,\,\bar x_1\oplus\bar x_2),\\
q_3(\bar x)&=\Big(\mu(\bar x)-\mu(\bar x_1)\oplus\mu(\bar x_2)\,,\,\bar x_1\oplus\bar x_2\Big).
\end{align*}
We have the following fiber diagram of Artin stacks
\begin{align}\label{seq4}
\begin{split}
\xymatrix{
&(\fraku\times \R_L)/P&\R_P/P\ar[l]_-{q_3}&\\
\Lambda^\flat_L/L&\Lambda^\flat_L/P\ar@{^{(}->}[u]^-{c}\ar[l]_-{q_1}
&\Lambda^\flat_{P}/P\ar@{^{(}->}[u]^-{}\ar[l]_-{q_2}\ar[r]^-p&\Lambda^\flat(v)/G(v).
}
\end{split}
\end{align}
The vertical maps are closed immersions.
The morphism $q_3$ is an l.c.i. because $\fraku\times \R_L$ and $\R_P$ are smooth.
Hence, the refined pullback morphism $(q_3,q_2)^!$ is well-defined.

\smallskip

Now, we have
\begin{align*}
\dim(U)=v_1\cdot v_2,\quad\dim(\fraku\times \R_L)-\dim \R_P=(v_1,v_2)-v_1\cdot v_2.
\end{align*}
The pullback by $q_1$ yields a morphism of $\bbZ$-graded $\Bbbk$-modules
$$q_1^*\,:\,M_*(\Lambda^\flat_L/L)\to M_{*-\varepsilon\,v_1\cdot v_2}(\Lambda^\flat_L/P).$$
The refined pullback by $(q_3,q_2)$ yields a morphism of $\bbZ$-graded $\Bbbk$-modules
$$(q_3,q_2)^!\,:\,M_*(\Lambda^\flat_L/P)\to M_{*+\varepsilon\,v_1\cdot v_2-\varepsilon\,(v_1,v_2)}(\Lambda^\flat_{P}/P).$$
The pushforward yields a morphism of $\bbZ$-graded $\Bbbk$-modules
$$p_*\,:\,M_*(\Lambda^\flat_{P}/P)\to M_*(\Lambda^\flat(v)/G(v)).$$
The maps above are $ T$-equivariant. Since $\dim(G(v))=d_v+(v,v)/2$,
composing the map $p_*\circ (q_3,q_2)^!\circ q_1^*$ with the Kunneth homomorphism
$$M_*(\Lambda^\flat(v_2)/G(v_2))\otimes M_*(\Lambda^\flat(v_1)/G(v_1))\to M_*(\Lambda^\flat_L/L)$$
we get a $\Bbbk$-linear map
\begin{align}\label{star}\star\,:\,
M_{k+\varepsilon\,d_{v_2}}^{T\times G(v_2)}(\Lambda^\flat(v_2))\otimes M_{l+\varepsilon\,d_{v_1}}^{T\times G(v_1)}(\Lambda^\flat(v_1))\to
M_{k+l+\varepsilon\,d_{v}}^{T\times G(v)}(\Lambda^\flat(v)),
\end{align}
hence a morphism of $\bbZ$-graded $\Bbbk$-modules
$$\star\,:\,\Y^\flat(v_2)\otimes\Y^\flat(v_1)\to\Y^\flat(v).$$

\smallskip

\begin{proposition}\label{prop:1.15} Let $\flat$ be either $0$ or $1$.\hfill
\begin{itemize}[leftmargin=8mm]

\item[$\mathrm{(a)}$]
$(\Y^\flat,\star)$ is an $\bbX_+$-graded associative algebra over $\Bbbk$.

\item[$\mathrm{(b)}$]
$\Y^\flat(0)\simeq\Bbbk$ with
$[\Lambda^\flat(0)]\mapsto 1.$
The unit of $\Y^\flat$ is the element $[\Lambda^\flat(0)]$.

\item[$\mathrm{(c)}$]
$\Y^\flat(v,0)$ is spanned over $\bbQ$ by the set $\{[X]\,;\,X\in\Irr(\Lambda^\flat(v))\}$.

\item[$\mathrm{(d)}$]
$\cl : \Y^\flat_A\to \Y^\flat_H$ is a surjective $\bbQ$-algebra homomorphism.
\end{itemize}
\end{proposition}

\smallskip

\begin{proof}
Part (a) is proved in \cite{SV13a}, \cite{SV13b} in the particular case where $|I|=1$ and in \cite{YZ14} for arbitrary quivers.
Part (d) follows from Theorem \ref{thm:1.1}.
The other claims are obvious.
\end{proof}

\smallskip

\subsection{Comparison  of $\Y^0$ and $\Y^1$}\hfill\\

We can also define an $\bbX_+$-graded $\Bbbk$-algebra structure on the $\Bbbk$-module
$$\Y=\bigoplus_{v,k}\Y(v,k),\quad\Y(v,k)=M_{k+\varepsilon d_v}^{T\times G(v)}(\M(v)).$$
As $v$ runs over $\bbN^I$, the pushforward by the canonical closed embeddings 
$$\Lambda^1(v)\subset\Lambda^0(v)\subset \M(v)$$ 
gives $\bbX$-graded $\Bbbk$-linear maps 
$$\Y^1\to\Y^0\to \Y.$$

\smallskip

\begin{proposition}\label{prop:ratform}\hfill
\begin{itemize}[leftmargin=8mm]
\item[$\mathrm{(a)}$]  The maps $\Y^1\to\Y^0$ and $\Y^0\to\Y$ above are $\Bbbk$-algebra homomorphisms.
\item[$\mathrm{(b)}$] 
If the torus $T$ contains a one parameter subgroup which scales all the quiver data by the same scalar, then
\begin{itemize}
\item[$\bullet$]
the maps $\Y^1\otimes \,K\to\Y^0\otimes \,K$ and $\Y^0\otimes \,K\to\Y\otimes \,K$
are $K$-algebra isomorphisms,
\item[$\bullet$]
for each $v,$ $w$ the pushforward yields an isomorphism
$$H_*^{T\times H(w)}(\frakL^\flat(v,w))\otimes K\to H_*^{T\times H(w)}(\frakM(v,w))\otimes K.$$
\end{itemize}
\end{itemize}
\end{proposition}

\smallskip

\begin{proof}
Part (a) follows from the proper base change property of refined pullbacks applied to the map $(q_3 ,q_2)^!$
in \S\ref{sec:def-product}.
Now, let us check that the map $\Y^\flat\otimes \,K\to\Y\otimes \,K$ is an isomorphism.
Although the set $\Lambda^\flat(v)$ is not projective and may be very singular,
it admits only finitely many orbit types relative to the $T$-action, which means that
only finitely many stabilizers occur since $T$ is abelian. 
Therefore, the localization theorem applies.
We have the following lemma.

\smallskip

\begin{lemma}\label{lem:LOC} Let $T$ be a torus and $G$ be a product of general linear groups over $\k$.
Let $V$ be a finite dimensional rational representation of $T\times G$ and $X\subset V$
be a closed $T\times G$-equivariant subset. Let $I\subset\k[V]$ be the ideal of polynomials vanishing at 0 and
$\Lambda\subset X$ be the 0-set of $I^G$.
Assume that $I^{T\times G}=\{0\}.$
Then, the pushforward yields an isomorphism
$$M^{T\times G}_*(\Lambda)\otimes K\to M^{T\times G}_*(X)\otimes K.$$
\end{lemma}

\smallskip

\begin{proof}
Fix a maximal torus $D\subseteq G$ and let $W=N_G(D)/D$ be the Weyl group.
Then, the $N_G(D)$-action on $X$ gives rise to a $W$-action on $H^D_*(X)$ such that
$$M_*^G(X)=(M_*^D(X))^W.$$
Hence, we may assume in the rest of the proof that $G$ is a torus.

\smallskip

Set $A=T\times G$ and $\fraka=\text{Lie}(A)$. The pushforward by the inclusion $\Lambda\subset X$
gives an $M^*_{A}$-module homomorphism $i_*:M^{A}_*(\Lambda)\to M^A_*(X)$ which is invertible over the (open) 
complement of
$\bigcup_{x\in X\setminus \Lambda}\fraka_x$ in $\fraka$. Here $\fraka_x$ is the annihilator of $x$ in $\fraka$.

\smallskip

The assumptions imply that for each point $x\in X\setminus \Lambda$ there is a non zero $G$-invariant polynomial 
$f\in I$ which does not vanish at $x$ 
and a non trivial character $\chi:T\to\bbG_m$ such that $f$ has the weight $\chi$
relative to the $T$-action on $V$.
We deduce that $\fraka_x\subseteq\{0\}\times\frakg$.
Therefore, the map $i_*$ is invertible over the open subset 
$(\frakt\setminus\{0\})\times\frakg$ of $\fraka$.

\end{proof}

\smallskip

Now, set  $V=\R(v)$, $X=\M(v)$ and $G=G(v)$.
The pushforward by the inclusion $\Lambda\subseteq\Lambda^\flat(v)$ yields a chain of maps
$$M^{T\times G(v)}_*(\Lambda)\to M^{T\times G(v)}_*(\Lambda^\flat(v))\to M^{T\times G(v)}_*(X).$$
Since the torus $T$ contains a one parameter subgroup which scales all the quiver data by the same scalar, we have
$I^{T}=\{0\}.$
Then, the first part of (b) follows from Lemma \ref{lem:LOC}.

\smallskip

Now, we prove the second part of (b). Set 
$$\L^\flat(v,w)=\Lambda^\flat(v)\times\{0\}\times\Hom_I(v,w),\quad
\frakL^\flat_0(v,w)=\pi(\frakL^\flat(v,w)).$$
Set also $V=\R(v,w)$
and $G=G(v)$.
By Proposition \ref{prop:TL}, we have
$$\frakL^\flat(v,w)=\pi^{-1}(\frakL^\flat_0(v,w)),\quad 
\frakL^\flat_0(v,w)=\L^\flat(v,w)/\!\!/G,\quad
\Lambda\subseteq\L^\flat(v,w)\subseteq V.$$
We deduce that for each point $x\in \frakM(v,w)\setminus\frakL^\flat(v,w)$ 
there is a non zero $G$-invariant polynomial 
$f\in I$ which does not vanish at $\pi(x)$ 
and a non trivial character $\chi:T\to\bbG_m$ such that $f$ has the weight $\chi$
relative to the $T$-action on $V$.
Set $A=T\times H(w)$.
We deduce that
$\fraka_x\subseteq\{0\}\times\frakh(w)$.
Therefore, the pushforward map 
$$M_*^{T\times H(w)}(\frakL^\flat(v,w))\to M_*^{T\times H(w)}(\frakM(v,w))$$ 
is invertible over the open subset 
$(\frakt\setminus\{0\})\times\frakh(w)$ of $\fraka$.
\end{proof}

\smallskip

\subsection{Dimension}\hfill\\

Given commuting formal variables $z_i$ with $i\in I$, we write
$z^v=\prod_i(z_i)^{v_i}$ for each dimension vector $v\in \bbN^I$.
Consider the formal series in $\bbQ[[z_i\,;\,i \in I]]$ given by
\begin{align*}
\lambda^\flat(q,z)&=\sum_{v}
\frac{|\Lambda^\flat(v)(\mathbb{F}_q)|}{|G(v)(\mathbb{F}_q)|}\, q^{\langle v, v \rangle}z^{\v}.
\end{align*}
Let $A^\flat_v(t)=\sum_ra^\flat_{v,r}\,t^r$ with $\flat=0,1$
be the \emph{nilpotent} and the \emph{1-nilpotent} Kac polynomials in $\bbZ[t]$ considered in \cite{BSV} :
for each prime power $q$ the integers $A^0_v(q)$ and $A^1_v(q)$ count the number absolutely 
indecomposable nilpotent and 1-nilpotent representations of the quiver $Q$ of dimension $v$ over $\fq$.
Then, by \cite{BSV}, we have 
\begin{equation*}
\lambda^\flat(q, z)= \text{exp}\Big(\sum_{v\neq 0}\sum_{l\geqslant 1}
{A}^\flat_{v}(q^{-l}) \,z^{lv}/\,l\,(1-q^{-l})\Big)
\end{equation*}
Hence, 
the variety $\Lambda^\flat(v)$ has polynomial count.
Let $\tau$ be the rank of $T$.
Assume that $M_*=H_*$. 
Then, Corollary \ref{cor:poincare} and Theorem \ref{thm:1.1} yield
\begin{align*}
\sum_k\dim_\bbQ\!\big(\!\Y^\flat(v,k)\big)\,q^{k/2}&=
\sum_k\dim_\bbQ\!\big(H_{k}(\Lambda^\flat(v)\,/\,T\times G(v))\big)\,q^{\langle v,v\rangle+\tau+k/2}\\
&=q^{\langle v,v\rangle+\tau}\,\big|\Lambda^\flat(v)(\fq)\big|\,\big/\,\big|T(\fq)\times G(v)(\fq)\big|.
\end{align*}
We deduce that $\Y^\flat(v,k)=0$ if $k$ is odd and that we have
\begin{align*}
\sum_{v}\sum_k\dim_\bbQ\!\big(\!\Y^\flat(v,k)\big)\,q^{k/2}\,z^v
&=(1-q^{-1})^{-\tau}\cdot\text{exp}\Big(\sum_{v\neq 0}\sum_{l\geqslant 1}\sum_{k\geqslant 0}{A}^\flat_{v}(q^{-l}) \,q^{-lk}\,z^{lv}/\,l\,\Big).
\end{align*}
This yields the following formula.
\smallskip

\begin{theorem}\label{thm:dim}
Assume that $M_*=H_*$. We have
\begin{align*}
\sum_{v}\sum_k\dim_\bbQ\!\big(\!\Y^\flat(v,k)\big)\,q^{k/2}\,z^v
&=(1-q^{-1})^{-\tau}\cdot\prod_{v\neq 0}\prod_{r\geqslant 0}\prod_{k\geqslant 0}\big(1-q^{-k-r}\,z^{v}\big)^{-{a}^\flat_{v,r}}.
\end{align*}
\qed
\end{theorem}

\begin{example}
Let $Q$ be the Jordan quiver.
Assume that $T$ is the rank 2 torus which acts by dilatation on the variables $x$ and $x^*$.
Then we have $\Y^0=\Y^1$ and
$$\sum_{v}\sum_{k}\dim_\bbQ\!\big(\!\Y^\flat(v,k)\big)\,q^{k/2}z^v=
(1-q^{-1})^{-2}\cdot\prod_{v\geqslant 1}\prod_{k\geqslant 0}(1-q^{-k}z^v)^{-1}.$$
In this case, the algebra $\Y^\flat$ is (the positive half of) the spherical degenerate double affine Hecke algebra of type 
$\GL(\infty)$ defined in \cite{SV13b}.
\end{example}

\smallskip

\subsection{Generators}\smallskip

\subsubsection{The stratification of $\Lambda^1(v)$}

\smallskip

\begin{definition} Fix a vertex $i$ and an integer $l\geqslant 0$.
\begin{itemize}[leftmargin=8mm]
\item[$\mathrm{(a)}$] 
For each element $\bar x$ in $\R(v)$ let $\varepsilon_i(\bar x)$ be the codimension of the smallest subspace of $V_i$ containing
$\sum_{j\neq i}\sum_{h\in\bar\Omega_{ji}}\Im(\bar x_h\big)$ and preserved by $\bar x_h$ for all $h\in\bar\Omega_{ii}$.
\item[$\mathrm{(b)}$] 
Let $\R(v)^{l,i}$ and $\Lambda(v)^{l,i}$ be the locally closed subsets of $\R(v)$ given by
\begin{align*}
\R(v)^{l,i}&=\{\bar x\in \R(v)\,;\,\varepsilon_i(\bar x)=l\},\\
\Lambda^1(v)^{l,i}&=\Lambda^1(v)\cap \R(v)^{l,i}.
\end{align*}
\end{itemize}
\end{definition}

\smallskip

Fix dimension vectors $v$, $v_1$, $v_2$ with $v=v_1+v_2$.  
We use the same notation as in \eqref{f5}, \eqref{f6}.
We define the following sets
\begin{align*}
\R_P^{l,i}&=\R_P\cap \R(v)^{l,i},\\
\Lambda_{P}^{l,i}&=\Lambda^1_{P}\cap \R(v)^{l,i},\\
\R_L^{0,i}&=R(v_1)^{0,i}\times R(v_2),\\ 
\Lambda_L^{0,i}&=\Lambda^1(v_1)^{0,i}\times \Lambda^1(v_2).
\end{align*}

\smallskip

Now, assume that $v_2=l\, \delta_i$.
From \eqref{seq4} we get the following fiber diagram of stacks
\smallskip
\begin{equation}\label{diag1}
\begin{split}
\xymatrix{
(\fraku\times \R_L)/P&\R_P/P\ar[l]_-{q_3}&\\
(\fraku\times \R_L^{0,i})/P\ar@{^{(}->}[u]^-{j_3}&R_{P}^{l,i}/P\ar[l]_-{q_{3,0}}\ar@{^{(}->}[u]^-{j_4}&\\
\Lambda_L^{0,i}/P\ar@{^{(}->}[u]^-{}\ar@{_{(}->}[d]_-{j_1}&\Lambda_{P}^{l,i}/P\ar[l]_-{q_{2,0}}\ar@{^{(}->}[u]^-{}\ar@{_{(}->}[d]_-{j_2}
\ar[r]^{p_0}&\Lambda(v)^{l,i}/G\\
\Lambda^1_L/P&\Lambda^1_{P}/P\ar[l]_-{q_2}.&
}
\end{split}
\end{equation}

\smallskip

\begin{lemma} \label{lem:isomorphism1} 
Assume that $v_2=l\, \delta_i$.
\begin{itemize}[leftmargin=8mm]

\item[$\mathrm{(a)}$] $j_1,j_2,j_3$ and $j_4$ are open immersions. 

\item[$\mathrm{(b)}$] $q_{3,0}$ and $q_{2,0}$ are affine space bundles,

\item[$\mathrm{(c)}$] 
$q_{2,0}^*\circ j_1^*=j_2^*\circ (q_3,q_2)^!,$

\item[$\mathrm{(d)}$] 
$p_0$ is an isomorphism,

\item[$\mathrm{(e)}$] 
$p_{0,*}\circ q_{2,0}^*$ is an isomorphism 
$M_*^{T\times P}(\Lambda_L^{0,i})\to M_*^{T\times G}(\Lambda(v)^{l,i})$
such that $[\Lambda_L^{0,i}]\mapsto[\Lambda(v)^{l,i}].$

\end{itemize}
\end{lemma}

\begin{proof} Part (a) is obvious, because, since $v_2=l\,\delta_i$, we have
$\Lambda_{P}^{\geqslant l,i}=\Lambda^1_{P}$ and $R_{P}^{\geqslant l,i}=R_{P}$.
Fix $I$-graded vector spaces $V$, $V_1$, $V_2$ with dimension vectors $v$, $v_1$, $v_2$.
We abbreviate
$$\Hom_{\bar\Omega}(V_2,V_1)=\bigoplus_{h\in\bar\Omega}\Hom(V_{2,h'}\,,V_{1,h''}).$$
The fiber of $q_3$ over any tuple $(z,\bar x_1\oplus\bar x_2)$ in $\fraku\times \R_L$ is
\begin{align*}
q_3^{-1}(z,\bar x_1\oplus\bar x_2)
&=\{\bar y\in\Hom_{\bar\Omega}(V_2,V_1)\,;\,\mu(\bar x_1\oplus\bar x_2\oplus\bar y)=z\},\\
&=\big\{\bar y\in\Hom_{\bar\Omega}(V_2,V_1)\,;\,
\gamma_{\bar x_1\oplus\bar x_2}(\bar y)=z-\mu(\bar x_1)\oplus\mu(\bar x_2)\big\},
\end{align*}
where $\gamma_{\bar x_1\oplus\bar x_2}$ is the linear map given by
$$
\gamma_{\bar x_1\oplus\bar x_2}\,:\,\Hom_{\bar\Omega}(V_2,V_1)\to \fraku,\ 
\bar y\mapsto \sum_{h\in\bar\Omega}\varepsilon(h)\,\Big((\bar x_1)_{h^*}\bar y_h-\bar y_h(\bar x_2)_{h^*}\Big)
.$$
So, this fiber is either empty or a torsor over the vector space $\Ker(\gamma_{\bar x_1\oplus\bar x_2})$.
Now, since $v_2=l\,\delta_i$ we have $\fraku=\Hom(V_{2,i},V_{1,i})$ and, for any $\bar y\in\Hom_{\bar\Omega}(V_2,V_1)$,
\begin{align*}
\bar y_h\neq 0&\Rightarrow h'=i,\\
(\bar x_2)_h\neq 0&\Rightarrow h'=h''=i.
\end{align*}
Thus, we have
\begin{align*}
\gamma_{\bar x_1\oplus\bar x_2}(\bar y)=
\sum_{h\in\bar\Omega_{i\bullet}}\varepsilon(h)\,(\bar x_1)_{h^*}\bar y_h-
\!\sum_{h\in\bar\Omega_{ii}}\varepsilon(h)\,\bar y_h(\bar x_2)_{h^*}.
\end{align*}
We deduce that the orthogonal to $\Im(\gamma_{\bar x_1\oplus\bar x_2})$ relative to the trace is
\begin{align*}
\Im(\gamma_{\bar x_1\oplus\bar x_2})^\perp=\{z\in\Hom(V_{1,i}\,,V_{2,i})\,;\,z\,(\bar x_1)_h=(\bar x_2)_h\,z\,,\,\forall h\in\bar\Omega_{\bullet i}\}.
\end{align*}
In particular, if $z\in \Im(\gamma_{\bar x_1\oplus\bar x_2})^\perp$ then 
\begin{align}\label{imgamma} 
\begin{split}
\sum_{j\neq i}\sum_{h\in\bar\Omega_{ji}}\Im((\bar x_1)_h\big)&\subseteq\Ker(z),\\
\sum_{h\in\bar\Omega_{ii}}(\bar x_1)_h(\Ker(z))&\subseteq\Ker(z).
\end{split}
\end{align}
Now, assume that $\varepsilon_i(\bar x_1)=0$.
From \eqref{imgamma} we deduce that $$\Im(\gamma_{\bar x_1\oplus\bar x_2})^\perp=\{0\}.$$
Hence
the map $\gamma_{\bar x_1\oplus\bar x_2}$ is onto.
Thus the fiber $(q_3)^{-1}(z,\bar x_1\oplus\bar x_2)$ is not empty and has a constant dimension as 
$\bar x_1\oplus\bar x_2$ runs over the set $\R_L^{0,i}$. 
Hence, the map $q_{3,0}$ is an affine space bundle, so the map $q_{2,0}$ is also an affine space bundle.
This proves the claim (b).
Compare the proof of \cite[prop.~1.14]{B14}.

\smallskip

Part (c) follows from (a), (b) and standard properties of the refined pullback. More precisely, given a cartesian square as in 
\eqref{DDIAG} and an open immersion $X_0\subset X$ we set $Y_0=g^{-1}(X_0)$ and we consider the fiber diagram
$$\xymatrix{
W\ar[d]_-{f'}&\ar@{_{(}->}[l]_-{i}W_0\ar[d]_{f'_0}\ar[r]^{}&Z_0\ar@{^{(}->}[r]\ar[d]_{f_0}&Z\ar[d]_f\\
Y&\ar@{_{(}->}[l]_-{j}Y_0\ar[r]^{}&X_0\ar@{^{(}->}[r]&X.}$$
If the map $f_0$ is smooth,
then $f'_0$ is smooth by base change,
so the pullback $(f_0')^*$ is well-defined and coincides with $(f_0,f'_0)^!$, so the functoriality of the refined pull-back 
with respect to the smooth morphisms 
$i$, $j$
yields
$i^*\circ (f,f')^!=(f'_0)^*\circ j^*.$

\smallskip

To prove (d) note that, since we have $v_2=l\,\delta_i$, for
any tuple $(\bar x,\phi_1,\phi_2)$ in $G\times_U\Lambda_P^{l,i}$ the degree $i$ part of the $I$-graded vector 
space $\Im(\phi_1)$ is equal to $\sum_{j\neq i}\sum_{h\in\bar\Omega_{ji}}\Im(\bar x_h)$,
hence it is uniquely determined by $\bar x$.
Part (e) follows from (b) and (d).
\end{proof}

\medskip

\subsubsection{The reduction to one vertex quivers}
\label{sec:noloops}
For each vertex $i\in I$ we abbreviate
$$\Y^\flat(\bbN\,\delta_i)=\bigoplus_{l\geqslant 0}\Y^\flat(l\,\delta_i).$$
The first step to compute 
a system of generators of $\Y^1$ is the following.

\smallskip

\begin{proposition}\label{prop:2.2} The $\Bbbk$-algebra $\Y^1$ is generated by the subspace 
$\bigoplus_{i\in I}\Y^1(\bbN\,\delta_i)$.
\end{proposition}

\smallskip

\begin{proof}
By Proposition \ref{prop:1.15} it is enough to assume that $M_*=\A_*$. 
The proof is by induction. 
For each $i\in I$ we consider the finite filtration 
$$\{0\}\subseteq \cdots\subseteq \Y^1(v)^{>l,i}\subseteq \Y^1(v)^{\geqslant l,i}\subseteq \cdots
\subseteq\Y^1(v)$$ 
such that $\Y^1(v)^{\geqslant l,i}$ is the pushforward of $\A_*^{T\times G}(\Lambda^1(v)^{\geqslant l,i})$ in $\Y^1(v)$ by the closed 
embedding
$$\Lambda^1(v)^{\geqslant l,i}\subseteq\Lambda^1(v).$$

\smallskip

\begin{lemma}\label{lem:lemma1}
For each dimension vector $v$ we have $\Y^1(v)=\sum_{i\in I}\Y^1(v)^{\geqslant 1,i}.$
\end{lemma}
\begin{proof}
The Mayer-Vietoris exact sequence for Chow groups implies that 
$$\Y^1(v)=\sum_X(i_X)_*\big(\A^{T\times G}_*(X)\big),$$
where $i_X$ is the embedding of an irreducible component $X\subseteq\Lambda^1(v)$.
Therefore, it is enough to check that any irreducible component of $\Lambda^1(v)$ is contained into a closed 
subset of the form
$\Lambda(v)^{\geqslant \delta_i}$ for some vertex $i$. Since any element of 
$\Lambda^1(v)$ belongs to $\Lambda_W$ for some restricted increasing flag of $I$-graded vector spaces $(W_p)$, we have
$\Lambda^1(v)=\bigcup_i\Lambda(v)^{\geqslant \delta_i}$.
This proves the lemma.
\end{proof}

\smallskip

Let $\Y^{\star\, v}$ be the subalgebra of $\Y^1(v)$ generated by $\bigoplus_{i\in I}\Y^1(\bbN\,\delta_i)$.
Since $\Y^1(v)^{>v_i,i}=\{0\}$, by descending induction we may assume that 
for some $l>0$ we have
\begin{align*}\Y^1(v)^{>l,i}\subseteq\Y^{\star\, v}\end{align*} 
and 
we must prove that $\Y^1(v)^{\geqslant l,i}\subseteq\Y^{\star\, v}.$
The excision yields an exact sequence
$$0\to\Y^1(v)^{>l,i}\,\longrightarrow\,\Y^1(v)\,{\buildrel j^*\over\longrightarrow}\,\A_*^{T\times G}(\Lambda^1(v)^{\leqslant l,i})\to 0.$$
It is enough to prove that 
\begin{align}\label{form1}
j^*(\Y^1(v)^{\geqslant l,i})\subseteq j^*(\Y^{\star \,v}).
\end{align}
To do that, set $v_2=l\,\delta_i$ and $v_1=v-v_2$ and define $L$, $P$, etc, as in \eqref{f5}, \eqref{f6}.
Then, we have a commutative diagram
\begin{equation*}
\begin{split}
\xymatrix{
\A^{T\times P}_*(\Lambda^1_L)\ar@{->>}[r]^-{j_1^*}\ar[d]^-{(q_3,q_2)^!}&\A^{T\times P}_*(\Lambda_L^{0,i})\ar[d]^-{q_{2,0}^*}_-\wr\\
\A^{T\times P}_*(\Lambda^1_{P})\ar@{->>}[r]^-{j_2^*}\ar[d]^-{p_*}&\A^{T\times P}_*(\Lambda_{P}^{l,i})\ar[d]^-{p_{0,*}}_-\wr \\
\A^{T\times G}_*(\Lambda^1(v)^{\geqslant l,i})\ar@{->>}[r]^-{j^*}\ar@{->>}[d]^-{c_*}&\A^{T\times G}_*(\Lambda^1(v)^{l,i})\ar[d]^-{c_*}\\
\Y^1(v)^{\geqslant l,i}\ar[r]^-{j^*}&\A^{T\times G}_*(\Lambda^1(v)^{\leqslant l,i}).
}
\end{split}
\end{equation*}
More precisely, note that
\begin{itemize}
\item the commutativity of the upper square and the invertibility of the pullback morphism
$q_{2,0}^*$ follow from Lemma \ref{lem:isomorphism1}, 
\item the commutativity of the middle square follows from proper base change and the equality 
$\Lambda_{P}^{\geqslant l,i}=\Lambda^1_{P}$,
\item the commutativity of the lower square follows from proper base change relative to the cartesian square
\begin{equation*}
\xymatrix{
\Lambda^1(v)^{\geqslant l,i}\ar@{_{(}->}[d]_-c&\Lambda^1(v)^{l,i}\ar@{_{(}->}[d]_-c\ar@{_{(}->}[l]_-{j}\\
\Lambda^1(v)&\Lambda^1(v)^{\leqslant l,i}\ar@{_{(}->}[l]_-{j}.}
\end{equation*}
\item the invertibility of the morphism $p_{0,*}$ follows from the invertibility of the map $p_0$.
\end{itemize}
We deduce that for each $x\in\Y^1(v)^{\geqslant l,i}$ there is an element $y\in \A^{T\times P}_*(\Lambda^1_L)$ such that 
$$j^*(x)=j^*c_*\, p_*(q_3,q_2)^!(y),$$
from which we deduce that the following inclusion holds
\begin{align}\label{form2}
j^*(\Y^1(v)^{\geqslant l,i})\subseteq j^*(\Y^1(v_1)\star\Y^1(v_2)).
\end{align}
Now, since $v_1<v$, by an increasing induction on $v$ we can assume that 
$\Y^1(v_1)=\Y^{\star \,v_1}.$
Therefore, the identity \eqref{form1} follows from \eqref{form2}.
The proposition is proved.
\end{proof}

\smallskip

\begin{remark}\label{rem:q=0}
In the particular case where the quiver $Q$ has no 1-loops, the proposition implies that
the algebra $\Y^1$ is indeed generated by the subspace $\bigoplus_{i\in I}\Y^1(\delta_i).$
\end{remark}

\smallskip

\subsubsection{The case of  one vertex quivers}\label{sec:single}
Fix a vertex $i\in I$. Since $\Y^0(\bbN\,\delta_i)=\Y^1(\bbN\,\delta_i)$ and $\Lambda^0=\Lambda^1$, we drop the exponents 0, 1.

\smallskip

\begin{definition}
Given any integer $l>0$ let $\x_{i,l}$ be the element in $\Y$ given by
\begin{align*}
q_i>0&\,\Rightarrow\,\x_{i,l}=[\,\Lambda_{(l\,\delta_i)}\,],\\
q_i=0&\,\Rightarrow\, \x_{i,l}=\delta_{l,1}\,[\,\Lambda(\delta_i)\,].
\end{align*}
\end{definition}

\smallskip

\begin{proposition}\label{prop:2.6}
$\Y^\flat(\bbN\,\delta_i)$ is generated by the elements
$u\cap\,\x_{i,l}$ with $l>0$ and $u\in\Bbbk[l\,\delta_i]$.
\end{proposition}

\smallskip

\begin{proof}
We assume that $Q=\J(i,q)$ and $q>0$,  since the case $q=0$ is trivial by Remark \ref{rem:q=0}.
To avoid cumbersome notation we abbreviate 
$\x_l=\x_{i,l}$.
By Theorem \ref{thm:1.1} it is enough to assume that $M_*=\A_*$ is the Chow group.

\smallskip

Let us first concentrate on the hyperbolic case.
Fix a composition $\nu=(\nu_1,\dots,\nu_r)$ of $v$ and
choose an increasing $I$-graded flag $W$ in $V$ of type $\nu$. 
Let $P$ be the parabolic subgroup of $G(v)$ of block type $\nu$ which fixes the flag $W$. 
Let $L$, $U$ be the standard Levi subgroup of $P$ and its unipotent radical.
Let $\frakl$, $\fraku$ be their Lie algebras.
We define the sets $\R_P$, $\R_L$, $\Lambda_P$ and $\Lambda_L$ as in \eqref{f6}, using the parabolic subgroup $P$
instead of a 2 blocks parabolic subgroup as in loc.~cit.
Consider the closed subset $\Lambda_{(L)}$ of $\Lambda_L$ given by
\begin{align}\label{LambdaL}\Lambda_{(L)}=\{0\}\times\frakl^{\,\oplus\, q}.\end{align}
Then, we have an obvious isomorphism
$$\Lambda_W=\Lambda_P\times_{\Lambda_L}\Lambda_{(L)}.$$
Consider the following fiber diagram of stacks
\begin{align}\label{DIAG}
\begin{split}
\xymatrix{
(\fraku\times \R_L)/P&\R_P/P\ar[l]_-{q_3}&\\
\Lambda_L/P\ar@{^{(}->}[u]^-{}&\Lambda_{P}/P\ar@{^{(}->}[u]^-{}\ar[l]_-{q_2}\ar[r]^-p&\Lambda(v)/G(v)\\
\Lambda_{(L)}/P\ar@{^{(}->}[u]^-{c}&\Lambda_{W}/P\ar@{^{(}->}[u]^-{c}\ar[l]_-{q_4}\ar[r]^-p&
\Lambda_{\nu}/G(v)\ar@{^{(}->}[u]^-c\\
&\Lambda_{W}^\circ/P\ar@{^{(}->}[u]^-{j}\ar[r]^-{p}_-\sim&
\Lambda_{\nu}^\circ/G(v)\ar@{^{(}->}[u]^-j
.}
\end{split}
\end{align}
By proper base change we have
\begin{align}\label{proper}c_*\circ p_*\circ (q_3,q_4)^!=p_*\circ (q_3,q_2)^!\circ c_*.\end{align}
Further, we claim that 
\begin{align}\label{main>1}
j^*p_*(q_3,q_4)^!\big(\A_*^{T\times P}(\Lambda_{(L)})\big)=\A_*^{T\times G(v)}(\Lambda^\circ_{\nu}).
\end{align}
Since we have
\begin{align}\label{isom1}\A_*^{T\times P}(\Lambda_{(L)})=\bigotimes_{s=1}^r\A_*^{T\times G(\nu_s)}(\Lambda_{(\nu_s)})=
\bigotimes_{s=1}^r\Bbbk[\nu_s]\cap\,\x_{\nu_s},\end{align}
the proposition will follow by induction using the excision exact sequence
$$\A_*^{T\times G(v)}(\Lambda^\circ_{\triangleleft\nu})\longrightarrow 
\A_*^{T\times G(v)}(\Lambda^\circ_{\lescc\nu})
\,{\buildrel j^*\over\longrightarrow}\, \A_*^{T\times G(v)}(\Lambda^\circ_{\nu})\longrightarrow 0.$$

\smallskip

Now, we prove the identity \eqref{main>1}. 
Since the map $p$ is an isomorphism $G(v)\times_P\Lambda^\circ_W\to \Lambda^\circ_\nu$,
by base change relatively to \eqref{DIAG} we must check the relation
$$j^*(q_3,q_4)^!\big(\A_*^{T\times P}(\Lambda_{(L)})\big)=\A_*^{T\times P}(\Lambda^\circ_{W}).$$

\smallskip

\begin{lemma}\label{lem:2.8} Assume that $q>1$. Then, we have
\begin{itemize}[leftmargin=8mm]
\item[$\mathrm{(a)}$] $A_*^{T\times P}(\Lambda^\circ_{W})=\A^*_{T\times P}\cap\,[\Lambda^\circ_{W}]$,

\item[$\mathrm{(b)}$]  $\x_{\nu_1}\star\dots\star\x_{\nu_r}=[\Lambda_{\nu}]$.
\end{itemize}

\end{lemma}

\smallskip

\noindent By Lemma \ref{lem:2.8}(a) and the $\A_{T\times P}^*$-linearity of the pullback we are reduced to prove that 
$$[\Lambda^\circ_{W}]\in j^*(q_3,q_4)^!\big(\A_*^{T\times P}(\Lambda_{(L)})\big).$$
This relation follows from Lemma \ref{lem:2.8}(b).

\smallskip

Now, we consider the isotropic case. 
In this case, the proof of the proposition is similar to the proof in the hyperbolic case. 
Since $q=1$, 
the irreducible components of $\Lambda(v)$ are labelled by the partitions $\nu$ of $v$.
Then, we have again the fiber diagram \eqref{DIAG} with the notations as in \S \ref{sec:isotrope}.
Let $\nu_1\geqslant \nu_2\geqslant \cdots\geqslant \nu_r$ be the parts of $\nu$.
The proposition follows from the following lemma.

\smallskip

\begin{lemma}\label{lem:2.9} Assume that $q=1$. Then, we have
\begin{itemize}[leftmargin=8mm]
\item[$\mathrm{(a)}$] $A_*^{T\times P}(\Lambda^\circ_{W})=\A^*_{T\times P}\cap\,[\Lambda^\circ_{W}]$,

\item[$\mathrm{(b)}$]  $\x_{\nu_{1}}\star\dots\star\x_{\nu_{r}}\in [\overline{\Lambda_{\nu}^\circ}]+
\sum_{\mu\lescc\nu}\bbQ\,[\overline{\Lambda^\circ_\mu}]$.
\end{itemize}
\end{lemma}

\smallskip

\end{proof}

\smallskip

\begin{proof}[Proof of Lemma \ref{lem:2.8}]
We use the notation in  \S \ref{sec:BBQ} with $w=v$ and 
$\gamma=\xi^{-1}$, $\rho=\rho_\nu$ as in \eqref{action0}, \eqref{rho-nu}. 
Let $P$ denote the parabolic subgroup associated with the cocharacter $\rho_\nu$ in \eqref{274}.
It is the standard parabolic of block type $\nu$.
Since $Q=\J(i,q)$ we can omit the upperscript $\flat$ and we abbreviate 
$\L_s(v,v)=\L_s^\flat(v,v)$ and $\frakL(v,v)=\frakL^\flat(v,v)$.
By \eqref{ATT} we have the locally closed subset $\frakL[\rho_\nu]\subset\frakL(v,v)$, which is the geometric quotient of the 
$T\times P\times G(v)$-equivariant subset
$\L[\rho_\nu]\subset\L_{s}(v,v)$.
By \eqref{subset1}, the open immersion $\alpha:\Lambda(v)\times G(v)\to \L_{s}(v,v)$ in \eqref{diag4} satisfies
$$\Lambda^\circ_{W}\times G(v)=\alpha^{-1}(\L[\rho_\nu]).$$
From descent and excision, this yields a surjective $\A^*_{T\times P}\,$-module homomorphism 
$$\A^{T\times G(v)}_*(\frakL[\rho_\nu])=\A^{T\times P\times G(v)}_*(\L[\rho_\nu])\twoheadrightarrow 
\A^{T\times P\times G(v)}_*(\Lambda^\circ_{W}\times G(v))=\A^{T\times P}_*(\Lambda^\circ_{W}).$$
This homomorphism maps $[\,\frakL[\rho_\nu]\,]$ to $[\Lambda^\circ_{W}]$. 
Therefore, the part (a) of the lemma follows from the next result~:

\smallskip

\begin{lemma}\label{lem:2.12} We have
$\A_*^{T\times P\times G(v)}(\L[\rho_\nu])=\A^*_{T\times P\times G(v)}\cap[\,\L[\rho_\nu]\,].$
\end{lemma}
\begin{proof}
We use the same notation as in the proof of Lemma \ref{lem:2.8}. 
Since $\L[\rho_\nu]$ is an affine space bundle over $\M[\rho_\nu]$, it is smooth.
Thus, the Poincar\'e duality yields
$$A_*^{P \times T\times G(v)}(\L[\rho_\nu]) = A^*_{P \times T\times G(v)}(\L[\rho_\nu]) \cap [\,\L[\rho_\nu]\,].$$
Moreover, we have a chain of isomorphisms
\begin{align*}
A^*_{P \times T\times G(v)}(\L[\rho_\nu]) 
&\simeq A^*_{P \times T\times G(v)}(\M[\rho_\nu]) \\
&\simeq A^*_{L \times T\times G(v)}(\M[\rho_\nu]).
\end{align*}
By Theorem~\ref{thm:Kirwangraded} the map
$$\A^*_{L \times T\times G(v)} \to \A^*_{L \times T\times G(v)}(\M[\rho_\nu])$$
is a surjective ring homomorphism. It follows that 
$$A^*_{L \times T\times G(v)}\to A_*^{L \times T\times G(v)}(\L[\rho_\nu]), \quad \alpha \mapsto \alpha \cap [\,\L[\rho_\nu]\,]$$
is surjective as well. The lemma is proved.
\end{proof}

\smallskip

Now, let us prove the part (b).
The fundamental class of the set $\Lambda_{(L)}$ in \eqref{LambdaL} is identified with the monomial 
$\x_{\nu_1}\otimes\dots\otimes\x_{\nu_r}$ under the isomorphism \eqref{isom1}.
Thus, from the equality \eqref{proper} we deduce that
\begin{align}\label{EQ0}
\x_{\nu_1}\star\dots\star\x_{\nu_r}
&=c_*p_*(q_3,q_4)^!([\Lambda_{(L)}]).
\end{align}
In particular, the element $\x_{\nu_1}\star\dots\star\x_{\nu_r}$ in $\Y^\flat(v,0)$ 
is supported on the closed subset $\Lambda_\nu$ of $\Lambda(v)$.
Since $\Lambda_\nu$ is irreducible by Proposition \ref{prop:geomirrlambda}, we deduce that 
this class is a rational multiple of $[\Lambda_\nu]$.
To prove that it is equal to $[\Lambda_\nu]$,  it is enough to check that 
\begin{align}\label{ID}(q_3,q_4)^!([\Lambda_{(L)}])=[\Lambda_W].\end{align}

\smallskip

To prove this, recall first that $\Lambda_W$ and $\Lambda_{(L)}$ are both irreducible.
Then, applying recursively the Lemma \ref{L:CBext}, we deduce that  the map $q_4:\Lambda_W\to\Lambda_{(L)}$ 
restricts to a smooth map from a dense open subset of $\Lambda_W$ 
onto a dense open subset of $\Lambda_{(L)}$. Next, a direct computation yields
$$\dim(u\times \R_L)=2q\sum_s\nu_s^2+\sum_{s<t}\nu_s\nu_t,\quad
 \dim \R_P=2q(\sum_s\nu_s^2+\sum_{s<t}\nu_s\nu_t),$$
$$ \dim\Lambda_{(L)}=q\sum_s\nu_s^2,\quad
 \dim\Lambda_W=q\sum_s\nu_s^2+(2q-1)\sum_{s<t}\nu_s\nu_t.$$
Hence $q_3$ and $q_4$ have the same relative dimension.
This proves the relation \eqref{ID}.

\end{proof}

\smallskip

\begin{proof}[Proof of Lemma \ref{lem:2.9}]
Part (a) can be proved as in Lemma \ref{lem:2.8}(a).
Let us give a more direct argument. 
Recall that $\Lambda_W^\circ$ is a vector bundle over the $P$-orbit $\fraku^\circ$. 
Hence we have the chain of maps
$$\xymatrix{A^*_{T\times P}\cap[\fraku]=A_*^{T\times P}(\fraku)\ar[r]^-{a}&
A_*^{T\times P}(\fraku^\circ)\ar[r]^-b&A_*^{T\times P}(\Lambda^\circ_W),}$$
where $a$ is given by excision, hence it is surjective, and $b$ is the pullback by the bundle $\Lambda_W^\circ\to\fraku^\circ$,
hence it is an isomorphism. The part (a) follows.

\smallskip

To prove (b), note that the equality \eqref{EQ0} implies that $\x_{\nu_{1}}\star\dots\star\x_{\nu_{r}}$
is supported on the set 
$\Lambda_{\lescc\,\nu}^\circ$. 
Hence a degree argument implies that
$$\x_{\nu_{1}}\star\dots\star\x_{\nu_{r}}\in \bbQ\cdot[\overline{\Lambda_{\nu}^\circ}]+
\sum_{\mu\lescc\nu}\bbQ\cdot[\overline{\Lambda^\circ_\mu}].$$
We must compute the first coefficient.
To do that, we factorize the maps $q_3$ and $q_4$ as in the following fiber diagram
\begin{align}\label{DIAG'}
\begin{split}
\xymatrix{
(\fraku\times \R_L)/P&(\fraku\times \R_P)/P\ar[l]_-{q'_3}&\R_P/P\ar[l]_-{q''_3}&&\\
\Lambda_{(L)}/P\ar@{^{(}->}[u]^-{}&(\fraku\times\frakp)/P\ar@{^{(}->}[u]^-{}\ar[l]_-{q'_4}&
\Lambda_{W}/P\ar@{^{(}->}[u]^-{}\ar[l]_-{q''_4}\ar[r]^-p&\Lambda_{\nu}/G(v)\ar@{^{(}->}[r]^-c&\Lambda(v)/G(v)\\
&(\fraku^\circ\times\frakp)/P\ar@{^{(}->}[u]^-{j}&\Lambda_{W}^\circ/P\ar@{^{(}->}[u]^-{j}\ar[r]^-{p}_-\sim\ar[l]_-{q^\circ_4}&
\Lambda_{\nu}^\circ/G(v).\ar@{^{(}->}[u]^-j&
}
\end{split}
\end{align}
The maps $q'_3$, $q'_4$ are smooth of relative dimension  $2\dim(\fraku)$.
The map $q''_3$ is a regular embedding of relative dimension $-\dim(\fraku)$ given by the component in $\fraku$ 
of the commutator of elements of $\frakp$.
Fix an element $e\in\fraku^\circ$. Let $P_e$ be the stabilizer of $e$ for the adjoint action of $P$ on $\fraku$.
Let $\frakp_e$ be its Lie algebra. The map $q^\circ_4$ is given by the inclusion
$$\fraku^\circ\times\frakp=P\times_{P_e}(\{e\}\times\frakp)\supseteq P\times_{P_e}(\{e\}\times\frakp_e)=\Lambda^\circ_W.$$
Since $\frakp_e$ is the kernel of the submersion $\frakp\to\fraku$ such that $x\mapsto[e,x]$, 
we deduce that the map $q_4^\circ$ is also
a regular embedding of relative dimension $-\dim(\fraku)$.
Hence, we have
\begin{align}\label{eq1}
j^*(q_3,q_4)^!([\Lambda_{(L)}/P])=(q_4^\circ)^*([(\fraku^\circ\times\frakp)/P])=[\Lambda^\circ_W/P].
\end{align}
This finishes the proof of the lemma.

\end{proof}

\smallskip

\subsection{The algebra $\bfY^\flat$}\hfill\\

In this subsection we introduce an extension of $\Y^{\flat}$ by adding a Cartan part.

\subsubsection{Definition}
Let $\Bbbk(\infty)$ 
be Macdonald's ring of symmetric functions with coefficients in $\Bbbk$.
We view it as the free commutative $\Bbbk$-algebra generated by the power sum polynomials
$p_l$ for each integer $l>0$.
It carries a comultiplication $\Delta$, a counit $\eta$ and an antipode $S$ given by
$\Delta(p_l)=p_l\otimes 1+1\otimes p_l,$ $\eta(p_l)=0$ and $S(p_l)=-p_l.$
It is equipped with the $\bbZ$-grading such that the element $p_l$ has the degree $\varepsilon l$.
We define 
$$\bfY(0)=\Bbbk(\infty)^{\otimes I}.$$
For each vertex $i\in I$ and each integer $l>0$, let $p_{i,l}$ be the element of $\bfY(0)$ given by
$$p_{i,l}=(1\otimes\cdots\otimes 1\otimes p_l\otimes 1\otimes \cdots\otimes 1)/l!,$$
where $p_l$ is at the $i$-th spot.
For any dimension vector $v$, restricting a representation of $G(\infty^I)$ to the subgroup $G(v)$, yields a 
$\Bbbk$-algebra homomorphism 
$$\bfY(0)\to\Bbbk[v],\quad \rho\mapsto \rho_v.$$
Let $\calU_i$ be the \emph{tautological} 
$T\times G(v)$-equivariant vector bundle on $\Lambda^\flat(v)$, which is given by
\begin{align*}\begin{split}
\calU_i=\Lambda^\flat(v)\times V_i,
\end{split}\end{align*}
with the obvious $T\times G(v)$-action. 
Consider the $\bfY(0)$-action $\circledast$ on $\Y^\flat$ such that the element
$p_{i,l}$ acts via the cap product with the class $\ch_l(\calU_i)$.
We use Sweedler's notation for the coproduct. The following result is an easy consequence of the definitions and is left to the reader.

\smallskip

\begin{proposition}\label{prop:wilson} 
For each $x\in\bfY(0)$ and each $y,z\in\Y^\flat$ we have 
$$x\circledast(y\star z)=\sum(x_1\circledast y)\star(x_2\circledast z).$$
We deduce that the $\bbX$-graded $\Bbbk$-algebra $\Y^\flat$ is a $\bfY(0)$-module $\bbX$-graded algebra.
\qed
\end{proposition}

\smallskip

\begin{definition} The COHA is the smash product $\bfY^\flat$ of $\bfY(0)$ and $\Y^\flat$.
It is a free $\bbX_+$-graded $\Bbbk$-algebra, with
$\bfY^\flat(v,k)=\Y^\flat(v,k)$ for each nonzero dimension vector $v$ and each integer $k$ 
and $\bfY^\flat(0,k)$ equal to the degree $k$ part of $\bfY(0)$. We set $\bfY_K=\bfY^1 \otimes K.$
\end{definition}

\smallskip

\subsubsection{The generators of $\bfY^1$}
Let $Q$ be any quiver. Set
$$I^\re=\{i\in I\,;\, q_i=0\}, \quad I^\el=\{i\in I\,;\, q_i=1\}, \quad I^\hyp=\{i\in I\,;\, q_i>1\}.$$

\smallskip

\begin{theorem}\label{thm:gen}\hfill
\begin{itemize}[leftmargin=8mm]

\item[$\mathrm{(a)}$] 
The $\Bbbk$-algebra  $\bfY^1$ is generated by the subset
$$\bfY(0)\cup\{\x_{i,1}\,;\,i\in I^\re\}\cup\{\x_{i,l}\,;\,l>0\,,\,i\in I^\el \cup I^\hyp\}.$$
\item[$\mathrm{(b)}$] 
Assume that $M_*=H_*$, and that $T$ contains $\theta, \theta^*$. Then, the $K$-algebra  $\bfY_K$ is generated by the subset
$$(\bfY(0)\otimes K)\cup\{\x_{i,1}\otimes 1\,;\,i\in I^\re\cup I^\el\}\cup\{\x_{i,l}\otimes 1\,;\,l>0\,,\,i\in I^\hyp\}.$$
\end{itemize}
\end{theorem}

\smallskip

\begin{proof}
Part (a) follows from Propositions \ref{prop:2.2} and \ref{prop:2.6}.
To prove (b) we must check that if $Q$ is the Jordan quiver with vertex $i$ and $T$ is the torus
$\bbG_m\times\bbG_m$ generated by the cocharacters $\theta$, $\theta^*$, then the $K$-algebra
$\bfY_K$ is generated by $\bfY(0)\otimes K$ and the element $\x_{i,1}.$
To do this, let us abbreviate $\x_l=\x_{i,l}$ for each integer $l>0$ as in the proof of Proposition \ref{prop:2.6}.
We may also omit the symbol $\delta_i$ to unburden the notation.
Let $u\in\Bbbk[\delta_i]$ be the first equivariant Chern class of the linear character of $G(\delta_i)$.
For each integer $k$, consider the element $D_k=u^k\cap\x_1$ in $\Y(\delta_i)$.
To simplify the notation we omit the symbol $\delta_i$.
We must check that for each $l$ the element $\x_l$ belongs to the subalgebra $\Y_\bullet\subset\Y$ generated by
$\{D_k\,;\,k\in\bbN\}$. To do that, consider the closed embedding 
$$i\,:\,\bigsqcup_{v\in\bbN}\Lambda(v)\to \bigsqcup_{v\in\bbN}\R(v).$$
By proper base change, the pushforward gives an $\bbX_+$-graded $\Bbbk$-algebra homomorphism
\begin{align}\label{i*} i_*\,:\,\Y=\bigoplus_{v}\Y(v)\to\Z=\bigoplus_{v} \Z(v)\end{align}
where $\Z(v)=H_*^{T\times G(v)}(\R(v))$. The map $i_*$ is injective because the $\Bbbk[v]$-module
$\Y(v)$ is torsion-free by Proposition~\ref{prop:torsionfree}.
We identify $\Z(v)$ with $\Bbbk[v]$ under the map
$$\Bbbk[v]\to \Z(v),\quad\alpha\mapsto\alpha\cap[\R(v)].$$
The multiplication in $\Z$  has been computed in \cite{SV13b}. 
Let $(t,t^*)$ be the basis of $\frakt^*$ dual to $(\theta,\theta^*)$. Modulo the canonical identification
$$\Z(v)=\Bbbk[x_1,x_2,\dots,x_v]^{S_v},\quad \Bbbk=\bbQ[t,t^*],$$
this product is identified with the \emph{shuffle product} given by
\begin{align}\label{shuffle}
\begin{split}f(x_1,\dots,x_v)*g(x_1,\dots,x_w)=&\\
\text{SYM}\Big(f(x_1,\dots,x_v)*g(x_{v+1},\dots,x_{v+w})&
\prod_{i=1}^{v}\prod_{j=v+1}^{v+w}\zeta(x_i-x_j)\Big)\,\Big/v!\,w!
\end{split}
\end{align}
where SYM denotes summing over all permutations of the variables $x_1,\dots,x_{v+w}$ and $\zeta$ is the rational function
$$\zeta(x)={(z-t-t^*)(z+t)(z+t^*)\over z}.$$
Now, let $\Z_\bullet\subset\Z$ be the image of $\Y_\bullet$. To finish the proof, it remains to show that
$i_*(\x_l)\in\Z_\bullet\otimes K$ for each $l$. To do that, we may use \cite[thm.~2.8]{Ne15} which implies that 
$\Z_\bullet\otimes K$ consists of the symmetric polynomials 
over $K$ which satisfy the \emph{wheel conditions} in \cite[def.~2.5]{Ne15}. 
Indeed, since $\x_l$ is the fundamental class of $\Lambda_{(l\delta_i)}=\{0\}\times\frakgl(l)$,
 the element $i_*(\x_l)$ coincides with the 
polynomial 
$$\prod_{1\leqslant i,j\leqslant l}(t+x_i-x_j),$$
hence it satisfies the wheel conditions. 
\end{proof}

\smallskip

\subsection{The representation of $\bfY^\flat$ in the homology of quiver varieties}\hfill\\

Now, we turn to the action of $\bfY^\flat$ on the homology of quiver varieties.
We set $M_*=H_*$.
We need this to compare $\bfY^1$ with the Yangian introduced in \cite{MO}, see \cite{SV17b}.

\smallskip

\subsubsection{The faithful representation of $\Y^\flat$}\label{sec:ffrep}
Let $F_{w}$ be the $\bbX$-graded $\Bbbk\langle w\rangle$-module given by
$$F_w=\bigoplus_{v\in\bbN^I}F_w(v),\quad F_w(v)=\bigoplus_{k\in\bbZ}F_w(v,k),\quad
F_w(v,k)=H_{k+2 d_{v,w}}^{T\times H(w)}\big(\frakM(v,w)\big).$$
Set $|w\rangle$ equal to the fundamental class $[\frakM(0,w)]$ in $F_w.$

\smallskip

\begin{proposition}\label{prop:A} Let $\flat$ be either $0$ or $1$.\hfill
\begin{itemize}[leftmargin=8mm]

\item[$\mathrm{(a)}$] The $\bbX$-graded $\Bbbk$-algebra $\Y^\flat$ acts on the 
$\bbX$-graded $\Bbbk\langle w\rangle$-module $F_w$.

\item[$\mathrm{(b)}$] If $H(v)=G(v)$ then the action 
on the element $|v\rangle$ yields an injective map $\Y^\flat(v)\to F_{v}(v).$ 
In particular, the representation of $\Y^\flat$ in $\bigoplus_wF_w$ is faithful.

\item[$\mathrm{(c)}$] The action of
$\Y^\flat(v)$ on the element $|w\rangle$ yields a $\Bbbk\langle w\rangle$-linear map 
\begin{align*}\Y^\flat(v)\otimes\Bbbk\langle w\rangle\to F_{w}(v).\end{align*}
Its image is the pushward of $H_*^{T\times H(w)}(\frakL^\flat (v,w))$ by the inclusion
$\frakL^\flat (v,w)\subset\frakM(v,w)$.
\end{itemize}
\end{proposition}

\smallskip

\begin{proof}
The proof of (a) is similar to the proof given in \cite[\S 7.6]{SV13a}, \cite[\S 6]{SV13b} in the particular case where $|I|=1$.
See also \cite{YZ14}. 
Let $L$, $P$, $U$, $\fraku$, $v_1$, $v_2$  be as in \eqref{f5}, \eqref{f5b}.
Let $\R_P$ and $\R_{s,P}$ be as in \eqref{R1}. Set
\begin{align}\label{RL}
\begin{split}
\R_L&=\R(v_1,w)\times \,R(v_2),\\
\R_{s,L}&=\R_s(v_1,w)\times\, R(v_2),\\
\M_{s,L}&=\M_s(v_1,w)\times\, \Lambda^\flat(v_2).
\end{split}
\end{align}
For each pair $(\bar x,\bar a)$ in $\R_P$  let $\bar x_1$, $\bar x_2$, $\phi_1$, $\phi_2$ be as in \eqref{seq1}, \eqref{seq2} 
and define
$$a_1=a,\quad a_1^*=a^*\circ\phi_1.$$
Then, we have the following fiber diagram of $P$-varieties
\begin{align}\label{seq5}
\begin{split}
\xymatrix{
&\fraku\times\R_{s,L}&\R_{s,P}\ar[l]_-{q_3}\\
&\M_{s,L}\ar@{^{(}->}[u]^-{c}&\H(v_1,\Lambda^\flat(v_2)\,;\,w)\ar@{^{(}->}[u]^-{}\ar[l]_-{q_2}}
\end{split}
\end{align}
where the maps $c$, $q_2$ and $q_3$ are given by
\begin{align*}
c(z)&=\big(0,z),\\
q_2(\bar x,\bar a)&=(\bar x_1,\bar a_1)\oplus \bar x_2,\\
q_3(\bar x,\bar a)&=\big(\mu(\bar x,\bar a)-\mu(\bar x_1,\bar a_1)\oplus\mu(\bar x_2)\,,\,(\bar x_1,\bar a_1)\oplus \bar x_2\big).
\end{align*}
Since the map $q_3$ is an l.c.i.~morphism, 
the refined pullback $(q_3,q_2)^!$ is well-defined.
Consider the obvious maps
$$p\,:\,\frakh[v_1,\Lambda^\flat(v_2)\,;\,w]\to \frakM(v,w),\quad q_1\,:\,\M_{s,L}/P\to \M_{s,L}/L,$$
where $p$ is as in Proposition \ref{prop:hecke}.
As in
\eqref{star}, the composition $p_*\circ (q_2,q_3)^!\circ q_1^*$
defines a $\Bbbk$-linear map
\begin{align}\label{action3}
H_{k+2d_{v_2}}^{T\times G(v_2)}(\Lambda^\flat(v_2))\otimes H_{l+d_{v_1,w}}^{T\times H(w)}(\frakM(v_1,w))\to  
H_{k+l+d_{v,w}}^{T\times H(w)}(\frakM(v,w)).\end{align}
Taking the sum over all  $v_1, v_2\in \bbN^I$ yields an $\bbX$-graded $\Bbbk$-module homomorphism
\begin{align*}
\star\,:\,\Y^\flat\,\otimes\, F_{w}\to F_{w}.
\end{align*}
The proof that this map gives indeed a representation of $\Y^\flat$ on $F_w$ is the same as in the case of the algebra $\Y$ and
the Jordan quiver in
\cite{SV13a}, \cite{SV13b}, see, e.g., \cite{YZ14}.

\smallskip

Now, we prove part (b). Set $v=v_2$ and $v_1=0$ in the diagram \eqref{seq5}. Set 
\begin{align*}
\L_s(v,w)&=\R_s(v,w)\cap\,\big(\R(v)\times\{0\}\times\Hom_I(v,w)\big).
\end{align*}
We get the following fiber diagram of $G(v)$-varieties
\begin{align*}
\begin{split}
\xymatrix{
\R(v)&\L_s(v,w)\ar[l]_-{q_3}\\
\Lambda^\flat(v)\ar@{^{(}->}[u]&\L_s^\flat(v,w)\ar@{^{(}->}[u]^-{}\ar[l]_-{q_2}\ar[r]^-p&\M_s(v,w),
}
\end{split}
\end{align*}
where $p$ and the vertical maps are the obvious closed immersions, while
$q_2$ and $q_3$ are the projections on the first factors.
The maps $q_2$ and $q_3$  are l.c.i.~morphisms of the same relative dimension.
So, we have $(q_3,q_2)^!=q_2^*$ by the excess intersection formula. 
As in \S\ref{sec:LL}, we define the maps
\begin{align*}
\beta&:\L^\flat_{s}(v,w)\to\Lambda^\flat(v)\times\Hom_I(v,w),\\
\rho_1&:T\times G(v)\times H(w)\to T\times G(v),\\
p_1&:\Lambda^\flat(v)\times\Hom_I(v,w)\to\Lambda^\flat(v).
\end{align*}
Thus, we have $q_2=p_1\circ \beta$. Further, the map 
\begin{align}\label{M1}\Y^\flat(v)\to F_{v}(w),\quad x\mapsto x\star |w\rangle\end{align} 
coincides with the composed map
$$\xymatrix{H_*^{T\times G(v)}(\Lambda^\flat(v))\ar[r]^-{q_2^*\circ\rho_1^\sharp}&
H_*^{T\times G(v)\times H(w)}(\L_s^\flat(v,w))\ar[r]^-{p_*}&
H_*^{T\times G(v)\times H(w)}(\M_s(v,w)).}$$
Now, set $w=v$ and $H(v)=G(v)$. Then, the map
$q_2^*\circ\rho_1^\sharp$
is injective by Lemma \ref{lem:inj-surj}.
Finally, by Proposition \ref{prop:L}, the pushforward by $p$ yields an injective map
$$p_*\,:\,H_*^{T\times G(v)}(\frakL^\flat(v,v))\to H_*^{T\times  G(v)}(\frakM(v,v))$$
because
the left hand side is torsion free by loc.~cit.~and $p_*$ induces an isomorphism after 
base change from $\Bbbk[v]$ to its fraction field by the localization theorem.

\smallskip

Finally, to prove (c) note that 
by  Proposition \ref{prop:L}, the cycle map gives an isomorphism
$$\cl:A_{k+d_{v,w}}^{T\times H(w)}\big(\frakM(v,w)\big)\to H_{2k+2 d_{v,w}}^{T\times H(w)}\big(\frakM(v,w)\big)=
F_w(v,2k).$$
Since the $\Bbbk$-algebra $\Y^\flat_A$ acts on the $\Bbbk\langle w\rangle$-module
$$\bigoplus_v\bigoplus_kA_{k+d_{v,w}}^{T\times H(w)}\big(\frakM(v,w)\big)$$
by convolution, 
the map \eqref{M1} fits into the fiber diagram
\begin{align*}
\xymatrix{
\Y^\flat_H(v)\otimes \Bbbk\langle w\rangle\ar[r]^-{q_2^*}&H_*^{T\times H(w)}(\frakL^\flat(v,w))\ar[r]^-{p_*}&F_v(w)\\
\Y^\flat_A(v)\otimes \Bbbk\langle w\rangle\ar[r]^-{q_2^*}\ar[u]_-\cl&A_*^{T\times H(w)}(\frakL^\flat(v,w)).\ar[u]_-\cl&
}
\end{align*}
Since $q_2$ is the composition of the vector bundle $p_1$ and the open immersion $\beta$, the lower map is surjective by excision.
The right vertical map is surjective by Proposition \ref{prop:L}.
This finishes the proof of the proposition.
\end{proof}

\smallskip

Now, for any $i$, $l$ we compute the action of  the element $\x_{i,l}$ on  $F_w$. 
To do so, assume that $v_2=l\,\delta_i.$
Then, we consider the subvariety of $\frakh[v_1,v_2\,;\,w]$ given by
\begin{align*}
q_i>0,\quad l>0 &\,\Rightarrow\,
\frakC(v_1,v_2\,;\,w)=\frakh[v_1,\Lambda_{(v_2)}\,;\,w],\\
q_i=0,\quad l=1&\,\Rightarrow \,
\frakC(v_1,v_2\,;\,w)=\frakh[v_1,v_2\,;\,w],\\
q_i=0,\quad l>1&\,\Rightarrow\, \frakC(v_1,v_2\,;\,w)=\emptyset.
\end{align*}

\smallskip

Let $\frakM(w)$ be the disjoint sum of all $\frakM(v,w)$'s and $\frakC_{i,l}(w)$ be the disjoint sum of all
$\frakC(v_1,v_2\,;\,w)$'s where $v,v_1,v_2$ are as above.
We view $\frakC_{i,l}(w)$ as a closed $T\times G(w)$-invariant subvariety of  
the symplectic manifold $\frakM(w)^{\op}\times\frakM(w)$ which is proper over the first factor. 
Hence, it acts by convolution on $F_w$.
We use the same symbol to denote $\frakC_{i,l}(w)$ and the convolution by this correspondence,
which is a $\Bbbk\langle w\rangle$-linear operator on $F_w$.

\smallskip

\begin{proposition}\label{prop:B} For each $i$, $l$, $w$ the following hold
\begin{itemize}[leftmargin=8mm]
\item[$\mathrm{(a)}$]
$\frakC_{i,l}(w)$ is either empty or a Lagrangian local complete intersection in $\frakM(w)^{\op}\times\frakM(w)$.
\item[$\mathrm{(b)}$]
If either $q_i>1$ or $q_i=0$ and $l=1$ then $\frakC_{i,l}(w)$ is irreducible.
\item[$\mathrm{(c)}$]
$\x_{i,l}$ acts on $F_w$ via the operator $\frakC_{i,l}(w)$.
 \end{itemize}
\end{proposition}

\smallskip

\begin{proof} If $q_i=0$ then 
$\frakC(v_1,v_2\,;\,w)$ coincides with the Hecke correspondence in \cite{N98} (for the case of quiver without loops)
and all statements are straightforwards. If $q_i>0$ then parts (a), (b) follow from Propositions \ref{prop:LAG}, \ref{prop:irrheckenu}.

\smallskip

Now, we prove (c).
Let $\R_{s,L}$, $\R_{s,P}$ and $q_2$, $q_3$ be as in \eqref{RL}, \eqref{seq5}.
We assume that $q_i>0$, the proof in the case $q_i=0$ is very similar.
We have the Cartesian square
\begin{align*}
\xymatrix{
\fraku\times\R_{s,L}&\R_{s,P}\ar[l]_-{q_3}\\
\M_s(v_1,w)\times \Lambda_{(v_2)}\ar@{^{(}->}[u]^-{c}&
\H[v_1,\Lambda_{(v_2)}\,;\,w]\ar@{^{(}->}[u]^-{}\ar[l]_-{q_2}.}
\end{align*}
The map $q_3$ is an l.c.i.~morphism of relative dimension $v_1\cdot v_2-(v_1,v_2)+w\cdot v_2$,
because $\fraku\times\R_{s,L}$ and $\R_{s,P}$ are smooth.
Since $\frakC(v_1,v_2\,;\,w)$ is a local complete intersection, so is also the variety $\H[v_1,\Lambda_{(v_2)}\,;\,w]$ by \eqref{heckeB}.
Hence, the map $q_2$ is also an l.c.i.~morphism
of relative dimension $v_1\cdot v_2-(v_1,v_2)+w\cdot v_2$.
We deduce that the pullback morphism $q_2^*$ is well-defined and that $(q_3,q_2)^!=q_2^*$ 
by the excess intersection formula,
proving the claim.
\end{proof}

\smallskip

\begin{corollary}
If $i\in I^\hyp$ and $\nu\vDash v_2$, then the element
$[\Lambda_\nu]$ in $\Y^\flat(v_2)$ acts on $F_w$ by convolution with the correspondence $\frakh(\Lambda_\nu)$.
\end{corollary}

\smallskip

\begin{proof}
Set $\nu=(\nu_1,\dots,\nu_r)$.
By Lemma \ref{lem:2.8} we have $[\Lambda_\nu]=\x_{\nu_1}\star\cdots\star\x_{\nu_r}$ in $\Y^\flat(v_2)$.
Hence, the claim follows from Propositions \ref{prop:comp}, \ref{prop:B}.

\end{proof}

\smallskip

\subsubsection{The geometric realization of $\bfY^1$ via Hecke correspondences}\label{sec:geomY1}
For each dimension vector $w$,
let $A_w$ be the $\bbX$-graded $\Bbbk\langle w\rangle$-algebra 
given by 
\begin{align}\label{A2}
A_w&=\bigoplus_{v_2\in\bbZ^I} A_{w}(v_2)
=\bigoplus_{v_2\in\bbZ^I}\bigoplus_{k_2\in\bbZ} A_{w}(v_2,k_2)\subset\End_{\Bbbk\langle w\rangle}(F_w),
\end{align}
where $A_{w}(v_2,k_2)$ consists of all $\Bbbk\langle w\rangle$-linear endomorphisms of $F_w$ which are homogeneous of degree $(v_2,k_2)$. Under the convolution, any cycle in $M_*^{T\times H(w)}\big(\frakM(u,w)\times\frakM(v,w)\big)$
which is proper over $\frakM(u,w)$ gives rise to an operator
$F_w(v)\to F_w(u),$ see \cite{CG} for details. Therefore, for each $i\in I$ and $l\in\bbZ_{>0}$ the familly of correspondences $\frakC_{i,l}(w)$ defines an element 
$$\frakC_{i,l}\in\prod_{w}A_{w}.$$
Let $\calV_i$ be the \emph{universal} 
$T\times H(w)$-equivariant vector bundles on $\frakM(w)$, given by
\begin{align}\label{bundle}\begin{split}
\calV_i=\M_s(v,w)\times_{G(v)} V_i.
\end{split}\end{align}
Let 
$\psi_{i,l}$ be the operators in $\prod_wA_w$ given by the cap product 
with the $T$-equivariant cohomology class $\ch_l(\calV_i)$.

\smallskip

\begin{theorem}\label{cor:C}
Set $H(w)=G(w)$ for all $w$'s.
Then, the $\Bbbk$-algebra $\bfY^1$ is isomorphic to the $\Bbbk$-subalgebra of $\prod_{w}A_{w}$ generated by
$$\{\psi_{i,l}\,;\,l>0\,,\,i\in I\}\cup\{\frakC_{i,1}\,;\,i\in I^\re\}\cup\{\frakC_{i,l}\,;\,l>0\,,\,i\in I^\el \cup I^\hyp\}.$$
\end{theorem}

\smallskip

\begin{proof}
We extend the $\Y^1$-action on $F_w$ to a $\bfY^1$-action taking the element
$p_{i,l}\in\bfY(0)$ to $\psi_{i,l}$.
Then, the theorem follows from Theorem \ref{thm:gen} and Propositions \ref{prop:A}, \ref{prop:B},
because $H(w)=G(w)$.

\end{proof}

\vfill

\newpage


\appendix

\section{}

\medskip

\subsection{Proof of Proposition \ref{P:CYpreproj}}\label{sec:6.1} \hfill\\

The proposition is proved in the literature under the assumption that $Q$ does not carry oriented cycles (or edge loops). The following 
argument which works in full generality was explained to us by W. Crawley-Boevey.  
Let us begin by sketching the proof that $\Pi$ is of homological dimension two. We may and will restrict ourselves to the case of a 
connected quiver. 
For any vertex $i \in I$ we denote by $e_i$ the corresponding idempotent of $\Pi$. 
Consider the complex of $(\Pi,\Pi)$-bimodules
\begin{equation}\label{E:CBarg}
\begin{split}
\xymatrix{ 0 \ar[r] & \bigoplus_i \Pi e_i \otimes e_i \Pi \ar[r]^-{f} & \bigoplus_{h \in \bar{\Omega}} \Pi e_{h''} \otimes e_{h'} \Pi \ar[d]^-{g}\\
0& \Pi \ar[l]&\bigoplus_i \Pi e_i \otimes e_i \Pi \ar[l]^-{k}}
\end{split}
\end{equation}
where all the tensor products are taken over $\mathbb{C}$, and where the maps are defined as follows~: $k$ is the multiplication map,
$g(ae_{h''} \otimes e_{h'}b)=ahe_{h'} \otimes e_{h'}b - ae_{h''}\otimes e_{h''}hb$ and 
$$f(ae_i \otimes e_ib)=\sum_{h \in \bar{\Omega}, h''=i} \epsilon_h ae_{(h^*)'} \otimes 
e_{(h^*)''}h^*b - \sum_{h \in \bar{\Omega}, h'=i} \epsilon_h ah^*e_{h''} \otimes e_{h'}b$$
where $\epsilon_h = 1$ if $h\in \Omega$ and $\epsilon_h=-1$ if $h \in \Omega^*$. The fact that (\ref{E:CBarg}) is a complex is a direct 
consequence of the defining relations of $\Pi$. We claim that it is in addition exact. This exactness everywhere except for the leftmost 
term may be checked using standard arguments. 
Recall that we have assumed that $Q$ is not of finite Dynkin type.
Equipping $\Pi$ with the $\mathbb{N}$-grading obtained by assigning degree $1$ to any edge $h \in \bar{\Omega}$ we then have that 
$\Pi$ is a Koszul algebra and that its Hilbert series
$$H(t)=(h_{i,j}(t))_{i,j \in I}, \qquad h_{i,j}(t)=\sum_{l \geqslant 0} \dim(e_j \Pi[l]e_i)\,t^l$$
is equal to
$H_Q(t)=(\Id + t \,(\bfQ+ {}^t\bfQ) + t^2\,\Id)^{-1}$ where $\bfQ$ is the adjacency matrix of $Q$, see \cite{MOV}, \cite{EE}. Observing that
$f,g,k$ are of respective degrees $1,1,0$ we deduce that 
$$\dim(\Ker(f))=-H(t)+H(t)^2 -tH(t)(\bfQ + {}^t\bfQ)H(t) +t^2H(t)^2=0$$
as wanted. This proves that $\Pi$ is of homological dimension two. 
Next, let $M,$ $N$ be any finite dimensional $\Pi$-modules. Tensoring (\ref{E:CBarg}) by $M$ yields the projective resolution of $M$
\begin{equation}\label{E:CBargMod}
\xymatrix{ 0 \ar[r] & \bigoplus_i \Pi e_i \otimes e_i M \ar[r]^-{f} & \bigoplus_{h \in \bar{\Omega}} \Pi e_{h''} \otimes e_{h'} M \ar[r]^-{g} &
\bigoplus_i \Pi e_i \otimes e_i M \ar[r]^-{k} & M \ar[r] & 0.}
\end{equation}
Applying the functor $\Hom(N,\cdot)$ we obtain the following complex computing $\Ext^\bullet(N,M)$~:
$$\xymatrix{0 \ar[r] &\bigoplus_i  (e_iN)^* \otimes e_i M \ar[r]^-{g'} & \bigoplus_{h \in \bar{\Omega}}  (e_{h''}N)^* \otimes e_{h'} M \ar[r]^-{f'} &
\bigoplus_i  (e_iN)^* \otimes e_i M  \ar[r] &0.}$$
It remains to observe that this complex is canonically isomorphic to the dual of the same complex in which the roles of $M$ and $N$
are exchanged.

\medskip

\subsection{Proof of Lemma~\ref{L:muirr}}\hfill\\

Being the fiber of the moment map $\frakgl(v)^{2q} \to \fraksl(v)$, the set $\M(v)$ has all its irreducible components of dimension
at least $(2q-1)v^2+1$. Note also that $\M(v)$ is defined over $\mathbb{Z}$ and may be therefore reduced to
any finite field. By the Lang-Weil theorem the irreducibility of $\M(v)$ will thus be a consequence of the estimate
as $l\to\infty$
\begin{equation}\label{E:langweil}
 \#\M(v)(\mathbb{F}_l) \sim  l^{(2q-1)v^2+1} .
\end{equation}
By \cite[thm. 5.1]{Moz} the number of points of $\M(v)(\mathbb{F}_l)$ is given by the following generating series
\begin{equation}\label{E:moz}
\sum_{v \geqslant 0} \frac{\#\M(v)(\mathbb{F}_l)}{\#GL(v)(\mathbb{F}_l)}l^{\langle v, v \rangle} z^v= \text{Exp}\left( \frac{l}{l-1}\sum_v A_v(l) z^v \right)
\end{equation}
where $\text{Exp}$ is the plethystic exponential and $A_v(l)$ is the Kac polynomial, see e.g. \cite[\S 2]{BSV}.  
It is well-known
that $A_v(l)$ is a monic polynomial of degree $1-\langle v,v\rangle=1+(g-1)v^2$. The coefficient of $z^v$ in (\ref{E:moz}) reads
\begin{equation}\label{E:moz2}
\#\M(v)(\mathbb{F}_l) \frac{l^{(1-q)v^2}}{l^{v^2}\prod_{i=1}^v (1-l^{-i})}=\frac{l}{l-1}A_v(l) + \sum_{r \geqslant 2} \frac{1}{r!}\sum_{(n_i,v_i)_i}
\frac{1}{\prod_i n_i} \prod_i \frac{l^{n_i}}{l^{n_i}-1}A_{v_i}(l^{n_i})
\end{equation}
where the sum ranges over all $r$-tuples $((n_1,v_1), \ldots, (n_r,v_r))$ satisfying $\sum_i n_iv_i=v$. We claim that only the first term
on the right hand side of (\ref{E:moz2}) contributes to the leading term. Indeed, $\frac{l}{l-1}A_v(l) \sim l^{1+(q-1)v^2}$ while for any $r\geqslant 2$
and tuple $(n_i,v_i)_i$ we have 
$$\prod_i \frac{l^{n_i}}{l^{n_i}-1}A_{v_i}(l^{n_i}) \sim l^{\sum_i n_i(1+(q-1)v_i^2)}$$
and it is elementary to check that $\sum_i n_i(1+(q-1)v_i^2) < 1 + (q-1)v^2$ whenever $r \geqslant 2$. It now follows by taking the leading term of
(\ref{E:moz2}) that $\#\M(v)(\mathbb{F}_l) \sim l^{(2q-1)v^2+1}$ as wanted.

\smallskip

\begin{remark} One may alternatively use Beilinson and Drinfeld's notion of a \textit{very good stack}, see 
\cite[sec.~1.1.1.]{BD} and show that 
$Rep(\k Q,v)/G(v)$ is very good for any $v$ and any $Q=\mathcal{Q}(i,q)$ with $q >1$.

\end{remark}

\medskip

\subsection{Proof of Proposition~\ref{P:impi}}\hfill\\

A simple representation $S$ of $\tilde{\Pi}$ will be called rigid
if $\Ext^1(S,S)=\{0\}$. Since 
\begin{equation*}
\begin{split}
(\dim S,\dim S) &= \dim(\Hom(S,S)) - \dim(\Ext^1(S,S)) + \dim(\Ext^2(S,S))\\
&=2-\dim(\Ext^1(S,S))
\end{split}
\end{equation*}
a simple representation $S$ is rigid if and only if $(\dim S,\dim S)=2$. Alternatively, a representation $S$ is rigid if and only if every other simple representation of the same dimension is isomorphic to $S$. On the other hand, if $S$ is non rigid then there are infinitely many non-isomorphic simple representations of the same dimension (recall that we work over the field $\k=\mathbb{C}$), and we have $(\dim S,\dim S) \leqslant 0$ (note that $(u,u') \in 2\mathbb{Z}$ for any dimension vectors $u,u'$).
For any dimension type $d$ there exists a unique representation type $\tau=\tau(d)$ of dimension type $d$ in which 
all non rigid simple representations occur with multiplicity one. The stratum $\frakM_0(\tau)$ is open and dense in 
$\frakM_0(\!(d)\!)$. Since $\pi$ is proper, the set $\Im(\pi)$ is closed and the proposition will be proved once we show
 that for any $d$ we have $\frakM_0(\tau) \subset \Im(\pi)$ as soon as there exists $\tau'$ of dimension type $d$ 
 such that $\frakM_0(\tau') \subset \Im(\pi)$.  By definition, the existence of such a $\tau'$ means that there exists a stable 
representation $M \in \M_s(v,w)$ of semisimple type $\tau'$. In other words, there exists a collection of simple 
$\tilde{\Pi}$-modules $S_1, S_2, \ldots, S_r$ such that $\#\{i\;;\; \dim S_i=u\}=d_u$ for any $u$, and an iterated extension
\begin{equation}\label{E:extension}
M_1 \subset M_2 \subset \cdots \subset M_r, \qquad M_i / M_{i-1} \simeq S_{i}
\end{equation}
which is \textit{stable}. Note that we automatically have $\dim S_1 \in \mathbb{N}^I \times \{w\}$ and 
$\dim S_i \in \mathbb{N}^I \times \{0\}$ for any $i >1$. Our aim is then to prove, under the above hypothesis, the existence of a similar iterated extension $N_1 \subset N_2 \subset \cdots \subset N_r$ which is stable, for which
$T_i:= N_i/N_{i-1}$ is simple, satisfies $\dim T_i=\dim S_i$ and in which all non-rigid simple factors $T_i$'s are \textit{non isomorphic}.
We will use the following simple observation, which might be of independent interest.

\smallskip

\begin{lemma}\label{L:stable} Let $N_1 \subset N_2 \subset \cdots \subset N_r$ be a filtration with simple factors $T_i:= N_i/N_{i-1}$, such that $\dim T_1 \not\in \mathbb{N}^I \times \{0\}$ while 
$\dim T_i \in \mathbb{N}^I \times\{0\}$ for any $i >1$. Let $\xi_i \in \Ext^1_{\tilde{\Pi}}(T_i,N_{i-1})$ be the associated elements. Then $N_r$ is stable if and only if $\xi_i \neq 0$ for all $i$.
\end{lemma}
\begin{proof} The module $N_r$ is stable if and only if it contains no submodule whose dimension vector belongs to $\mathbb{N}^I \times \{0\}$, and hence if and only if $\Hom_{\tilde{\Pi}}(T_i, N_j)=\{0\}$ for all $i,j$. This in turn is equivalent to the condition that $\Hom_{\tilde{\Pi}}(T_i, N_i) =\{0\}$ for all $i$.
From the exact sequence
$$\xymatrix{ \Hom(T_i,N_{i-1}) \ar[r] & \Hom(T_i,N_i) \ar[r] & \Hom(T_i,T_i) \ar[r]^-{\partial} & \Ext^1(T_i,N_{i-1})}$$
and the fact that $\xi_i=\partial(\id)$ we see that $N_{i}$ is stable if and only if $\xi_i \neq 0$ and $N_{i-1}$ is stable. The claim follows.
\end{proof}

\smallskip

Let us fix some $w$ and prove the Proposition by induction on $v$. Let us also fix a dimension type $d$ such that 
$\sum_u d_u =(v,w)$. Let us assume that Proposition~\ref{P:impi} is proved for $(w,v)$ for any dimension vector 
$v' <v$. Let $\tau'$ be a representation type of dimension type $\underline{d}$, and of dimension $(v,w)$. 
Let $M_1 \subset M_2 \subset \cdots \subset M_r$ be a filtration as above of representation type $\tau'$, and set $S_i=M_i/M_{i-1}$. If $r=1$ there is nothing to prove, so we assume that $r >1$. There are two cases to consider~:

%

\smallskip

\noindent
\emph{$\mathbf{Case\ 1 :}$  $S_r$ is not rigid}.
By our induction hypothesis, there exists a filtration $N_1 \subset N_2 \subset\cdots \subset N_{r-1}$ such that $T_i:=N_i/N_{i-1}$ is a simple $\tilde{\Pi}$-module of dimension $\dim S_i$, $N_{r-1}$ is stable and all non-rigid simple factors $T_j$ are non isomorphic. Let us choose a simple object $T_r$ of the same dimension as $S_i$, but non-isomorphic to all the $T_j$ for $j <i$. 
This is possible since there are infinitely many non-isomorphic simples of dimension $\dim S_r$. We have to show that $\Ext^1(T_r, N_{r-1}) \neq \{0\}$. Let $i_0 \leqslant i$ be the smallest index for which $\dim T_{i_0} =\dim S_r$.
Since $(\dim T_r,\dim T_r) \leqslant 0$ and $(\dim T_r,\dim S) = -\dim\Ext^1(T_r,S) \leqslant 0$ for any $S$ which is not isomorphic to $T_r$, we have 
$$(\dim T_r,\dim N_{i_0-1}) <0,\quad 
(\dim T_r,\dim(N_{r-1}/N_{i_0-1})) \leqslant 0,$$ therefore $(\dim T_r,\dim N_{i-1}) <0$ and hence
$\Ext^1(T_r,N_{i-1}) \neq \{0\}$. We choose any nonzero $\xi \in \Ext^1(T_r,N_{i-1})$ to define $N_r$. By Lemma~\ref{L:stable} the representation $N_r$ is stable, and by construction it is of representation type $\tau$.

\smallskip

\noindent
\emph{$\mathbf{Case\ 2 :}$  $S_r$ is rigid}.
Set $n= \dim\Hom(M_r,S_r)$. We claim that there exists a filtration 
$M'_1 \subset M'_2 \subset \cdots \subset M'_{r-1} \subset M'_{r}=M_r$ of $M_{r}$ for which $M'_{r-j}/M'_{r-j-1} \simeq S_r$ for $j=0, \ldots, n-1$ and 
$$\Hom(M'_{r-n},S_r)=\{0\}, \qquad \dim \Ext^1(S_r, M'_{r-n}) \geq n.$$
 Indeed, let $M$ be any $\tilde{\Pi}$-module satisfying $\Hom(S_r,M)=\{0\}$ and $\Hom(M,S_r) \neq \{0\}$. Applying $\Hom(\cdot, S_r)$ to any exact sequence $\xymatrix{0 \ar[r] &K \ar[r] & M \ar[r] & S_r \ar[r] &0}$ yields a short exact sequence
$$\xymatrix{0 \ar[r] &\mathbb{C} \ar[r] & \Hom(M, S_r) \ar[r] & \Hom(K,S_r) \ar[r] &0}$$
from which we get $\dim\Hom(K,S_r)=\dim\Hom(M,S_r) -1$. Similarly, applying $\Hom(S_r, \bullet)$ 
we get the short exact sequence
$$\xymatrix{0 \ar[r] &\mathbb{C} \ar[r] & \Hom(M, S_r) \ar[r] & \Hom(K,S_r) \ar[r] &0}$$
from which we deduce that $\dim\Ext^1(S_r,K)=\dim\Ext^1(S_r,M) +1.$
Iterating this process $n$ times starting with $M$ we get a filtration of $M_r$ with the required properties. The representation type $\nu'$ of $M'_{r-n}$ is obtained from $\tau'$ by decreasing the multiplicity of
$S_r$ by $n$. By the induction hypothesis, there exists a stable representation $N_{r-n}$ of type $\nu$, obtained from $\tau$ by decreasing the multiplicity of $S_k$ by $n$. Now, we have
\begin{equation*}
\begin{split}
(S_r&,N_{r-n})\\
&=(S_r,M'_{r-n})=\dim\Hom(S_r,M'_{r-n})-\dim\Ext^1(M'_{r-n},S_r)+\dim\Hom(M'_{r-n},S_r)\\
&=-\dim\Ext^1(S_r,M'_{r-n}) \leqslant -n
\end{split}
\end{equation*}
hence $\dim\Ext^1(S_r,N_{r-n}) \geqslant n$. Reversing the process above, we may now consider $n$ successive non split extensions of $N_{r-n}$ by $S_r$. The resulting module $N_r$ is stable by Lemma~\ref{L:stable} and is of type $\tau$ as wanted.

\smallskip

The induction step is completed, and Proposition~\ref{P:impi} is proved.

\medskip

\subsection{Kirwan surjectivity for local quiver varieties}\label{sec:APPL}\hfill\\

In this appendix, we provide an independent proof of the Kirwan surjectivity for the fixed point quiver varieties used in this paper. 

\subsubsection{Local quiver varieties} \label{sec:local}
We define a new quiver $Q^\#=(I^\#,\Omega^\#)$ defined as follows :
\begin{itemize}
\item $I^\#=\{(i,l)\}=I\times\bbZ$,
\item $\Omega^\#=\{(h,l):(h',l)\to (h'',l-1)\,;\,h\in\Omega\}\sqcup
\{(h,l):(h',l)\to (h'',l)\,;\,h\in\Omega^*\}$.
\end{itemize}
For each dimension vectors $v$,  $w$ of $Q^\#$ we define 
$$\R(v,w)^\#=Rep(\k Q^\#\,,v)\oplus\Hom_{I^\#}(W,V[-1])\oplus\Hom_{I^\#}(V,W),$$
where we abbreviate 
$V=\k^v$, $W=\k^w$ and $(V[-1])_{i,l}=V_{i,l-1}$.
The group $G(v)\times G(w)$ acts in the obvious way on $\R(v,w)^\#$. 
The map
$$\mu_\#\,:\,\R(v,w)^\#\to\Hom_{I^\#}(V,V[-1]),\quad(\bar x,\bar a)\mapsto[x,x^*]+aa^*$$
is $G(v)\times G(w)$-equivariant. Let 
\begin{align*}
\M(v,w)^{\#}&=\mu_\#^{-1}(0),\\
\frakM_0(v,w)^{\#}&=\M(v,w)^{\#}/\hspace{-.05in}/ G(v).
\end{align*}
The affine variety $\frakM_0(v,w)^{\#}$ carries an action of the group $T\times G(w)$.
Let $s$ be the character of $G(v)$ defined by 
$$s(g)=\prod_{i,l}\det(g_{i,l})^{-1}.$$
The \emph{local quiver variety} associated with $Q$, $v$, $w$ 
is the quasiprojective variety given by
$$\frakM_s(v,w)^\#=\text{Proj}\Big(\bigoplus_{n\in\bbN}\k[\M(v,w)^{\#}]^{\,s^n}\Bigr).$$
We may abbreviate $\frakM(v,w)^\#=\frakM_s(v,w)^\#$.
There is a natural projective morphism 
$$\pi\,:\,\frakM(v,w)^\# \to \frakM_0(v,w)^{\#}.$$
A representation $(\bar x,\bar a)$ in $\M(v,w)^{\#}$ is said to be \emph{semistable} 
if and only if it does not admit any nonzero subrepresentation
whose dimension vector belongs to $\bbN^{I^\#}\times\{0\}$.
Let $$\M_s(v,w)^{\#}\subset\M(v,w)^{\#}$$ be the open subset consisting of the semistable points. 
This set is preserved by the $T\times G(v)\times G(w)$-action on $\R(v,w)^\#$.
The same arguments as in \cite{N94}, \cite{N04} implies that $\frakM(v,w)^\#$ is the geometric quotient of $\M_s(v,w)^{\#}$ 
by $G(v)$ and is smooth quasi-projective with an action of $T\times G(w)$. 
If $Q^{\#}$ (or equivalently $Q$) has no oriented cycles then $\frakM_0(v,w)^{\#}$ is reduced to a single point and 
thus $\frakM(v,w)^\#$ is projective. However, this is not the case in general.

\smallskip

The varieties $\frakM(v,w)^\#$ arise as a union of fixed point connected components 
for some suitable $\mathbb{G}_m$-actions on the
 Nakajima quiver varieties attached to the quiver $Q$, see~\S \ref{sec:BBQ}. More precisely, fix dimension vectors 
 $u,$ $y$ in $\bbN^I$ and consider the $\circ$-action on  $\frakM(u,y)$ in \eqref{action1}.
We have
\begin{align}\label{decomposition1}
\frakM(u,y)^\circ = \bigsqcup_{v,w} \frakM(v,w)^\#
\end{align}
where the union ranges over the elements $v,$ $w$ in $\bbN^{I^{\#}}$ such that 
$$\sum_l v_{i,l}=u_i,\quad w_{i,l}=\delta_{l,0}\,y_i.$$
This decomposition is compatible with the $T$-action.

\medskip

\subsubsection{The homology of local quiver varieties}
For any vertex $(i,l)\in I^{\#}$ the universal and the tautological $T\times G(w)$-equivariant 
vector bundles over $\frakM(v,w)^\#$ are denoted by
\begin{align}\label{local-bundle}\calV_{i,l},\quad\calW_{i,l}.\end{align} 
They are defined as in \eqref{bundle}. Consider the subspace
$$M_{v,w}^\#\subseteq M_*^{T\times G(w)}(\frakM(v,w)^\#)$$ 
generated by the action of the Chern character components
$\ch_k(\calV_{i,l})$ and $\ch_k(\calW_{i,l})$, for all possible $i,l,k$,
on the fundamental class $[\frakM(v,w)^\#]$.

\smallskip

\begin{theorem}\label{T:Kirwan} For any $v,$ $w$ we have 
$M_{v,w}^\#=M_*^{T\times G(w)}(\frakM(v,w)^\#)$.
\end{theorem}

When $Q$ has no oriented cycle then $Q^\#$ has no oriented cycle either, hence $\frakM(v,w)^\#$ is projective.
In this case, the theorem follows from Nakajima's resolution of the diagonal as in \cite[\S 7.3]{N01}. 
We will prove the theorem by constructing a suitable compactification of the local quiver variety. 
A very similar construction has independently appeared in McGerty and Nevins' recent proof of Kirwan surjectivity for
(non-graded, non-local) quiver varieties, see \cite{MNKirwan}. 

\smallskip

\begin{corollary} For any $v,w \in \mathbb{N}^{I\#}$, the variety $\frakM(v,w)^\#$ is connected. 
\end{corollary}

\smallskip

\subsubsection{The quivers $Q^\diamond$ and $\bar Q^\diamond$}

Consider a new quiver $Q^{\diamond}=(I^\diamond,\Omega^\diamond)$ defined as follows : 
\begin{itemize}
\item $I^\diamond=\{i_1\,;\,i\in I\}\sqcup\{i_2\,;\,i\in I\}\simeq I \times \{1,2\},$
\item $\Omega^\diamond=\{h : (h')_1 \longrightarrow (h'')_2\,;\, h \in \Omega\} \sqcup \{i : i_1 \longrightarrow i_2\,;\, i \in I\}.$
\end{itemize}

\smallskip

\noindent
\textit{Example.} We have
$$Q= \xymatrix{ i \ar@/^/[r]^-{a} & j \ar@/^/[l]^-{b}}\quad\Rightarrow\quad
Q^{\diamond}=\xymatrix{i_1\ar[r]_-{i} \ar@/^1.5pc/[rrr]^-{a} & i_2 & j_1 \ar[r]_-{j} \ar@/^/[l]^-{b} & j_2}.$$

\smallskip

The quiver $Q^\diamond$ is \emph{strongly bipartite}~: 
all arrows go from a vertex in $I \times \{1\}$ to one in $I \times \{2\}$. In particular, it contains no oriented cycles. 

\smallskip

The moduli stacks of representations of $Q$ and $Q^\diamond$ are related in the following fashion. 
For each
$v\in \bbN^I$ we define $v^\diamond \in \bbN^{{}I^\diamond}$ 
such that $v^\diamond_{i_1}=v^\diamond_{i_2}=v_i$. 
Let 
$$Rep(\k\,Q^\diamond,{}v^\diamond)^{\heartsuit}\subset Rep(\k\,Q^\diamond,{}v^\diamond)$$
be the open subset consisting of the representations $x^\diamond$ such that $x^\diamond_{i}$ 
is invertible for all $i$. Then, the assignment 
$ x^\diamond\mapsto(x\,,\,g)$ where
$$x_h= (x^\diamond_{{h''}})^{-1}\circ x^\diamond_h\,, \quad g_i=x^\diamond_{i}$$ 
gives an isomorphism 
\begin{align}\label{loc1}
Rep(\k\, Q^\diamond, v^\diamond)^\heartsuit\simeq Rep(\k Q,v) \times G(v),
\end{align}
hence an isomorphism of Artin stacks 
$$Rep(\k\,Q^\diamond,{}v^\diamond)^{\heartsuit}\,/\,G(v^\diamond)\simeq Rep(\k{}Q,v)\,/\,G(v).$$ 
This allows us to view the stack $Rep(\k Q,v)\,/\,G(v)$ as an open substack of $Rep(\k\,Q^\diamond,{}v^\diamond)\,/\,G(v^\diamond)$.

\smallskip

We will now perform a similar construction for the double quiver. Let  $\bar Q^\diamond$ be
the double quiver of $Q^\diamond$. For each $v,$ $w$ in $\bbN^I$ we define 
$v^\diamond$, $w^\diamond$ in $\bbN^{I^\diamond}$ such that 
\begin{align}\label{diamond-dim}
\begin{split}
v^\diamond_{i_1}&=v^\diamond_{i_2}=v_i\\
w^\diamond_{i_1}&=w_i,\\
w^\diamond_{i_2}&=0.
\end{split}
\end{align}
The group $G({}v^\diamond)$ acts in a Hamiltonian fashion on
$$\R({}v^\diamond,{}w^\diamond)= Rep(\k\,\bar Q^\diamond, {}v^\diamond) \oplus 
\Hom_{{}I^\diamond}({}w^\diamond,{}v^\diamond) \oplus \Hom_{{}I^\diamond}({}v^\diamond,{}w^\diamond)$$
and we denote by 
$\mu_\diamond: \R({}v^\diamond,{}w^\diamond) \to \mathfrak{g}(v^\diamond)$
the associated moment map. 
Let $\R({}v^\diamond,{}w^\diamond)^{\heartsuit}$ be the open subset of $\R({}v^\diamond,{}w^\diamond)$ consisting of pairs $(\bar x^\diamond, \bar a^\diamond)$ such that 
$x^\diamond_{i}$ is an isomorphism for all $i$. We also put 
\begin{align*}
\M(v^\diamond,w^\diamond)&=\mu_{\diamond}^{-1}(0),\\
\M(v^\diamond,w^\diamond)^{\heartsuit}&=\R({}v^\diamond,{}w^\diamond)^{\heartsuit} \cap \M(v^\diamond,w^\diamond).
\end{align*}

\smallskip

Now, let $\R(v,w)$ and $\M(v,w)$ be as in \S \ref{sec:QV}. Consider the map 
$$\phi\,:\,\R(v,w) \times G(v) \to \R(v^\diamond, w^\diamond)^{\heartsuit},\quad
((\bar{x},\bar{a})\,,\,g)\mapsto(\bar x^{\diamond}, \bar a^\diamond),$$ 
defined, for all $h\in\Omega$ and $i\in I$, by
\begin{align}\label{E:defmap}
\begin{split}
\bar a^{\diamond}_{i_1}&=\bar a_i,\\
x^{\diamond}_{h}&= g_{h''}\circ x_h,\\ 
x^{\diamond}_{h^*}&= x_{h^*} \circ (g_{h''})^{-1},\\
x^\diamond_{i}&=g_i,\\ 
x^\diamond_{i^*}&=-\sum_{h''=i} (x_{i}^\diamond)^{-1}\circ x^\diamond_h\circ x^\diamond_{h^*}.
\end{split}
\end{align}
A direct computation gives

\begin{lemma} Given $v,w$ let $v^\diamond, w^\diamond$ be as in \eqref{diamond-dim}.
Then, the map $\phi$ gives an isomorphism $$\M(v,w) \times G(v) \to \M(v^\diamond, w^\diamond)^{\heartsuit}.$$
\qed
\end{lemma}

\smallskip

We will next consider appropriate stability conditions on $\R(v^\diamond, w^\diamond)$. 
Given a tuple of integers $(\theta_i)$, let $\theta$ be the character of $G(v^\diamond)$ defined by $\theta(g)=\prod_{i \in I^\diamond}\det(g_i)^{-\theta_i}$. 
We set
\begin{align*}
\frakM_\theta(v^\diamond, w^\diamond)&=\text{Proj}\Big(\bigoplus_{n\in\bbN}\k[\M(v^\diamond, w^\diamond)]^{\,\theta^n}\Bigr),
\end{align*}
which is a quasi-projective $T\times G(w^\diamond)$-variety equipped with a projective morphism
$$\pi : \frakM_{\theta}(v^\diamond, w^\diamond) \to \frakM_0(v^\diamond, w^\diamond)=\M(v^\diamond, w^\diamond)\,/\hspace{-.03in}/\,G(v^\diamond).$$
We can describe the set of closed points of $\frakM_{\theta}(v^\diamond, w^\diamond)$ explicitly. 
Set 
$$\theta_{\infty}=- \sum_{i \in I^\diamond} \theta_i\, v^{\diamond}_i,\quad W^\diamond=\bigoplus_{i \in I^\diamond} W^\diamond_i
,\quad V^\diamond=\bigoplus_{i \in I^\diamond} V^\diamond_i.$$ 
We say that an element $(\bar x^\diamond\,,\, \bar a^\diamond)$ is \textit{$\theta$-semistable} if the following conditions are satisfied~:
\begin{itemize}
\item for any $(\bar x^\diamond, \bar a^\diamond)$-stable $I^\diamond$-graded subspace $U \subseteq V^{\diamond}$ we have
$$\sum_{i \in I^\diamond} \theta_i \dim(U_i) \leqslant 0,$$
\item for any $I^\diamond$-graded subspace $U \subseteq V^{\diamond}$ such that $U\oplus W^\diamond$ is 
$(\bar x^\diamond, \bar a^\diamond)$-stable we have
$$\sum_{i \in I^\diamond} \theta_i \dim(U_i) +\theta_{\infty} \leqslant 0.$$
\end{itemize}
We will further say that $(\bar x^\diamond, \bar a^\diamond)$ is \textit{stable} if the above inequalities are strict for proper subspaces $U$.
Let us denote by $\R_\theta(v^{\diamond},w^{\diamond})$ the open subset of $\R(v^\diamond,w^\diamond)$ consisting of $\theta$-semistable
elements and by $\M_\theta(v^{\diamond},w^{\diamond})$ its intersection with $\M(v^{\diamond},w^{\diamond})$. Then \cite[prop.~2.9]{N09} yields
$$\frakM_{\theta}(v^\diamond,w^\diamond)=\M_\theta(v^{\diamond},w^{\diamond})\,/\hspace{-.04in}/\, G(v^\diamond).$$
We also set 
\begin{align*}
\M_\theta(v^{\diamond},w^{\diamond})^{\heartsuit}&=\M_\theta(v^{\diamond},w^{\diamond})\cap \R(v^\diamond,w^\diamond)^\heartsuit,\\
\frakM_{\theta}(v^\diamond,w^\diamond)^\heartsuit&=\M_\theta(v^{\diamond},w^{\diamond})^{\heartsuit}\,/\hspace{-.04in}/ \,G(v^\diamond).
\end{align*}
We say that $\theta$ is \textit{generic} if neither equations
$$\sum_{i \in I^\diamond} \theta_i u_i=0 \ \text{and}\  \sum_{i \in I^\diamond} \theta_i u_i + \theta_{\infty}=0$$
have integer solutions $(u_i)$ satisfying $0 \leqslant u_i \leqslant v_i^{\diamond}$ other than the trivial solutions 
$$(u_i)=0\ \text{or}\ (u_i)=(v^\diamond_i).$$
If $\theta$ is generic then any semistable pair $(\bar x^\diamond,\bar a^\diamond)$ is stable and in that case the map 
$$\M_\theta(v^{\diamond},w^{\diamond}) \to \frakM_{\theta}(v^\diamond,w^\diamond)$$ is a $G(v^{\diamond})$-torsor, 
hence the variety $\frakM_{\theta}(v^\diamond,w^\diamond)$ is smooth.


\smallskip

\begin{lemma}\label{L:4.3} Given $v,w$ let $v^\diamond, w^\diamond$ be as in \eqref{diamond-dim}.
Let $\theta\in \bbZ^{I^\diamond}$ be generic and such that
\begin{align}\label{theta}\theta_{i_1} \gg 0, \qquad \theta_{i_2} \ll 0, \qquad \theta_{i_1} + \theta_{i_2}=1, \quad \forall i \in I.
\end{align}
Then, the map $\phi$ restricts to an isomorphism  
$$\M_\theta(v,w) \times G(v) \stackrel{\sim}{\longrightarrow} \M_\theta(v^{\diamond},w^{\diamond})^{\heartsuit}.$$
\end{lemma}

\begin{proof} Fix a tuple $z=(\bar x^\diamond, \bar a^\diamond)$ in $\M(v^{\diamond},w^{\diamond})^{\heartsuit}$. 
We will prove that 
$$z\in \M_\theta(v^{\diamond},w^{\diamond})^{\heartsuit}\iff\phi^{-1}(z) \in \M_\theta(v,w) \times G(v).$$
Let $U$ be an $I^\diamond$-graded subspace of $V^\diamond$.
\smallskip
Let us first assume that $U$ is $z$-stable and non-zero. Since $x^\diamond_{i}$ is invertible, we have
$$\dim(U_{i_1}) \leqslant \dim(U_{i_2}),\quad i \in I.$$ If there exists $j\in I$ such that $\dim(U_{j_1}) < \dim(U_{j_2}),$ then
\begin{equation*}
\begin{split}
\sum_{i \in I^\diamond} \theta_i \dim(U_i) &= \sum_{i \in I} \big[ \theta_{i_2}(\dim(U_{i_2})-\dim(U_{i_1})) + \dim(U_{i_1})\big]\\
&\leqslant \theta_{j_2} + \sum_{i \in I} v_i\\
&<0
\end{split}
\end{equation*}
because $\theta_{j_2} \ll 0$. Thus $U$ is not destabilizing. On the other hand, if $x_{i}^\diamond$ is an isomorphism
$U_{i_1} \to U_{i_2}$ for all $i$, then we have
$$\sum_{i \in I^\diamond} \theta_i \dim(U_i)=\sum_i \dim(U_{i_1}) >0$$
hence $U$ is destabilizing.

\smallskip

Assume now that $U \oplus W^\diamond$ is $z$-stable and proper. 
As above, we have $\dim(U_{i_1}) \leqslant \dim(U_{i_2})$ for all $i$. Thus
\begin{equation*}
\begin{split}
\sum_{i \in I^\diamond} \theta_i \dim(U_i) + \theta_{\infty} &=\sum_{i \in I} \big[ \theta_{i_2}(\dim(U_{i_2})-\dim(U_{i_1})) + \dim(U_{i_1})\big] +\theta_{\infty}\\
&\leqslant \sum_{i \in I} \dim(U_{i_1})-v_i\\
&<0
\end{split}
\end{equation*}
for otherwise $\dim(U_{i_1})=\dim(U_{i_2})=v_i$ for all $i$ and $U$ is not a proper subspace of
$V^\diamond$. Therefore $U$ is not destabilizing.

\smallskip

By the above, we conclude that $z$ is $\theta$-semistable if and only if there does not exist a nonzero $z$-stable 
$U \subset V^\diamond$ stable under all the maps
$x_{i}^{-1}$. This is easily seen to be equivalent to the condition that $\phi^{-1}(z)$ belongs to $\M_s(v,w) \times T\times G(v)$.
\end{proof}

\smallskip

\begin{corollary}\label{cor:psi} Given $v,w$ let $v^\diamond, w^\diamond$ be as in \eqref{diamond-dim}.
Let $\theta$ be a generic character as in \eqref{theta}.
Then, we have an isomorphism $\psi\,:\,\frakM(v,w)\, {\to}\, \frakM_\theta(v^\diamond,w^\diamond)^\heartsuit$ and an open immersion 
\begin{align*} i~: \frakM(v,w) \to \frakM_\theta(v^\diamond,w^\diamond).\end{align*}
\qed
\end{corollary}

\smallskip

\subsubsection{The compactification of local quiver varieties}

From now on we fix $u,y,u^\diamond, y^\diamond$ and $\theta$ as above. The existence of such a character $\theta$ is clear. 
For simplicity, we abbreviate
$$\frakM=\frakM(u,y), \qquad \frakM^\diamond=\frakM_{\theta}(u^\diamond,y^\diamond), \quad 
\frakM^{\diamond,\heartsuit}=\frakM_\theta(u^\diamond,y^\diamond)^\heartsuit.$$
It is straightforward to check from (\ref{E:defmap}) that the $\bullet$-actions on
$\frakM$ and $\frakM^\diamond$ given by
\begin{align*}
t \bullet (\bar{x},\bar{a})&=(x,tx^*,a,ta^*), \\ t \bullet (\bar x^\diamond, \bar a^\diamond)&=(x^\diamond, tx^{*,\diamond},a^\diamond,ta^{*,\diamond}),
\end{align*}
are compatible with the open embedding $i$ in Corollary \ref{cor:psi}.
Applying \eqref{decomposition1} to the quivers $Q$ and $Q^\diamond$,
we get the decompositions into smooth disjoint subvarieties
\begin{align*}
\frakM^{\bullet} = \bigsqcup_{v,w} \frakM(v,w)^\#,\quad
(\frakM^{\diamond})^{\bullet}=\bigsqcup_{v^\diamond,w^\diamond}\frakM(v^\diamond,w^\diamond)^\#
\end{align*}
where $v,$ $w$, $v^\diamond$, $w^\diamond$ run  over the sets of dimension vectors in $\bbN^{I^{\#}}$ and $\bbN^{I^{\diamond,\#}}$ 
such that 
$$\sum_l v_{i,l}=u_i,\quad
\sum_L w_{i,l}=y_i,\quad
u^{\diamond}_i=\sum_l v^\diamond_{i,l},\quad
y^{\diamond}_i=\sum_l w^\diamond_{i,l}.$$
The variety $\frakM(v^\diamond,w^\diamond)^\#$ is realized in the same way as $\frakM(v,w)^\#$ by
replacing the quiver $Q$ by the quiver $Q^\diamond$ throughout. 
Taking the fixed points under the $\bullet$-actions, the map $i$ yields an open immersion
\begin{align}\label{psibullet}i~: \frakM(v,w)^\# \to \frakM(v^\diamond,w^\diamond)^{\diamond,\#}\end{align}
where $v^\diamond,w^\diamond$ are defined by 
$$v^{\diamond}_{i_1,l}=v^{\diamond}_{i_2,l}=v_{i,l}\,,\quad w^\diamond_{i_1,l}=w_{i,l}\,,\quad w^\diamond_{i_2,l}=0.$$
Since the quiver $Q^\diamond$ has no oriented cycle, each $\frakM^{\diamond,\#}$ is projective and we have obtained in this way the 
desired compactification of $\frakM^\#$.

\smallskip

\subsubsection{Proof of Theorem~\ref{T:Kirwan}}

Now, let us consider the resolution of the diagonal. 
The smooth variety $\frakM^\diamond \times \frakM^\diamond$ carries $T\times G(w^\diamond)\times \mathbb{G}_m$-equivariant universal and
tautological bundles given, for all $i\in I^\diamond$, by
$$\calV_i^{\diamond, 1}=\calV_i^\diamond \boxtimes \mathcal{O}_{\frakM^\diamond},\quad  
\calV_i^{\diamond, 2}=\mathcal{O}_{\frakM^\diamond} \boxtimes \calV_i^\diamond,\quad
\calW_i^{\diamond,1},\quad \calW_i^{\diamond, 2}.$$
Note that $\calW_i^\diamond=\{0\}$ for $i \in I_2$ and that $\calW_i^{\diamond,1}=\calW_i^{\diamond,2}$ 
is trivial as a vector bundle. We will simply denote this bundle by $\calW_i^{\diamond}$.
We set also $\calV^{\diamond}=\bigoplus_{i \in I^\diamond}\calV_i^{\diamond}$ and
$\calW^{\diamond}=\bigoplus_{i \in I^\diamond}\calW_i^{\diamond}$.
Consider the complex of equivariant bundles on $\frakM^\diamond \times \frakM^\diamond$ given by
\begin{equation}\label{Nakcomplex}
\xymatrix{
\Hom_{I^\diamond}(\calV^{\diamond,1}, \calV^{\diamond,2}) \ar[r]^-{\mathbf{a}} & C \ar[r]^-{\mathbf{b}} &
\Hom_{I^\diamond}(\calV^{\diamond,1},\calV^{\diamond,2})[1] }
\end{equation}
where
\begin{itemize}
\item
$C= \bigoplus_{h \in \bar{\Omega}^\diamond}\Hom(\calV_{h'}^{\diamond,1}, \calV_{h''}^{\diamond,2})[\varepsilon_h] \; \oplus \; 
\Hom_{I^\diamond}(\calV^{\diamond,1},\calW^\diamond)[1]\; \oplus\; 
\Hom_{I^\diamond}(\calW^\diamond,\calV^{\diamond,2}),$
 \item
$\varepsilon_h=1$ if $h \in \Omega^{\diamond,*}$ and $0$ otherwise, 
\item $[1]$ denotes a grading shift with respect to the 
$\mathbb{G}_m$-action,
\item the maps $\mathbf{a}, \mathbf{b}$ are defined as
$$\mathbf{a}((\xi_i))=(x^{\diamond,2}\,\xi_{h'} -\xi_{h''}\,x^{\diamond,1}_h) \oplus (a_i^{*,\diamond, 2}\,\xi_i)\oplus (-\xi_i\,a^{\diamond,1}_i)$$
$$\mathbf{b}(y_h, \gamma_i, \delta_i)=\sum_{h \in \Omega^\diamond} (x^{\diamond,2}_{h}\,y_{h^*} - y_{h^*}\,x^{\diamond,1}_h)+
\sum_i (a^{\diamond,2}_i\,\delta_i +\gamma_i\,a^{*,\diamond,1}).
$$
\end{itemize}

\smallskip

\begin{lemma}\label{L:Nakcompl} Over the open subset $\frakM^{\diamond,\heartsuit} \times \frakM^\diamond$ the map $\mathbf{a}$ is injective and the map 
$\mathbf{b}$ is surjective.
\end{lemma}

\begin{proof} 
Fix tuples $z_1 \in \frakM^{\diamond,\heartsuit}$ and $z_2 \in \frakM^{\diamond}$. 
Set $z_1=(\bar x_1^\diamond,\bar a_1^\diamond)$ and $z_2=(\bar x_2^\diamond,\bar a_2^\diamond)$. 
We will first prove that the map
$\mathbf{a}$ is injective over the point $(z_1,z_2)$. Let $(\xi_i) \in \Ker(\mathbf{a})$ and put
$$K_i=\ker(\xi_i) \subseteq V_i^{\diamond,1}, \quad J_i=\Im(\xi_i) \subseteq V_i^{\diamond,2}, \quad 
K=\bigoplus_i K_i \oplus W^\diamond, \quad J=\bigoplus_i J_i.$$
Since $(\xi_i) \in \Ker(\mathbf{a})$,
the space $K$ is a subrepresentation of $z_1$ while $I$ is a subrepresentation of $z_2$. Since $z_1 \in \frakM^{\diamond,\heartsuit}$ we 
have $\dim(K_{i_1}) \leqslant \dim(K_{i_2})$ and hence $\dim(J_{i_1}) \geqslant \dim(J_{i_2})$ for all $i \in I^\diamond$. 
We deduce that
$$\sum_{i \in I^\diamond} \theta_i \dim(J_i)=\sum_{i \in I} \big[ \theta_{i_1}(\dim(J_{i_1})-\dim(J_{i_2})) + \dim(J_{i_2})\big] \geqslant 0$$
with a strict inequality if there exists $j$ such that $J_{j_2} \neq \{0\}$ or $\dim(J_{j_1}) > \dim(J_{j_2})$, i.e., if $J\neq \{0\}$. 
By $\theta$-semistability of $z_2$ we conclude that $J=\{0\}$ 
and that $\mathbf{a}$ is injective.

\smallskip

To prove that $\mathbf{b}$ is surjective, we dualize the previous argument. Using the trace pairing, we may view the transpose
$\mathbf{b}^*$ of $\mathbf{b}$ as a map
$$\mathbf{b}^*~:\Hom_{I^\diamond}(\calV^{\diamond,2}, \calV^{\diamond,1})[-1] \to C^*$$
where
 $$C^*= \bigoplus_{h \in \bar{\Omega}^\diamond}\Hom(\calV_{h''}^{\diamond,2}, \calV_{h'1}^{\diamond,1})[-\varepsilon_h] \; \oplus \; 
 \bigoplus_{i \in I^\diamond} \Hom(\calV_i^{\diamond,2},\calW_i^\diamond)[-1]\; \oplus\; 
 \bigoplus_{i \in I^\diamond} \Hom(\calW_i^\diamond,\calV_i^{\diamond,1}).$$
Let $(\xi_i) \in \Ker(\mathbf{b}^*)$ and set 
$$K_i=\ker(\xi_i) \subseteq V_i^{\diamond,2}, \quad J_i=\Im(\xi_i) \subseteq V_i^{\diamond,1}, \quad K=\bigoplus_i K_i \oplus W^\diamond, 
\quad J=\bigoplus_i J_i.$$
Because $(\xi_i) \in \Ker(\mathbf{b}^*)$, the spaces $J$ and $K$ are respectively subrepresentations of $z_1$ and $z_2$. Since $z_1 \in \frakM^{\diamond,\heartsuit}$ it follows that $\dim(J_{i_1}) \leqslant \dim(J_{i_2}),$ hence $\dim(K_{i_1}) \geqslant \dim(K_{i_2})$ for all $i$. 
If $\dim(K_{j_1}) > \dim(K_{j_2})$ for some $j$, then we have
\begin{align*}
\sum_{i \in I^\diamond} \theta_i \dim(K_i) +\theta_{\infty} 
&= \sum_{i \in I^\diamond} \big[ \theta_{i_1}(\dim(K_{i_1})-\dim(K_{i_2})) + \dim(K_{i_2})\big] + \theta_{\infty} \\
&\geqslant \theta_{j_1} -\sum_i v_i ,\\
&>0,
\end{align*}
which would contradict the semistablity of $z_2$. Thus $\dim(J_{i_1})=\dim(J_{i_2})$ for all $i$. But then
$$\sum_{i \in I^\diamond} \theta_i \dim(J_i)=\sum_{i \in I} \dim(J_{i_1}) \geqslant 0$$
with strict inequality as soon as $J_{i} \neq \{0\}$ for some $i$. By the semistability of $z_1$, this forces
$J_i=\{0\}$ for all $i$, i.e., we have $(\xi_i)=0$ as wanted. Lemma~\ref{L:Nakcompl} is proved.
\end{proof}

\smallskip

By Lemma~\ref{L:Nakcompl} the restriction of the complex (\ref{Nakcomplex}) to $\frakM^{\diamond,\heartsuit} \times \frakM^\diamond$ is quasi-isomorphic to an 
equivariant vector bundle 
$\mathcal{E}=\Ker(\mathbf{b})/\Im(\mathbf{a})$. 
Moreover, the bundle $\mathcal{E}$ carries a section which vanishes precisely on the diagonal 
$\Delta \frakM^{\diamond,\heartsuit}$ in 
$$\frakM^{\diamond,\heartsuit} \times \frakM^{\diamond,\heartsuit} \subset \frakM^{\diamond,\heartsuit} \times \frakM^{\diamond}.$$
Observe that $\Delta \frakM^{\diamond,\heartsuit}$ is a closed subset of $\frakM^{\diamond,\heartsuit} \times \frakM^\diamond$. 
Indeed, see \cite{N94}, the map 
$$s=\big(0 \oplus -a^{\diamond,2} \oplus  a^{*,\diamond,1}\big)$$ defines
a section of $\mathcal{E}$ which vanishes precisely over the set of pairs $(z_1,z_2)$ such that $\Hom(z_1,z_2)\neq \{0\}$ as $\k\bar Q^\diamond$-modules. 
By the semistability condition, both $z_1$ and $z_2$ are simple so that $\Hom(z_1, z_2) \neq \{0\}$ if and only if $z_1 =z_2$.
Thus the zero set of $s$ is  $\Delta \frakM^{\diamond,\heartsuit}.$ 

\smallskip

As a consequence, we have the following equality in
$\A^{T\times G(w^\diamond) \times \mathbb{G}_m}_*(\frakM^{\diamond,\heartsuit} \times \frakM^\diamond)$~:
$$[{\Delta \frakM^{\diamond,\heartsuit}}]=
\sum_{i \geqslant 0} (-1)^i\, \ch(\Lambda^i \mathcal{E})\cap[\frakM^{\diamond,\heartsuit} \times \frakM^\diamond].$$

\smallskip

Let us now fix dimension vectors $v,$ $w$ for $Q^\#$ and denote by $v^\diamond,$ $w^\diamond$ 
the associated dimension vectors for $Q^{\diamond}$.
We abbreviate $\frakM^\#=\frakM(v,w)^\#$ and $\frakM^{\diamond,\#}=\frakM(v^\diamond,w^\diamond)^\#$
as before.
Let $\mathcal{E}^\#$ be the restriction of $\mathcal{E}$ to $\frakM^\# \times \frakM^{\diamond,\#}$ under the map \eqref{psibullet}.
We have
$$[{\Delta \frakM^\#}]=\sum_{i \geqslant 0} (-1)^i\,\ch(\Lambda^i \mathcal{E}^\#)\cap[\frakM^\# \times \frakM^{\diamond,\#}].$$

\smallskip

Consider the following diagram
$$\xymatrix{ & \frakM^\# \times \frakM^{\diamond,\#} \ar[r]^-{p_1} & \frakM^\#\\
\frakM^\# \ar[r]^-{\delta} \ar[ur]^-{\delta'} & \frakM^\# \times \frakM^\# \ar[u]_-{1 \times i} &}$$
in which $\delta~: \frakM^\# \to \frakM^\# \times \frakM^\#$ is the diagonal embeddings and $p_1$ is the projection onto the first factor. 
The maps $\delta,$ $\delta'$ and $p_1$ are proper. 
Let $\alpha \in \A^{T\times G(w)}_*(\frakM^\#)$ and choose
$\alpha^\diamond \in \A^{T\times G(w^\diamond)}_*(\frakM^{\diamond,\#})$ 
such that $i^*(\alpha^\diamond)=\alpha$. 

\smallskip

We have, on the one hand
\begin{equation}\label{E:ktheory}
\begin{split}
p_{1*}\left(\big([{\frakM^\#}] \boxtimes \alpha^\diamond\big) \cap [{\Delta\frakM^\#}]\right)
&=p_{1*}\left(\big([{\frakM^\#}] \boxtimes \alpha^\diamond\big) \cap\delta'_{*} [{\frakM^\#}]\right)\\
&=p_{1*}\delta'_*\,{\delta'}^{*}\big([{\frakM^\#}] \boxtimes \alpha^\diamond\big)\\
&={\delta'}^{*}\big([{\frakM^\#}] \boxtimes \alpha^\diamond\big)\\
&=i^*(\alpha^\diamond)\\
&=\alpha.
\end{split}
\end{equation}
On the other hand, using the Giambelli formula we may write 
$$\sum_{i \geqslant 0} (-1)^i\,\ch(\Lambda^i \mathcal{E}^\#)=\sum_k a_k \boxtimes b_k$$ for some classes $a_k$ belonging to the subring
of $\A^*_{T\times G(w)}(\frakM^\#)$ generated by the classes of the tautological bundles 
$\calV_{i,l}$ and $\calW_{i,l}$ as $(i,l)$ runs over $I^\#$, and some classes $b_k$ belonging to the subring  
of $\A^*_{T\times G(w^\diamond)}(\frakM^{\diamond,\#})$ generated by the classes of the tautological bundles 
$\calV^\diamond_{i,l}$ and $\calW^\diamond_{i,l}$ as $(l\delta_i)$ runs over $I^{\diamond,\#}$. 
We have isomorphisms $i^*(\calV^{\diamond}_{i_k,l})\simeq\calV_{i,l}$ and $\calW^\diamond_{i,l}\simeq\calW_{i,l}$.
Further, we have
\begin{equation}\label{E:ktheory2}
\begin{split}
p_{1*}\left( \big( [{\frakM^\#}] \boxtimes \alpha^\diamond\big) \cap [{\Delta \frakM^\#}]\right)
&=\sum_k p_{1*}\left( \big( [{\frakM^\#}] \boxtimes \alpha^\diamond\big) \cap \big( a_k \boxtimes b_k\big)\right)\\
&=\sum_k a_k \cap p_*(\alpha^\diamond \cap b_k)
\end{split}
\end{equation}
where $p: \frakM^{\diamond,\#} \to \bullet$ is the map to a point. 

\smallskip

Combining (\ref{E:ktheory}) and (\ref{E:ktheory2}), 
we deduce that 
$$\A_{v,w}^\#=\A_*^{T\times G(w)}(\frakM^\#).$$
The statement for the equivariant homology groups 
can be deduced from that for the Chow homology groups. 
Indeed, by the same argument as in Proposition~\ref{prop:L} (c), 
since there is a contracting $\mathbb{G}_m$-action on $\frakM^\#$ whose fixed point 
subvariety is a projective graded quiver variety, the cycle map 
$$\cl\,:\,\A_*^{T\times G(w)}(\frakM^\#) \to H_{2*}^{T\times G(w)}(\frakM^\#)$$ is surjective. 
This map preserves the tautological subrings.
Hence, we have 
$$H_{v,w}^\#=H_*^{T\times G(w)}(\frakM(v,w)^\#).$$
Theorem~\ref{T:Kirwan} is proved.

\medskip

\subsection{Non-injectivity of the shuffle realization map for small tori} \label{sec:non-injectivity}
In this section we provide an explicit instance in which the map to the shuffle algebra
$$i_* : H_*^{T \times G(v)}(\Lambda^0(v)) \to H_*^{T \times G(v)}(\R(v))$$ 
 is not injective when the assumption $\theta, \theta^*\subset T$ is not verified. 
 We will consider the case where $Q$ is the Jordan quiver and $v=2$.
For any torus $T$ we have, by Theorem~A, 
$$\dim H^{T \times G(2)}_{8-2i}(\Lambda(2))=\begin{cases} 2 & \text{if\;} i=0 \\ 2 + 2\;dim(T) & \text{if\;} i=1.\end{cases}$$
Note that the dimension of $\Lambda(2)$ is $4$. In addition, by Theorem~\ref{thm:gen}, 
$$H_*^{T \times G(2)}(\Lambda(2))=H^*_{T \times G(2)} [\Lambda_2] + (H^*_{T \times G(1)} [\Lambda_1]) \star (H^*_{T \times G(1)} [\Lambda_1]).$$
 
We will consider the image of the map $i_*$ and compare its graded dimension with that of 
$H^{T \times G(2)}_{\star}(\Lambda(2))$ for various choices of $T$. For this, we first compute the image of $i_*$ 
for the torus $T=\theta \times \theta^*$ and then specialize. 
Let us write $H^*_T=\mathbb{Q}[t,t_*]$ and $H^*_{G(2)}=\mathbb{Q}[z_1+z_2, z_1z_2]$. 
Since $\Lambda_{(2)}$ is the zero section of the the map $\R(2) \to Rep(\k Q^*,2)$ we have 
$$i_*([\Lambda_{(2)}])=(z_1+t)(z_2+t)=z_1z_2 + (z_1 + z_2)t + t^2 =:x_0$$
while a direct computation using the shuffle multiplication, see \eqref{shuffle}, yields
$$i_*([\Lambda_1] \star [\Lambda_1])=2(z_1-z_2)^2-2tt_* -2 (t^2+t_*^2)=:x_0'$$
$$i_*(z[\Lambda_1] \star [\Lambda_1])=(z_1+z_2)(z_1-z_2)^2-tt_*(z_1+z_2)-(t^2+t_*^2)(z_1+z_2)-tt_*(t+t_*)=:x_1$$
$$i_*(z[\Lambda_1] \star [\Lambda_1])=(z_1+z_2)(z_1-z_2)^2-tt_*(z_1+z_2)-(t^2+t_*^2)(z_1+z_2)+tt_*(t+t_*)=:x'_1
$$
We describe the image of $i_*$ in small cohomological degree. We have
$$\Im(i_*) \cap H_8^{T \times G(2)}(\R(2)) = Span_{\mathbb{Q}}(x_0,x_0')$$
$$\Im(i_*) \cap H_6^{T \times G(2)}(\R(2)) = Span_{\mathbb{Q}}(tx_0, t_*x_0, tx_0', t_*x_0',x_1,x_1')$$
 Let $T' \subset T$ be a one dimensional subtorus. Let us denote by $t,$ $t_*$ the images of $t,$ $t_*$ under the restriction map $H^*_{T \times G(2)} \to H^*_{T' \times G(2)}$. Then it is easy to check that the span of the six elements $\{tx_0, t_*x_0, tx_0', t_*x_0',x_1,x_1\}$ is of dimension $4$ over $\mathbb{Q}$ if and only if $tt_*(t+t_*) \neq 0$ and of dimension $3$ otherwise (observe that $x'_1-x_1 = 2 tt_*(t+t_*)$). In particular, the map $i_*$ is not injective when $T=\theta$ or $T=\theta^*$, or $T=\theta-\theta^*$.

\newpage

\centerline{\bf{Index of Notations}}

\bigskip

\begin{enumerate}
\item[\textbf{2.1.}] $H_*$, $\A_*$, $M_*$, $\varepsilon=1$ or $2$, $H_*^G$, $\A_*^G$, $M_*^G$,\\

\item[\textbf{2.2.}] $f^{\heartsuit}$,\\

\item[\textbf{2.3.}] $(f,f')^! $, $\cl:\A_*^G(X)\to H^G_{2*}(X)$, $\ch=\sum_{l\geqslant 0}\ch_l$,\\

\item[\textbf{3.1.}]
$Q, $ $Q^*$, $\bar{Q}$, $\Omega$, $\Omega^*$, $\bar{\Omega}$, $\varepsilon(h)$, $h', $ $h''$, $q_{ij},$ $ \bar{q}_{ij}, $ $q_i, $
$\Omega_{ij},$ $ \bar\Omega_{ij}$, $\langle\bullet\,,\,\bullet\rangle,$ 
$ (\bullet\,,\,\bullet)$, $d_v$, $\delta_i$, $\k Q$, $\bar{x}=(x,x^*)$, $\R(v)$,\\

\item[\textbf{3.2.}] $G(v)$, 
$\mathfrak{g}(v)$, 
$\mu: \R(v) \to \mathfrak{g}(v)$, $\M(v),$ $ \Pi$, $\langle \bullet\,,\,\bullet\rangle_{\Pi}$,\\

\item[\textbf{3.3.}] $\theta(z), $ $\theta^*(z),$ $ T=T_\sp\times T_\dil$,\\

\item[\textbf{3.4.}] $\Lambda_W, $ $\Lambda_{\nu}$, $\Lambda^\flat(v)$, $\flat=0$ or $1$,\\

\item[\textbf{3.5.}] $\Lambda^\circ_W,$ $ \Lambda^\circ_\nu$ ,${}^\circ\!\Lambda_{F},$ ${}^\circ\!\Lambda_{\nu}$, 
$\Lambda_{(v)}$,\\

\item[\textbf{3.6.}] $\J(i,q)$, $\mu\,\lescc\,\nu$,\\

\item[\textbf{3.7.}] $\R(v,w), $ $\Hom_I(V,W)$, $G(w)$, $\bar{x},$ $ \bar{a}$, $\M(v,w)$, $\frakM_0(v,w), $ $\frakM(v,w)$, 
$\pi_{\theta}: \frakM_\theta(v,w) \to \frakM_0(v,w)$,  $ \rho_{\theta}$, $R_{\theta}(v,w),$ $ \M_{\theta}(v,w), $ $\M_s(v,w)$, 
$d_{v,w}$,
$\frakM(v,w)_{\mathbb{A}^1}, $ $\frakM_0(v,w)_{\mathbb{A}^1},$
$\tilde{Q}=(\tilde{I}, $ $\tilde{\Omega}),$
$\tilde{v},$
$\tau$, $RT(v,w)$, $\frakM_0(\tau)$, $\kappa_{v,w}$, $\M(\tau),$ $\M_s(\tau)$, $ \frakM(\tau)$,
$\frakL(v,w)$, $\M[\rho],$ $\L[\rho],$ $ \frakM[\rho],$ $ \frakL[\rho]$,\\

\item[\textbf{3.8.}] $\H[v_1,v_2\,;\,w], $ $\mathfrak{h}[v_1,v_2\,;\,w],$
$\H[v_1,S\,;\,w], $ $\mathfrak{h}[v_1,S\,;\,w],$
$\H[v_1,S], $ $\mathfrak{h}[v_1,S]$,\\

\item[\textbf{4.1.}] $L^\flat_s(v,w),$ $ \frakL^\flat(v,w)$,
$\circ : \mathbb{G}_m \times \frakM(v,w) \to \frakM(v,w)$, $\bullet : \mathbb{G}_m \times \frakM(v,w) \to \frakM(v,w)$,\\

\item[\textbf{4.2.}] $H(w) \subset G(w),$\\

\item[\textbf{4.4.}] $\Bbbk$, $\Bbbk[v]$,\\

\item[\textbf{5.}]  $H(v)\subset G(v)$,  $\Bbbk\langle v \rangle$, $K,$\\

\item[\textbf{5.1.}] $\bbX,$ $ \bbX_+$,
$\Y^\flat,$ $\Y^\flat_A,$ $\Y^\flat_M,$ $\Y^\flat(v),$ $\Y^\flat(v,k),$\\

\item[\textbf{5.2.}] $\Y$\\

\item[\textbf{5.3.}] $\lambda^\flat(q,z)$, $\tau=\rank(T)$,
$A^\flat_v(t),$\\

\item[\textbf{5.4.}] $\varepsilon_i(\bar x)$, $\R(v)^{l,i},$ $ \Lambda^1(v)^{l,i}$,
$\Y^1(v)^{\geqslant l,i},$
$x_{i,l},$ $ x_l$,\\

\item[\textbf{5.5.}] $\bfY^{\flat}$, $\Bbbk(\infty)$, $\bfY(0)$, $\bfY_K$,
$\mathcal{U}_i$,
$I^\re,$ $ I^\el,$ $ I^\hyp$,\\

\item[\textbf{5.6.}]  $F_w$, $F_w(v)$, $F_w(v,k)$, $A_w$, $A_w(v)$, $A_w(v,k)$,
$\frakC(v_1,v_2\,;\,w)$, $\frakC_{i,l}(w)$, $\frakC_{i,l}$, $\frakM(w)$,
$\calV_i, $ $\psi_{i,l}$.

\end{enumerate}

\newpage

\end{document}